\documentclass[dvipsnames]{article}
\usepackage[a4paper, total={7in, 9.5in}]{geometry}
\usepackage{cite}
\usepackage{float}
\usepackage{pdflscape}
\usepackage{mathtools}  
\usepackage{amssymb,amsfonts} 
\usepackage{authblk}
\usepackage{fullpage}
\usepackage{graphicx}
\usepackage{lscape}
\usepackage{enumerate}
\usepackage{hyperref}
\usepackage[utf8]{inputenc} 
\usepackage{booktabs} 
\usepackage[T1]{fontenc}    
\usepackage{url}            
\usepackage{multirow}
\usepackage{amsthm}
\usepackage{amsmath}
\usepackage{siunitx}
\usepackage{subfigure}
\usepackage[textsize=tiny]{todonotes}

\usepackage{xcolor}

\newcommand{\bc}{\textbf{C}}

\newcommand{\p}{\partial}
\newcommand{\bb}{\boldsymbol}
\newcommand{\bu}{\boldsymbol{v}}

\newtheorem{remark}{Remark}[section]
\newtheorem{thm}{Theorem}[section]

\newtheorem{lemma}[thm]{Lemma}
\newtheorem{note}[thm]{Note}
\newtheorem{definition}{Definition}[section]
\newcommand{\jump}[1]{[\![#1]\!]}
\newcommand{\df}[2]{\frac{\partial #1}{\partial #2}} 

\newcommand{\pe}{p_e}
\newcommand{\pii}{p_i}

\newcommand{\con}{\mathbf{U}}

\newcommand{\f}{\textbf{f}}
\newcommand{\F}{\boldsymbol{f}}

\newcommand{\ent}{\mathcal{E}}
\newcommand{\entf}{\mathbf{q}}
\newcommand{\evar}{\boldsymbol{\mathcal{V}}}

\newcommand{\iph}{i+\frac{1}{2}}

\newcommand{\jph}{j+\frac{1}{2}}

\newcommand{\es}{\textbf{es}}
\newcommand{\explicit}{\textbf{exp}}

\newcommand{\ote}{\textbf{O2}^\es_\explicit}
\newcommand{\othe}{\textbf{O3}^\es_\explicit}
\newcommand{\ofe}{\textbf{O4}^\es_\explicit}

\usepackage{mdframed}
\usepackage[font=small]{caption}

\title{Entropy stable finite difference schemes for One-Fluid Two-Temperature Euler Non-equilibrium Hydrodynamics}
\author[a,*]{Chetan Singh}
\author[a,b]{Harish Kumar}
\affil[a]{Department of Mathematics, Indian Institute of Technology Delhi, India}
\affil[b]{Indian Institute of Technology Delhi (IITD) - Abu Dhabi, Abu Dhabi}
\affil[*]{Corresponding Author}
\date{}
\begin{document}
	\maketitle
\begin{abstract}
	One-Fluid Two-Temperature Euler (OFTT-Euler) equations are used for modeling non-equilibrium hydrodynamics and form a system of nonlinear hyperbolic partial differential equations with non-conservative
products. The model decomposed the total pressure into two scalar components: one for electrons and one for ions. In this work, we design entropy-stable finite difference numerical schemes for the model. This is achieved by introducing a novel reformulation of the equations, ensuring that the new non-conservative terms do not contribute to the entropy evolution. For this novel reformulation of the equations, we design higher-order entropy-conservative numerical schemes by using Tadmor’s relation for the conservative part and higher-order central differences for the non-conservative parts. Finally, we design the entropy-dissipation terms using the entropy-scaled right eigenvectors of the conservative
part, thereby ensuring entropy stability for the entire system at the semi-discrete level. We present several test cases in one and two dimensions to demonstrate the accuracy and entropy stability of the proposed schemes.
\end{abstract}

\section{Introduction}
Radiation hydrodynamics plays a crucial role in astrophysics~\cite{castor2007astrophysical,springel2010smoothed,teyssier2015grid,winkler2012astrophysical}, inertial confinement fusion (ICF)~\cite{atzeni2004physics,drake2006radiation,mihalas2013foundations,sanz2009radiation}, and high energy density physics fields~\cite{drake2006introduction,drake2018journey,davidson2004frontiers,matzen2005pulsed}. The phenomenon models radiative, high-energy-density plasmas containing ions and electrons, with each species in its own thermal equilibrium. This results in each species (ions and electrons) having different temperatures, which are typically not equal; hence, they need to be evolved separately.

Assuming quasi-neutrality and ignoring radiative effects in the modeling yields a fluid description with one density and velocity but two distinct temperatures. The resulting model is known as the One-Fluid Two-Temperature Euler (OFTT-Euler) system. The model was proposed in~\cite{cheng2024high}, and ignores viscosity, thermal relaxation, and thermal conduction. 

The OFTT model is a system of hyperbolic partial differential equations with non-conservative products. Hence, developing stable numerical methods is challenging. The key difficulty is choosing a specific path to define weak admissible solutions, which is usually unknown~\cite{dal1995definition}. Also, the numerical solutions are affected by the numerical viscosity~\cite{abgrall2010comment}. Usually,  a linear path is considered to design approximate Riemann solvers~\cite{cheng2024high,cheng2024highweno,cheng2024lagrangian,singh2024eigen,balsara2025physical}.
Numerical methods for the OFTT-Euler models have been developed in several articles~\cite{despres2001lagrangian,breil2011multi,moiseev2015solution,shiroto2018structure,aregba2018modelling,chalons2005riemann,lin2006discontinuous,cockburn1989tvb,cockburn1998runge,sangam2021derivation}. Also, for the related model, in ~\cite{cheng2024lagrangian,cheng2024highweno}, authors have presented numerical schemes for the three-temperature radiation hydrodynamics equations, which govern the evolution of the interaction between radiation and plasma flow. In~\cite{cheng2025mathematical}, the authors presented simulations of two-material, two-temperature compressible flows. Additionally, in~\cite{zhao2025fifth}, the authors have developed a fifth-order equilibrium-preserving, path-conservative, characteristic-wise AWENO scheme for the OFTT-Euler model that incorporates electron–ion energy-exchange source terms. For the OFTT-Euler equations,  in~\cite{cheng2024high}, a higher-order discontinuous Galerkin method is proposed. More recently, authors have proposed positivity-preserving higher-order discontinuous Galerkin methods for two-temperature compressible flows in~\cite{cheng2025analysis}.

As the OFTT-Euler system is hyperbolic, entropy stability is one of the few theoretical nonlinear stability estimates. Hence, it is desirable to design numerical methods that are consistent with the entropy condition associated with the system. Several authors have developed entropy-stable finite-difference schemes for hyperbolic systems in conservative form~\cite{bhoriya2023high,kumar_entropy_2012,sen_entropy_2018,chandrashekar2016entropy,chandrashekar2013kinetic,fjordholm2012arbitrarily,tadmor2003entropy,derigs2018ideal}. More recently, these methods have been extended to several non-conservative systems~\cite{yadav2023entropy,singh2024entropy,singh2026entropy,rueda2024entropy}. In this paper, we propose entropy-stable finite difference schemes for the OFTT-Euler system. This is achieved as follows:
\begin{enumerate}
	\item First, we present the entropy stability of the OFTT-Euler system at the continuous level. Then, following~\cite{singh2024entropy,yadav2023entropy}, we propose a novel reformulation of the OFTT-Euler system, such that the new non-conservative terms do not contribute to the entropy evolution. 
	\item For the new conservative part, we first derive the second-order entropy-conservative numerical flux, which is then used to design higher-order entropy-conservative numerical fluxes. Combining this with a suitable higher-order central difference-based discretization of the derivatives in non-conservative terms, we ensure the entropy conservation for the whole system.
	\item To achieve entropy stability, we then design a higher-order entropy diffusion operator using entropy-scaled eigenvectors for the conservative parts. The whole scheme is then shown to be entropy stable at the semi-discrete level.
\end{enumerate}

The rest of the article is organized as follows: In Section~\ref{OFTT_system}, we introduce the OFTT-Euler system and present its eigenvalues and eigenvectors. 
In Section~\ref{Entropy_analysis_OFTT}, we first analyze the entropy framework of the OFTT-Euler system. We then present a novel reformulation of the OFTT-Euler equations, which makes the equations suitable for designing entropy stable schemes. In Section~\ref{sec:semi_discrete}, we present the entropy-stable finite-difference scheme for the reformulated OFTT-Euler model. Section~\ref{sec:Fully_discrete} contains the time discretization. In Section~\ref{sec:NR}, we present numerical test cases. Finally, Section~\ref{sec:conc} provides concluding remarks.

\section{OFTT-Euler model for non-equilibrium Hydrodynamics}\label{OFTT_system}
Following~\cite{cheng2024high,cheng2025analysis}, the OFTT-Euler model is given by,
\begin{equation}\label{eq:tt_euler}
		\frac{\p}{\p t}\begin{pmatrix*}
			\rho\\
			\rho \bu\\
			\rho e_e\\
			\rho e_i\
		\end{pmatrix*}+\nabla\cdot\begin{pmatrix*}
		\rho \bu\\
		\rho \bu\otimes\bu+p \textbf{I}\\
		\rho e_e \bu\\
		\rho e_i \bu
		\end{pmatrix*}=\begin{pmatrix*}
		0\\
		\textbf{0}\\
		\textcolor{blue}{-\pe \nabla\cdot\bu}\\
		\textcolor{blue}{-\pii \nabla\cdot\bu}
		\end{pmatrix*}.\quad \begin{aligned}
		&\text{(a)}\\
		&\text{(b)}\\
		&\text{(c)}\\
		&\text{(d)}
		\end{aligned}
\end{equation}
Here, $\rho$ is the density, $\bu=(v_x,v_y)^\top$ is the velocity, $p$ is the total pressure of the fluid. The total fluid pressure $p$ is described as $p= \pe + \pii,$ where $p_e$ and $p_i$ are the pressures for electrons and ions, respectively. Similarly,  $e_e$ and $e_i$ are the specific internal energies for electrons and ions, respectively. Following~\cite {cheng2024high}, we consider the ideal gas behavior for both species and use the following equation of state:
\begin{equation}\label{EOS}
    \rho e_e = C_{ve}\rho T_e = \frac{\pe}{\gamma_e-1}~~\text{and}~~\rho e_i = C_{vi}\rho T_i = \frac{\pii}{\gamma_i-1}.
\end{equation}
Here, $C_{ve}$ and $C_{vi}$ denote the specific heats, and $T_e$ and $T_i$ are the temperatures for electrons and ions, respectively. The parameters $\gamma_e$ and $\gamma_i$ are the specific heat ratios for the species, with the constraints $C_{ve}>0$, $C_{vi}>0$, $\gamma_e >1$, and $\gamma_i> 1$. We also define the total energy,
\begin{equation}\label{Total energy}
    E = \rho e + \frac{\rho|\bu|^2}{2} = \frac{\pe}{\gamma_e-1} + \frac{\pii}{\gamma_i-1} + \frac{\rho|\bu|^2}{2}
\end{equation}
where $e = e_e + e_i$ is the specific total internal energy of the fluid. The equations~(\ref{eq:tt_euler}a) and (\ref{eq:tt_euler}b) are the conservation of mass and momentum, respectively. The other two equations are for the evolution of the internal energies for each species. We have highlighted the non-conservative terms in blue. Combining~\eqref{eq:tt_euler}, \eqref{EOS} and \eqref{Total energy}, we obtain the conservation of total energy,
\begin{equation*}
    \frac{\p E}{\p t} + \nabla\cdot\left(\left(E+p\right)\bu\right)=0.
\end{equation*}
Replacing the electron internal energy equations (\ref{eq:tt_euler}c)  with the total energy equation, the system~\eqref{eq:tt_euler} with working variables $\con=(\rho, \rho\bu, E, \pii)^\top$, can be expressed as follows: 
\begin{equation}\label{eq:tt_euler_con}
		\frac{\p}{\p t}\begin{pmatrix*}
			\rho\\
			\rho \bu\\
			E\\
			\pii
		\end{pmatrix*}+\nabla\cdot\begin{pmatrix*}
		\rho \bu\\
		\rho \bu\otimes\bu+p \textbf{I}\\
		\left(\left(E+p\right)\bu\right)\\
		\pii \bu
		\end{pmatrix*}=\textcolor{blue}{\underbrace{\begin{pmatrix*}
		0\\
		\textbf{0}\\
		0\\
		-(\gamma_i -1)\pii \nabla\cdot\bu
		\end{pmatrix*}}_{NC}}\quad \begin{aligned}
		&\text{(a)}\\
		&\text{(b)}\\
		&\text{(c)}\\
		&\text{(d)}
		\end{aligned}
\end{equation}
Again, we have highlighted the non-conservative terms in blue. The set of equations~\eqref{eq:tt_euler} can be written in quasilinear form,
\begin{equation*}
	\frac{\p \textbf{W}}{\p t} + \mathcal{A}_x\frac{\p \textbf{W}}{\p x} + \mathcal{A}_y\frac{\p \textbf{W}}{\p y} = 0,
\end{equation*}
where, $\textbf{W} = \left\{\rho,~\bu,~\pe,~\pii\right\}$ is the vector of primitive variables and the matrices $\mathcal{A}_x$ and $\mathcal{A}_y$ are given by,
\begin{align*}
		\mathcal{A}_x=\begin{pmatrix*}
			  v_x & \rho & 0 & 0 & 0\\
			0 & v_x & 0 & \frac{1}{\rho} & \frac{1}{\rho}\\
			0 & 0 & v_x & 0 & 0\\
			0 & \gamma_e \pe & 0 & v_x & 0\\
			0 & \gamma_i \pii & 0 & 0 & v_x
		\end{pmatrix*},~ \text{ and }
		\mathcal{A}_y=\begin{pmatrix*}
			  v_y & 0 & \rho & 0 & 0\\
			0 & v_y & 0 & 0 & 0\\
			0 & 0 & v_y & \frac{1}{\rho} & \frac{1}{\rho}\\
			0 & 0 & \gamma_e \pe & v_y & 0\\
			0 & 0 & \gamma_i \pii & 0 & v_y
		\end{pmatrix*}.
	\end{align*}
	For the solutions to be physically admissible, we consider the solution set,
	\begin{equation}
		\label{eq:tt_domain}
		\Omega=\{\con\in \mathbb{R}^5 |~\rho>0,~\pe>0,\pii>0\}.
	\end{equation}
	For the states $\con\in\Omega$, the system~\eqref{eq:tt_euler} is hyperbolic and the eigenvalues of the OFTT-Euler system are,
\begin{equation}
		\bb{\Lambda^d}=(v_d,~ v_d,~ v_d, ~v_d\pm c_f), \quad c_f = \sqrt{\frac{\left(\gamma_e \pe + \gamma_i \pii\right)}{\rho}}, \quad d\in\left\{x,y\right\}.
		\label{eq:eigenvalues_1}
\end{equation}
    The set of right eigenvectors of the matrix $\mathcal{A}_x$ corresponding to eigenvalues $v_x,~ v_x,~ v_x,$ and  $v_x\pm c_f$ are
    \begin{align*}\mathcal{R}^1_{v_x}=
		\begin{pmatrix*}
			  1\\
			0\\
			0\\
			0\\
			0 
		\end{pmatrix*},~\mathcal{R}^2_{v_x}=\begin{pmatrix*}
			  0\\
			0\\
			1\\
			0\\
			0 
		\end{pmatrix*},~\mathcal{R}^3_{v_x}=\begin{pmatrix*}
			  0\\
			0\\
			0\\
			-1\\
			1 
		\end{pmatrix*},\text{ and }~\mathcal{R}_{v_x\pm c_f}=\begin{pmatrix*}
			  \rho^2\\
			\pm\sqrt{\left(\gamma_e \pe + \gamma_i \pii\right)\rho}\\
			0\\
			\gamma_e \pe\rho\\
			\gamma_i \pii\rho 
		\end{pmatrix*},
	\end{align*}
	respectively.
  Similarly, the right eigenvectors of $\mathcal{A}_y$ corresponding to eigenvalues $v_y,~ v_y,~ v_y, ~v_y\pm c_f$ are,
    \begin{align*}\mathcal{R}^1_{v_y}=
		\begin{pmatrix*}
			  1\\
			0\\
			0\\
			0\\
			0 
		\end{pmatrix*},~\mathcal{R}^2_{v_y}=\begin{pmatrix*}
			  0\\
			1\\
			0\\
			0\\
			0 
		\end{pmatrix*},~\mathcal{R}^3_{v_y}=\begin{pmatrix*}
			  0\\
			0\\
			0\\
			-1\\
			1 
		\end{pmatrix*},\text{ and }~\mathcal{R}_{v_y\pm c_f}=\begin{pmatrix*}
			  \rho^2\\
			0\\
			\pm\sqrt{\left(\gamma_e \pe + \gamma_i \pii\right)\rho}\\
			\gamma_e \pe\rho\\
			\gamma_i \pii\rho 
		\end{pmatrix*},
	\end{align*}
	respectively. Now, we have the following result,
\begin{lemma} For the system \eqref{eq:tt_euler_con},  the characteristic fields corresponding to the eigenvalues $v_d$, $d\in\{x,y\}$ are linearly degenerate, and the characteristic fields corresponding to the eigenvalues $v_d\pm c_f$ are genuinely nonlinear.
\end{lemma}
\begin{proof}
	We will prove the result only in the $x$-direction. Proof for the $y$-direction is similar. For the eigenvalue $v_x$,
\[\nabla v_x = \{0,1,0,0,0\}.\] Now it's easy to see that,
\[\nabla v_x\cdot \mathcal{R}^i_{v_x} = 0,~~~i\in\{1,2,3\}.\]
For the eigenvalue $v_x\pm c_f$,
\[\nabla \left(v_x \pm c_f\right) = \left\{\mp\frac{\sqrt{\left(\pe\gamma_e+\pii\gamma_i\right)}}{2\rho\sqrt{\rho}},1,0,\pm\frac{\gamma_e}{2\sqrt{\left(\pe\gamma_e+\pii\gamma_i\right)\rho}},\pm\frac{\gamma_i}{2\sqrt{\left(\pe\gamma_e+\pii\gamma_i\right)\rho}}\right\}.\]
A simple calculation shows that,
\[\nabla \left(v_x\pm c_f\right)\cdot \mathcal{R}_{v_x\pm c_f} =\pm\frac{\left(\pe\gamma_e(1+\gamma_e)+\pii\gamma_i(1+\gamma_i)\right)\rho}{2\sqrt{\left(\pe\gamma_e+\pii\gamma_i\right)\rho}},\]
which clearly does not vanish for the admissible solutions.
\end{proof}
\section{Entropy analysis and reformulation}\label{Entropy_analysis_OFTT}
The entropy $\ent$ and the entropy flux $\entf_d$ for the OFTT-Euler system~\eqref{eq:tt_euler_con}, are given by,
\begin{align}\label{entropy-pair}
		\ent= -\rho s, \qquad \entf_d=-\rho v_d s,\qquad~~s=\frac{1}{(\gamma_e-1)}\ln \left( \dfrac{\pe}{\rho^{\gamma_e}} \right) + \frac{1}{(\gamma_i-1)}\ln \left( \dfrac{\pii}{\rho^{\gamma_i}} \right),
	\end{align}
    where, $d\in\left\{x,y\right\}$.  We now have the following result:

\begin{lemma}
	For a smooth solution of the OFTT-Euler system, we have the following equality,
	\begin{align*}
	\p_t s+ v_x \p_x s=0,
	\end{align*}
	which results in the entropy equality,
		\begin{equation}
		\p_t \ent+ \p_x \entf_x=0. \label{entropy_pair_equality}
	\end{equation}
\end{lemma}
\begin{proof}
From system~\eqref{eq:tt_euler}, we have, 
    \begin{align*}
	&\p_t \rho+ v_x \p_x \rho+ \rho \p_x v_x=0, \\
        &\p_t p_e+ v_x \p_x p_e+ p_e \p_x v_x + (\gamma_e-1)p_e\p_x v_x=0, \\
        &\p_t p_i+ v_x \p_x p_i+ p_i \p_x v_x + (\gamma_i-1)p_i\p_x v_x=0.
	\end{align*}
Using these, we get,
    \begin{align*}
        \p_t s =& -\frac{1}{(\gamma_e-1)\pe}\left\{\textcolor{magenta}{v_x \p_x p_e}+ \textcolor{red}{p_e \p_x v_x + (\gamma_e-1)p_e\p_x v_x}\right\} \\
        &-\frac{1}{(\gamma_i-1)\pii}\left\{\textcolor{JungleGreen}{v_x \p_x p_i}+ \textcolor{red}{p_i \p_x v_x + (\gamma_i-1)p_i\p_x v_x}\right\} \\
        &+ \left(\frac{\gamma_e}{(\gamma_e-1)\rho}+\frac{\gamma_i}{(\gamma_i-1)\rho}\right) \left\{\textcolor{Maroon}{v_x \p_x \rho}+ \textcolor{red}{\rho \p_x v_x}\right\},
    \end{align*}
    and
    
    \begin{align*}
       v_x \p_x s =& \textcolor{magenta}{\frac{v_x \p_x p_e}{(\gamma_e-1)\pe}} + \textcolor{JungleGreen}{\frac{v_x \p_x p_i}{(\gamma_i-1)\pii}}
        - \textcolor{Maroon}{\left(\frac{\gamma_e v_x\p_x \rho}{(\gamma_e-1)\rho}+\frac{\gamma_i v_x\p_x \rho}{(\gamma_i-1)\rho}\right)}.
    \end{align*}  
 Adding them results in,
    \begin{align*}
	\p_t s+ v_x \p_x s=0
	\end{align*}
	Combining this with mass conservation results in \eqref{entropy_pair_equality}.
\end{proof}  

\begin{remark}
		In two dimensions, the above entropy equality~\eqref{entropy_pair_equality} is,
		\begin{equation}
			\label{entropy_pair_equality_2d}
			\p_t \ent + \p_x \entf_x + \p_y \entf_y = 0,
		\end{equation}
		and for non-smooth solutions, we have the entropy inequality,
		\begin{equation}
			\label{entropy_pair_equality_inq}
			\p_t \ent+ \p_x \entf_x+\p_y\entf_y \le0.
		\end{equation}
	\end{remark}
\subsection{Reformulation of OFTT-Euler System}\label{sec:reformulation} Following ~\cite{yadav2023entropy,singh2024entropy}, we want to reformulate the OFTT-Euler equations in such a way that non-conservative terms do not contribute to the entropy. We propose the following reformulation of the OFTT-Euler system~\eqref{eq:tt_euler_con} as,
    \begin{equation}\label{eq:tt_euler_reformulate}
		\frac{\p}{\p t}\begin{pmatrix*}
			\rho\\
			\rho \bu\\
			E\\
			\pii
		\end{pmatrix*}+\nabla\cdot\begin{pmatrix*}
		\rho \bu\\
		\rho \bu\otimes\bu+2\pe \textbf{I}\\
		\left(E+2\pe\right)\bu\\
		\pii \bu
		\end{pmatrix*}=\textcolor{blue}{\begin{pmatrix*}
		0\\
		\nabla\cdot\left(\left(\pe-\pii\right)\textbf{I}\right)\\
		\nabla\cdot\left(\left(\pe-\pii\right)\bu\right)\\
		-(\gamma_i -1)\pii \nabla\cdot\bu
		\end{pmatrix*}}.\quad \begin{aligned}
		&\text{(a)}\\
		&\text{(b)}\\
		&\text{(c)}\\
		&\text{(d)}
		\end{aligned}
\end{equation}
where we will treat blue colored terms as a non-conservative product. The resulting system of equations can be written as,
    \begin{align}\label{eq:oftt_ref_noncons}
		\frac{\p \con}{\p t}+\frac{\p \f_{x}}{\p x} + \frac{\p \f_{y}}{\p y} + \bc_{x}(\con)\frac{\p \con}{\p x} + \bc_{y}(\con)\frac{\p \con}{\p y}=0.
    \end{align}
    where the flux functions $\f_x$ and $\f_y$ are,
    \begin{align*}
		\f_x=\begin{pmatrix*}
			\rho v_x\\
			\rho v_x^2 + 2\pe\\
			\rho v_x v_y \\
                \left(E+ 2\pe\right)v_x\\
			\pii v_x
		\end{pmatrix*},~
		\textbf{f}_y&=\begin{pmatrix*}
			\rho v_y\\
			\rho v_x v_y \\
                \rho v_y^2 + 2\pe\\
                \left(E+ 2\pe\right)v_y\\
			\pii v_y
		\end{pmatrix*},
	\end{align*}
    The matrices $\textcolor{blue}{\bc_{x}}$ and $\textcolor{blue}{\bc_y}$ are given by,
    \begin{align*}
		\textcolor{blue}{\bc_{x}(\con)=\begin{pmatrix*}
			0 & 0 & 0 & 0 & 0\\
			-\frac{(\gamma_e-1)|\bu|^2}{2} & (\gamma_e-1)v_x & (\gamma_e-1)v_y & -(\gamma_e-1) & \frac{\gamma_e-1}{\gamma_i-1} + 1\\
			0 & 0 & 0 & 0 & 0\\
                -\left(\frac{(\gamma_e-1)|\bu|^2}{2} + \frac{\pii-\pe}{\rho}\right)v_x & (\gamma_e-1) v_x^2 + \frac{\pii-\pe}{\rho} & (\gamma_e-1)v_x v_y & -(\gamma_e-1)v_x & \left(\frac{\gamma_e-1}{\gamma_i-1} + 1\right)v_x\\
			-\frac{(\gamma_i-1)\pii v_x}{\rho} & \frac{(\gamma_i-1)\pii}{\rho} & 0 & 0 & 0
		\end{pmatrix*}},
    \end{align*}
    and
    \begin{align*}
		\textcolor{blue}{\bc_{y}(\con)=\begin{pmatrix*}
			0 & 0 & 0 & 0 & 0\\
                0 & 0 & 0 & 0 & 0\\
			-\frac{(\gamma_e-1)|\bu|^2}{2} & (\gamma_e-1)v_x & (\gamma_e-1)v_y & -(\gamma_e-1) & \frac{\gamma_e-1}{\gamma_i-1} + 1\\
                -\left(\frac{(\gamma_e-1)|\bu|^2}{2} + \frac{\pii-\pe}{\rho}\right)v_y & (\gamma_e-1)v_x v_y & (\gamma_e-1) v_y^2 + \frac{\pii-\pe}{\rho}  & -(\gamma_e-1)v_y & \left(\frac{\gamma_e-1}{\gamma_i-1} + 1\right)v_y\\
			-\frac{(\gamma_i-1)\pii v_y}{\rho} & 0 & \frac{(\gamma_i-1)\pii}{\rho} & 0 & 0
		\end{pmatrix*}},
    \end{align*}
    respectively. The entropy variable $\evar=\frac{\p \ent}{\p \con}$ for the OFTT-Euler system~\eqref{eq:tt_euler_con}, is given by,
    \begin{align}
    		\evar&=\Bigg(\frac{\gamma_e}{(\gamma_e-1)}+\frac{\gamma_i}{(\gamma_i-1)}-s-\frac{\beta_e |\bu|^2}{2},~\beta_e\bu,~-\beta_e,~\frac{\beta_e-\beta_i}{(\gamma_i-1)}\Bigg)^\top,
    		\label{eq:envar}
    	\end{align} 
        where, $\beta_e=\frac{\rho}{\pe}$ and $\beta_i=\frac{\rho}{\pii}.$ We note that similar to the case in~\cite{yadav2023entropy,singh2024entropy}, the flux functions $\f_d$ with $d\in\{x,y\},$ satisfy the condition,
\begin{equation}\label{eq:ent_def}
			{\entf_d}'(\con) = \evar{\f_d}'(\con),
\end{equation}
and the non-conservative matrices $\bc_{d}(\con), d\in\{x,y\}$ satisfy the condition, 
\begin{equation}\label{eq:ent_noncons}
			\evar^\top \bc_d(\con) = 0.
\end{equation}
The proof is provided in Appendix~\ref{Non_con_entropy_evolution}.

\begin{remark}
    In Appendix~\ref{Non_Symmetrizability}, we discuss the symmetrizability of the OFTT-Euler system. We show that the complete system is not symmetrizable. However, if we only consider the conservative part, the system is indeed symmetrizable due to the presence of an entropy pair.
\end{remark}
\section{Semi-discrete numerical schemes}\label{sec:semi_discrete}
The reformulated OFTT-Euler equation~\eqref{eq:oftt_ref_noncons} are suitable for designing entropy stable finite difference numerical schemes. We start by considering a uniform mesh of size $(\Delta x,\Delta y)$ with cell centers at the grid points, $(x_i,y_j), \;\;0 \leq i \leq n_x$ and $0 \leq j \leq n_y.$  The vertices of the cells are given by $(x_{i+\frac{1}{2}}, y_{j+\frac{1}{2}})$ with $x_{i+\frac{1}{2}} = \frac{x_{i} + x_{i+1}}{2}$ and $y_{j+\frac{1}{2}} = \frac{y_{j} + y_{j+1}}{2}.$

A semi-discrete finite difference numerical scheme for the system~\eqref{eq:oftt_ref_noncons}, can be written as,
        \begin{align}\label{eq:semi-discrete_fd}
			\frac{d \con_{i,j}}{d t}= -\frac{\F_{x, i+\frac{1}{2}, j}-\F_{x, i-\frac{1}{2}, j}}{\Delta x}-\frac{\F_{y, i, j+\frac{1}{2}}-\F_{y, i, j-\frac{1}{2}}}{\Delta y} -\textcolor{blue}{\bc_x(\con_{i,j})\left(\frac{\p \con}{\p x} \right)_{i,j}} -\textcolor{blue}{\bc_y(\con_{i,j}) \left(\frac{\p \con}{\p y} \right)_{i,j}}.
	\end{align}
  Here, the numerical fluxes $\F_{x, i+\frac{1}{2}, j}$ and $\F_{ y,i, j+\frac{1}{2}}$ are consistent with the continuous fluxes $\f_{x}$ and $\f_{y}$, respectively. Additionally, $\left(\frac{\p \con}{\p x}  \right)_{i,j}$ and $\left(\frac{\p \con}{\p y} \right)_{i,j}$ are discretized using central difference approximations of appropriate order. We begin by developing higher-order entropy conservative schemes for system~\eqref{eq:oftt_ref_noncons}.
\subsection{Higher order entropy conservative schemes}
\label{subsec:ent_cons}
For a grid function $a_{i,j}$, let us first introduce the following notation for the jumps 
\begin{align*}
	[\![(\cdot)]\!]_{\iph,j} =(\cdot)_{i+1,j}-(\cdot)_{i,j},\qquad 
	[\![(\cdot)]\!]_{i,\jph} =(\cdot)_{i,j+1}-(\cdot)_{i,j},
\end{align*}
and averages
\begin{align*}
\overline{(\cdot)}_{\iph,j}= \frac{(\cdot)_{i+1,j} + (\cdot)_{i,j}}{2}, \qquad\overline{(\cdot)}_{i,\jph}= \frac{(\cdot)_{i,j+1} + (\cdot)_{i,j}}{2},
\end{align*}
 across the cell boundaries. For the conservative part, we recall that symmetric, consistent numerical fluxes $\tilde{\F}_{x,i+\frac{1}{2}, j}$ and $\tilde{\F}_{ y,i, j+\frac{1}{2}}$ satisfying 
\begin{align}
	\quad \jump{\evar}_{\iph,j}  \cdot \tilde{\F}_{x,i+\frac{1}{2}, j}=\jump{\mathcal{F}_x}_{\iph,j},~~~~~~
	\quad\jump{\evar}_{i,\jph}  \cdot \tilde{\F}_{y,i,j+\frac{1}{2}}&=\jump{\mathcal{F}_y}_{i,\jph}.\label{eq:consrvative_flux}
\end{align}
are entropy conservative (see~\cite{tadmor1987numerical,tadmor2003entropy}) and second-order accurate. Here,
\begin{align}
	\mathcal{F}_d&=\evar \cdot \textbf{f}_d -\entf_d=2\rho v_d,\quad d\in\{x,y\}.\label{entropy_potential}
\end{align}
are {\em entropy potentials.} We will now derive the expressions for these numerical fluxes. For the simplicity of the notation, we will drop the grid notation by subscripts $i$ and $j$ and instead use subscripts $l$ and $r$ to represent left and right states, respectively. Then the jumps and averages of these states are denoted by,
$$\jump{a}= a_r-a_l\quad \text{ and }\quad\bar{a} = \frac{a_r+a_l}{2},$$
for a scalar function $a$. 
Let us also denote the logarithmic average 
$$
a^{\ln} = \frac{\jump{a}}{\jump{\ln{a}}}.
$$
We will now describe the derivation of the entropy conservative numerical flux $\tilde{\F}_x$  in $x$-direction, denoted componentwise as,
$$
\tilde{\F}_x(\con_l,\con_r)=\left[ \tilde{f}_{x}^{(1)},\tilde{f}_{x}^{(2)}, \tilde{f}_{x}^{(3)}, \tilde{f}_{x}^{(4)}, \tilde{f}_{x}^{(5)}\right]^{\top}.
$$ 
 In~\cite{tadmor1987numerical,tadmor2003entropy,chandrashekar2013kinetic,roe2006affordable,ismail2009affordable}, several authors have proposed different approaches for calculating an entropy-conservative flux. To derive the expression for the flux $\tilde{\F}_{x},$ we follow the procedure described in~\cite{chandrashekar2013kinetic}.

Our aim is to construct an entropy conservative flux $\tilde{\F}_x$ that satisfied the identity 
\begin{equation}\label{eq:EC_condition}
	\jump{\evar}\cdot \tilde{\F}_x = \jump{\mathcal{F}_x}.
\end{equation}
We first change the expression for jump in $[\![ \evar]\!],$ to express it in terms of the jumps $[\![ \rho]\!], [\![ v_{x}]\!], [\![ v_{y}]\!], [\![ \beta_e]\!], [\![ \beta_i]\!]$ as
\[[\![ \evar]\!]=\begin{pmatrix*}
	\frac{2}{\rho^{\ln}} [\![ \rho]\!]  + \frac{1}{(\gamma_i-1)}\dfrac{[\![ \beta_{i}]\!]}{{\beta}^{\ln}_{i}} -  \bar{\beta}_e\left(\bar{v}_{x} [\![ v_{x}]\!] +\bar{v}_{y} [\![ v_{y}]\!]  \right)+\left[\frac{1}{(\gamma_e-1){\beta}^{\ln}_e}-\frac{\overline{|v|^{2}}}{2}\right][\![ \beta_e]\!] \\
	 \bar{\beta}_e [\![ v_{x}]\!] + \bar{v}_{x} [\![ \beta_e]\!] \\
	 \bar{\beta}_e [\![ v_{y}]\!] + \bar{v}_{y} [\![ \beta_e]\!] \\
	- [\![ \beta_e]\!] \\
	\frac{ [\![ \beta_e]\!] - [\![ \beta_i]\!]}{\gamma_i-1}
\end{pmatrix*}.\] Then  $\jump{\evar}\cdot \tilde{\F}_x$  can be written as,
\begin{align}
		[\![ \evar]\!] \cdot \tilde{\F}_{x}=& \frac{2\tilde{f}_{x}^{(1)}}{{\rho}^{\ln}} [\![ \rho]\!] +\left(- \bar{v}_{x} \bar{\beta}_e \tilde{f}_{x}^{(1)} + \bar{\beta}_e \tilde{f}_{x}^{(2)}\right) [\![ v_{x}]\!]\nonumber \\
		&+ \left(- \bar{v}_{y} \bar{\beta}_e \tilde{f}_{x}^{(1)} + \bar{\beta}_e \tilde{f}_{x}^{(3)}\right) [\![ v_{y}]\!] +\left(\frac{\tilde{f}_{x}^{(1)}}{(\gamma_i-1){\beta}^{\ln}_i}-\frac{\tilde{f}_{x}^{(5)}}{(\gamma_i-1)}\right)[\![ \beta_i]\!] \nonumber\\
		\label{eq:vdotf}
		&+\Bigg[\left(\frac{1}{(\gamma_e-1){\beta}^{\ln}_e}-\frac{\overline{|v|^{2}}}{2}\right) \tilde{f}_{x}^{(1)} + \bar{v}_{x} \tilde{f}_{x}^{(2)} + \bar{v}_{y} \tilde{f}_{x}^{(3)} - \tilde{f}_{x}^{(4)}+\frac{\tilde{f}_{x}^{(5)}}{(\gamma_i-1)}\Bigg] [\![ \beta_e]\!].
	\end{align}
    Also, the right hand side of \eqref{eq:EC_condition}, simplifies to,
 \begin{align}
		[\![ \mathcal{F}_{x}]\!]=& [\![2\rho v_{x}]\!] = 2\bar{v}_{x} [\![ \rho]\!] + 2\bar{\rho}[\![ v_x]\!]. \label{eq:vdotf_rhs}
\end{align}
By equating to jumps $[\![ \rho]\!], [\![ v_x]\!], [\![ v_y]\!], [\![ \beta_e]\!],$ and $[\![\beta_i]\!]$, in~\eqref{eq:vdotf} and~\eqref{eq:vdotf_rhs}, we get,
$$
	\begin{aligned}
		\tilde{f}_{x}^{(1)}=& {\rho}^{\ln} \bar{v}_{x},~~\tilde{f}_{x}^{(2)}= \frac{2\bar{\rho}}{ \bar{\beta}_e}+\bar{v}_{x} \tilde{f}_{x}^{(1)},~~
		\tilde{f}_{x}^{(3)}= \bar{v}_{y} \tilde{f}_{x}^{(1)},~~\tilde{f}_{x}^{(5)}= \frac{\tilde{f}_{x}^{(1)}}{{\beta}^{\ln}_i},\\
		\tilde{f}_{x}^{(4)}=& \left[\frac{1}{ (\gamma_e-1) {\beta}^{\ln}_e} -\frac{\overline{|v|^{2}}}{2}\right] \tilde{f}_{x}^{(1)}+\bar{v}_{x} \tilde{f}_{x}^{(2)}+\bar{v}_{y} \tilde{f}_{x}^{(3)}+\frac{\tilde{f}_{x}^{(5)}}{(\gamma_i-1)}.
	\end{aligned}
	$$
	Similarly, we can derive the expression of entropy conservative numerical flux in $y$-direction $\tilde{\F}_y=\left[ \tilde{f}_{y}^{(1)},\tilde{f}_{y}^{(2)}, \tilde{f}_{y}^{(3)}, \tilde{f}_{y}^{(4)}, \tilde{f}_{y}^{(5)}\right]^{\top}$, which results in,
    $$
	\begin{aligned}
		\tilde{f}_{y}^{(1)}=& {\rho}^{\ln} \bar{v}_{y},~~
		\tilde{f}_{y}^{(2)}= \bar{v}_{x} \tilde{f}_{y}^{(1)},~~\tilde{f}_{y}^{(3)}= \frac{2\bar{\rho}}{ \bar{\beta}_e}+\bar{v}_{y} \tilde{f}_{y}^{(1)},~~\tilde{f}_{y}^{(5)}= \frac{\tilde{f}_{y}^{(1)}}{{\beta}^{\ln}_i},\\
		\tilde{f}_{y}^{(4)}=& \left[\frac{1}{ (\gamma_e-1) {\beta}^{\ln}_e} -\frac{\overline{|v|^{2}}}{2}\right] \tilde{f}_{y}^{(1)}+\bar{v}_{x} \tilde{f}_{y}^{(2)}+\bar{v}_{y} \tilde{f}_{y}^{(3)}+\frac{\tilde{f}_{y}^{(5)}}{(\gamma_i-1)}.
	\end{aligned}
	$$
	It is easy to check that the above numerical fluxes are consistent with the continuous fluxes,
$\f_x$ and $\f_y.$ 

\begin{remark}
The scheme~\eqref{eq:semi-discrete_fd} with fluxes $\tilde{\F}_{x,\iph,j} =\tilde{\F}_x(\con_{i,j},\con_{i+1,j})$ and $\tilde{\F}_{y,i,\jph} =\tilde{\F}_y(\con_{i,j},\con_{i,j+1})$ together with the $2^{nd}$-order central difference approximation of $\left(\frac{\p \con}{\p x}  \right)_{i,j}$ and $\left(\frac{\p \con}{\p y} \right)_{i,j}$, is $2^{nd}$-order accurate approximation of~\eqref{eq:oftt_ref_noncons} and entropy conservative i.e. 
	\begin{align}
			\label{eq:ent_eql_2nd_order}
			\frac{d}{d t} \ent\left(\con_{i, j}\right)+\frac{\tilde{\entf}_{x, i+\frac{1}{2}, j}-\tilde{\entf}_{x, i-\frac{1}{2}, j}}{\Delta x}+\frac{\tilde{\entf}_{y, i, j+\frac{1}{2}}-\tilde{\entf}_{y, i, j-\frac{1}{2}}}{\Delta y}=0,
		\end{align}
		holds. Here, the entropy fluxes $\tilde{\entf}_{x}$ and $\tilde{\entf}_{y}$ are defined as
    \begin{align*}
			\tilde{\entf}_{x, i+\frac{1}{2}, j}=\bar{\evar}_{i+\frac{1}{2}, j} \cdot \tilde{\F}_{x,i+\frac{1}{2}, j} - \bar{\mathcal{F}}_{x, i+\frac{1}{2}, j}~~\text{and}~~
			\tilde{\entf}_{y, i, j+\frac{1}{2}}=\bar{\evar}_{i, j+\frac{1}{2}} \cdot \tilde{\F}_{ y,i, j+\frac{1}{2}} - \bar{\mathcal{F}}_{y, i, j+\frac{1}{2}}
		\end{align*}
    which are consistent with~\eqref{entropy-pair} (see lemma~\eqref{lem:entropy_flux_consistency} in Appendix~\ref{Cons_entropy_flux}). 
 The proof follows from the entropy conservation property of the numerical fluxes and \eqref{eq:ent_noncons}.
\end{remark}

    Following~\cite{leFloch2002}, we can construct $2p^{th}$-order accurate numerical fluxes using second-order fluxes with $p\in\mathbb{Z}^+$. For $p=2$, the $4^{th}$-order entropy-conservative fluxes $\tilde{\F}_{x,i+\frac{1}{2},j}^4$ and $\tilde{\F^4}_{y,i,j+\frac{1}{2}}$ are given as follows:
    \begin{align}
	\tilde{\F}_{x,i+\frac{1}{2},j}^4&=\frac{4}{3}\tilde{\F}_{x}(\con_{i,j},\con_{i+1,j})-\frac{1}{6} \bigg( \tilde{\F}_{x} (\con_{i-1,j},\con_{i+1,j})+
	\tilde{\F}_{x}(\con_{i,j},\con_{i+2,j}) \bigg)\label{eq:4thorder_numflux_x}
\end{align}
and
\begin{align}
		\tilde{\F}_{y,i,j+\frac{1}{2}}^4&=\frac{4}{3}\tilde{\F}_{y}(\con_{i,j},\con_{i,j+1})-\frac{1}{6} \bigg( \tilde{\F}_{y} (\con_{i,j-1},\con_{i,j+1})+
		\tilde{\F}_{y}(\con_{i,j},\con_{i,j+2}) \bigg).\label{eq:4thorder_numflux_y}
\end{align}
\begin{remark}
By replacing the $2^{nd}$-order fluxes with the $4^{th}$-order fluxes and using $4^{th}$-order central difference approximations to approximate  $\left(\frac{\p \con}{\p x}  \right)_{i,j}$ and $\left(\frac{\p \con}{\p y} \right)_{i,j}$ , the scheme ~\eqref{eq:semi-discrete_fd} is a $4^{th}$-order entropy-conservative schemes i.e.
	\begin{align}
	\label{eq:ent_eql_4rth_order}
	\frac{d}{d t} \ent\left(\con_{i, j}\right)+\frac{\tilde{\entf}^4_{x, i+\frac{1}{2}, j}-\tilde{\entf}^4_{x, i-\frac{1}{2}, j}}{\Delta x}+\frac{\tilde{\entf}^4_{y, i, j+\frac{1}{2}}-\tilde{\entf}^4_{y, i, j-\frac{1}{2}}}{\Delta y}=0,
\end{align}
holds. Here, the numerical entropy fluxes $\tilde{\entf}^4_{x}$ and $\tilde{\entf}^4_{y}$ which are consistent with entropy fluxes~\eqref{entropy-pair} (see note~\eqref{note:consis_1} in Appendix~\ref{Cons_entropy_flux}). 
\end{remark}
\subsection{Higher-order entropy-stable schemes}
\label{subsec:ent_stable}
The schemes defined in Section~\ref{subsec:ent_cons} are designed to conserve entropy. However, in the presence of discontinuities, it is necessary to have appropriate entropy decay; otherwise, the numerical solution will contain undesirable numerical oscillations. Following~\cite{tadmor2003entropy}, we modify the numerical flux to include dissipation terms as,
\begin{equation}
	\begin{aligned}
		{\hat{\F}}_{x,i+\frac{1}{2},j} =\tilde{\F}_{x,i+\frac{1}{2},j} - \frac{1}{2} \textbf{D}_{x,i+\frac{1}{2},j}[\![ \evar]\!]_{i+\frac{1}{2},j},~
		\quad
		{\hat{\F}}_{y,i,j+\frac{1}{2}} = \tilde{\F}_{y,i,j+\frac{1}{2}} - \frac{1}{2} \textbf{D}_{y,i,j+\frac{1}{2}}[\![ \evar]\!]_{i,j+\frac{1}{2}}.
		\label{es_numflux}
	\end{aligned}
\end{equation}
The diffusion matrices $\textbf{D}_{x,i+\frac{1}{2},j}$ and $\textbf{D}_{y,i,j+\frac{1}{2}}$ are symmetric positive definite and are based on Rusanov’s type diffusion operators, given by,
\begin{equation}  \label{diffusiontype}
	\textbf{D}_{x,i+\frac{1}{2},j} = \tilde{R}_{x,i+\frac{1}{2},j} \Lambda_{x,i+\frac{1}{2},j} \tilde{R}_{x,i+\frac{1}{2},j}^{\top}~\text{and}~\textbf{D}_{y,i,j+\frac{1}{2}} = \tilde{R}_{y,i,j+\frac{1}{2}} \Lambda_{y,i,j+\frac{1}{2}} \tilde{R}_{y,i,j+\frac{1}{2}}^{\top}.
\end{equation}
Here, $\tilde{R}_d,~d\in\left\{x,y\right\}$ are the matrices of entropy-scaled right eigenvectors (see~\cite{barth1999numerical}) presented in Appendix~\ref{Eiegn_vector_for_con}. Also, the matrices are ${\Lambda_d}$ are diagonal matrices, given by,
$${\Lambda_d}=\left( \max_{\lambda\in \tilde{\mathbf{\Lambda}}_d} |\lambda|\right) \mathbf{I}_{5 \times 5},$$
where, $\tilde{\mathbf{\Lambda}}_d$ is the set of eigenvalues of the flux Jacobian $\frac{\p \f_d}{\p \con}$ given in Appendix~\ref{Eiegn_vector_for_con}. 

Thus, the numerical scheme~\eqref{eq:semi-discrete_fd} with the modified flux~\eqref{es_numflux} is entropy stable. However, the first-order jump terms $[\![\evar]\!]_{i+\frac{1}{2},j}$ and $[\![\evar]\!]_{i,j+\frac{1}{2}}$ restricts the overall accuracy to first order. Replacing these jumps with higher-order polynomial reconstructions improves accuracy but does not guarantee entropy stability. To address this issue, we follow the approach proposed in~\cite{fjordholm2012arbitrarily} and describe the process in $x-$direction. Same process is followed for reconstruction in $y-$direction. Let us introduce {\em scaled entropy variables}
$$\mathcal{W}^{\pm}_{x,i,j}= \tilde{R}^{{\top}}_{x,i\pm\frac{1}{2},j}\evar_{i,j}.$$
Using ENO reconstruction procedure, we reconstruct $\mathcal{W}^{\pm}_{x,i,j}$ with the $k^{th}$-degree polynomials $\mathcal{P}^{\pm}_{x,i,j}(x)$ and define, 
\[\hat{\mathcal{W}}_{x,i,j}^{\pm}=\mathcal{P}^{\pm}_{x,i,j}(x_{i\pm\frac{1}{2}}).\]
Using these values, the reconstructed $\jump{\evar}_{\iph,j}$ of $k$-th order is defined as,
$$
\jump{\hat{\evar}}^k_{x,i,j} = \hat{\evar}^-_{x,i+1,j}-\hat{\evar}^+_{x,i,j}.
$$
where
$$
\hat{\evar}^{\pm}_{x,i+\frac{1}{2},j} =  \left\lbrace \tilde{R}^{{\top}}_{x,i\pm\frac{1}{2},j}\right\rbrace ^{(-1)}\hat{\mathcal{W}}_{x,i,j}^{\pm}.
$$
Following~\cite{fjordholm2012arbitrarily}, we define the higher-order entropy stable  numerical flux in the $x$-direction as follows:
\begin{equation}
	{\hat{\F}}^{k}_{x,i+\frac{1}{2},j}\,=\,\tilde{\F}^{2p}_{x,i+\frac{1}{2},j}\,-\,\frac{1}{2}\,\textbf{D}_{x,i+\frac{1}{2},j}[\![ \hat{\evar}]\!]_{x,i+\frac{1}{2},j}^k,
	\label{eq:entropy_stable_flux_x}
\end{equation}
where, $k$ denotes the order of the scheme and $p\in\mathbb{N}$ is given by
$$p = \begin{cases}
	\frac{k}{2}, & \textrm{if } k ~\text{is even,}\\
	\frac{k+1}{2}, & \textrm{if } k ~\text{is odd} .
\end{cases}$$ 
To ensure the entropy stability of the modified flux, the reconstruction procedure for $\mathcal{W}$ must satisfy the sign-preserving property~\cite{fjordholm2012arbitrarily}, which is a sufficient condition for ensuring that the numerical flux~\eqref{eq:entropy_stable_flux_x} is entropy-stable. For the $2^{nd}$-order scheme, the MinMod reconstruction is used to ensure this property. For higher-order schemes i.e. $3^{rd}$-order $(k=3)$ and $4^{th}$-order $(k=4)$ schemes, we employ ENO-based reconstruction~\cite{fjordholm2013eno}, which ensures {\em sign-preserving property} of the reconstruction procedure. Following the same procedure in $y$-direction, we can define a higher numerical flux in the $y$-direction as follows:
\begin{equation}
	{\hat{\F}}^{k}_{y,i,j+\frac{1}{2}}\,=\,\tilde{\F}^{2p}_{y,i,j+\frac{1}{2}}\,-\,\frac{1}{2}\,\textbf{D}_{y,i,j+\frac{1}{2}}[\![ \hat{\evar}]\!]_{y,i,j+\frac{1}{2}}^k.
	\label{eq:entropy_stable_flux_y}
\end{equation}
Now we have the following result:
\begin{thm}[]\label{thm:ES_higher_order}
	The semi-discrete scheme \eqref{eq:semi-discrete_fd} with entropy stable fluxes~\eqref{eq:entropy_stable_flux_x},~\eqref{eq:entropy_stable_flux_y} and with $2^{nd}$-order (for $k=2$) and $4^{th}$-order (for $k=3,4$) central difference approximations to approximate $\left(\frac{\p \con}{\p x}\right)_{i,j}$ and $\left(\frac{\p \con}{\p y}\right)_{i,j},$  is $k^{th}$-order accurate and entropy stable, i.e. it satisfies,
	\begin{equation}
		\label{eq:semi-disc_ent_stab}
		\frac{d}{dt}  \ent(\con_{i,j})  +\frac{1}{\Delta x} \left( \hat{\entf}^k_{x,i+\frac{1}{2},j} - \hat{\entf}^k_{x,i-\frac{1}{2},j}\right)+\frac{1}{\Delta y}\left( \hat{\entf}^k_{y,i,j+\frac{1}{2}} - \hat{\entf}^k_{y,i,j-\frac{1}{2}}\right) \le 0,
	\end{equation}
	where $ \hat{\entf}^k_x$ and $ \hat{\entf}^k_y$ are given by,
	\begin{equation*}
		\begin{aligned}
			\hat{\entf}^k_{x,i+\frac{1}{2},j}=  \tilde{\entf}^{2p}_{x,i+\frac{1}{2},j} - \frac{1}{2}\bar{\evar}_{i+\frac{1}{2},j}^{\top}  \textbf{D}_{x,i+\frac{1}{2},j}[\![ \hat{\evar}]\!]^k_{x,i+\frac{1}{2},j}
	\end{aligned} \end{equation*}
	
	and \begin{equation*}
		\begin{aligned}
			\hat{\entf}^k_{y,i,j+\frac{1}{2}}=    \tilde{\entf}^{2p}_{y,i,j+\frac{1}{2}} -  \frac{1}{2}\bar{\evar}_{i,j+\frac{1}{2}}^{\top}  \textbf{D}_{y,i,j+\frac{1}{2}}[\![ \hat{\evar}]\!]^k_{y,i,j+\frac{1}{2}}.\end{aligned} \end{equation*}
\end{thm}
are consistent numerical entropy fluxes with continuous entropy fluxes $\entf_x$ and $\entf_y,$ respectively (see note~\eqref{note:consis_2} in Appendix~\ref{Cons_entropy_flux}).
\section{Fully discrete schemes}\label{sec:Fully_discrete}
The semi-discrete scheme~\eqref{eq:semi-discrete_fd} can be expressed as follows:
\begin{equation}
	\frac{d }{dt}\con_{i,j}(t) = \mathcal{L}_{i,j}(\con(t))
	\label{fullydiscrete}
\end{equation}
where,\\
\scalebox{0.9}{%
	\begin{minipage}{1.1\linewidth} 
		\begin{align*}
			\mathcal{L}_{i,j}(\con(t))=-\frac{\F_{x, i+\frac{1}{2}, j}-\F_{x, i-\frac{1}{2}, j}}{\Delta x}-\frac{\F_{y, i, j+\frac{1}{2}}-\F_{y, i, j-\frac{1}{2}}}{\Delta y} -\textcolor{blue}{\bc_x(\con_{i,j})\left(\frac{\p \con}{\p x} \right)_{i,j} -\bc_y(\con_{i,j}) \left(\frac{\p \con}{\p y} \right)_{i,j}}
		\end{align*}
	\end{minipage}%
}\vspace{0.5cm}\\
Here $\con$ is a grid function. Following \cite{gottlieb2001strong}, we use explicit strong stability-preserving Runge-Kutta (SSP-RK) methods for the time discretizations. Let $\con^n_{i,j}$ be the solution at time level $t^n$, with time step $\Delta t = t^{n+1}-t^n$. Then to update the solutions by one time step, the $2^{nd}$ and $3^{rd}$-order accurate SSP–RK schemes are given as follows:
\begin{enumerate}
		\item Set $\con^{0}_{i,j} \ = \ \con^n_{i,j}$.
		\item Compute
		{\small
			\begin{eqnarray*}
				\con_{i,j}^{(k)} \
				= \
				\sum_{l=0}^{k-1}\mu_{kl}\con_{i,j}^{(l)}
				+
				\nu_{kl}\Delta t \big(\mathcal{L}_{i,j}(\con^{(l)})\big),~~~~~~\text{For}~k\in\{1,\dots,m+1\}
		\end{eqnarray*}}
		where $\mu_{kl}$ and $\nu_{kl}$ are given in Table~\eqref{table:ssp}.
		\item Finally, $\con_{i,j}^{n+1} \ =  \ \con_{i,j}^{(m+1)}$.
	\end{enumerate}
	\begin{table}[h]
		\centering
		\begin{tabular}{l|ccc|ccc}
			\hline
			Order & \hspace{2.0cm}$\mu_{kl}$ & & & \hspace{2.0cm}$\nu_{kl}$ & & \\
			\hline
			2 & 1 & & & 1 & & \\
			& 1/2 & 1/2 & & 0 & 1/2 & \\
			\hline
			3 & 1 & & & 1 & & \\
			& 3/4 & 1/4 & & 0 & 1/4 & \\
			& 1/3 & 0 & 2/3 & 0 & 0 & 2/3 \\
			\hline
		\end{tabular}
				\caption{Coefficients for explicit SSP-Runge-Kutta schemes}
		\label{table:ssp}
	\end{table}
	The $4^{th}$-order SSP-RK scheme \cite{gottlieb2001strong} is given by, 
	\begin{subequations}
		\begin{align*}
			\con^{(1)} &= \textbf{U}^n + 0.39175222700392 \Delta t \big(\mathcal{L}(\con^n) \big)\\
			\con^{(2)} &= 0.44437049406734 \con^n + 0.55562950593266 \con^{(1)} +0.36841059262959  \Delta t \big( \mathcal{L}(\con^1)\big)\\
			\con^{(3)} &= 0.62010185138540 \con^n + 0.37989814861460 \con^{(2)} +0.25189177424738  \Delta t \big( \mathcal{L}(\con^2)\big)\\
			\con^{(4)} &= 0.17807995410773 \con^n + 0.82192004589227 \con^{(3)} + 0.54497475021237  \Delta t \big( \mathcal{L}(\con^3) \big)\\
			\con^{n+1} &= 0.00683325884039 \con^n + 0.51723167208978 \con^{(2)} + 0.12759831133288 \con^{(3)}\\
			&+ 0.34833675773694 \con^{(4)}+ 0.08460416338212  \Delta t \big( \mathcal{L}(\con^3) \big)+ 0.22600748319395  \Delta t \big( \mathcal{L}(\con^4) \big).
		\end{align*}	
	\end{subequations}
	Here, we have ignored the subscript for each cell for simplicity. Using the time discretizations, we now denote the compatible discretizations as listed in Table\eqref{tab:schemes}.
	
\begin{table}[h!]
\centering\small
\begin{tabular}{|c|c|c|c|}
\hline
\textbf{Scheme} & \textbf{Entropy Stable Scheme} & \textbf{$\left(\frac{\p \con}{\p x}\right)_{i,j}$ and $\left(\frac{\p \con}{\p y}\right)_{i,j}$ Approx.} & \textbf{Time Integration}  \\ \hline
$\ote$ & $2^{nd}$-order & $2^{nd}$-order central difference& SSP-RK2  \\ \hline
$\othe$  & $3^{rd}$-order & $4^{th}$-order central difference& SSP-RK3  \\ \hline
$\ofe$ & $4^{th}$-order & $4^{th}$-order central difference & SSP-RK4  \\ \hline
\end{tabular}
\caption{Description of fully discrete numerical schemes}
\label{tab:schemes}
\end{table}
\section{Numerical Results}
\label{sec:NR}
To calculate the time step, we use
$$
\Delta t = \text{CFL} \frac{1}{\max_{i,j} \left( \frac{\lambda_x(\con_{i,j})}{\Delta x} + \frac{\lambda_y(\con_{i,j})}{\Delta y} \right)},
$$
where,

$${\lambda_d}(\con_{i,j})=\max_{\lambda\in {\mathbf{\Lambda}}_d} |\lambda(\con_{i,j})|,\qquad d\in\{x,y\} $$
are the absolute values of the maximum eigenvalues in each direction. For one-dimensional test cases, we take $\text{CFL} = 0.475$ for all schemes. For two-dimensional test cases, we take $\text{CFL} = 0.3$ for all schemes.
To measure the total entropy production at every time step, we also compute the total entropy decay at each time step given by,
\begin{align}
	\label{eq:tot_ent_exp}
	\sum_{i,j} \Big[\ent^{n+1}(\con_{i,j}) - \ent^{n}(\con_{i,j})-\frac{\Delta t}{\Delta x} \left( \hat{\entf}^k_{x,i+\frac{1}{2},j} - \hat{\entf}^k_{x,i-\frac{1}{2},j}\right)
	-\frac{\Delta t}{\Delta y}\left( \hat{\entf}^k_{y,i,j+\frac{1}{2}} - \hat{\entf}^k_{y,i,j-\frac{1}{2}}\right)\Big].
\end{align}
This will test the consistency of the numerical results with~\eqref{eq:semi-disc_ent_stab}. We will now present the one and two-dimensional test cases.

For the one-dimensional Riemann problems, we compare the numerical solutions with exact solutions. Following~\cite{cheng2024high}, the exact solutions are obtained by solving the Riemann problems with Newton iterations. Across a shock, the Rankine–Hugoniot conditions for mass, momentum, and total energy do not determine how the energy is shared between electrons and ions, so an additional condition is needed. We use the two different conditions from~\cite{cheng2024high}. The first result in the solution denoted by PP, requires each species to satisfy its own Hugoniot relation, as described in Section 4.1 of~\cite{cheng2024high}. The second result, in the solution denoted by AE, assumes that the electrons remain isentropic across the shock, as described in Section 4.2 of~\cite{cheng2024high}. For the double rarefaction problem, both conditions give the same solution, since there is no shock. For the Sod problem, where $\gamma_e=\gamma_i$, both conditions results in the same density, velocity, and total pressure, but different electron and ion pressures behind the shock. For the Lax problem, where $\gamma_e\neq\gamma_i$, the split of energy changes the total internal energies; the two conditions also result in different densities, total pressures, and shock speeds.
\subsection{One-dimensional accuracy test-I}
\label{test:liu_2025}
To verify the accuracy and convergence rate of the proposed numerical schemes, we consider a smooth problem defined in~\cite{liu2025entropy}, which is essentially a scalar problem and has only one non-trivial component, namely $\rho$. The initial conditions are,
\begin{align*}
        \rho(x,0) &= 1+0.2 \sin{(x)},\\
	\left(\bu,\pe,\pii\right)& = \left(1,0,2,2\right),
\end{align*}
with $\gamma_e = \gamma_i = 1.4$. We consider the computational domain $[0,2\pi]$ and use periodic boundary conditions. The exact solution for the problem is $\rho(x,t) = 1+0.2 \sin{(x-t)}$. 
\begin{table}[tbhp]
	\footnotesize
	\caption{\textbf{\nameref{test:liu_2025}}: $L_1$ errors and order of accuracy for $\rho$.}
	\begin{center}
		\begin{tabular}{c|c|c|c|c|c|c|}
			\hline Number of cells  & \multicolumn{2}{|c}{$\ote$} & \multicolumn{2}{|c}{$\othe$} & \multicolumn{2}{|c}{$\ofe$}  \\
			\hline   & $L_1$ error  &  Order &  $L_1$ error      & Order & $L_1$ error      & Order \\
			\hline 20 & 1.57E-02 & -- & 6.01E-03 & -- & 1.49E-03 & -- \\
			40 & 6.78E-03 & 1.209554054 & 7.65E-04 & 2.973502425 & 1.22E-04 & 3.615641272 \\
			80 & 2.02E-03 & 1.751322546 & 9.60E-05 & 2.994715944 & 8.80E-06 & 3.792707546 \\
			160 & 5.59E-04 & 1.849905978 & 1.20E-05 & 2.998908566 & 6.07E-07 & 3.85875165 \\
			320 & 1.53E-04 & 1.87285767 & 1.50E-06 & 2.99964837 & 4.11E-08 & 3.882974064 \\
			640 & 4.08E-05 & 1.902504771 & 1.88E-07 & 2.999930068 & 2.73E-09 & 3.910359146 \\
			\hline
		\end{tabular}
	\end{center}
		\label{table1}
\end{table}

In Table\eqref{table1}, we have presented the $L_1$-errors for density using the $\ote$, $\othe$ and $\ofe$ schemes at the final time of $t=1.3$. We observe that all the schemes have achieved the theoretically predicted order of accuracy. 

\subsection{One-dimensional accuracy test-II}\label{test:AWENO_2025}
In this test case from~\cite{zhao2025fifth}, we consider the computational domain of $[0,2],$ with periodic boundary conditions. The initial conditions are given by,
\begin{align*}
	\left(\rho,\bu,\pe,\pii\right)& = \left(1 + 0.1 \sin{\pi x}, 1, 0, 0.6 + 0.1 \sin{\pi x}, 0.4 - 0.1 \sin{\pi x}\right).
\end{align*}
with $\gamma_e = \gamma_i = \frac{5}{3}$. 
\begin{table}[tbhp]
	\footnotesize
	\caption{\textbf{\nameref{test:AWENO_2025}}: $L_1$-errors and order of accuracy for $\pe$ and $\pii$.}
	\label{tab:2}
	\begin{center}
		\begin{tabular}{c|c|c|c|c|c|c|}
			
			\hline
			
			\multicolumn{7}{c}{\textbf{$L_1$-errors and order of accuracy for $\pe$}} \\
			\hline
			
			Number of cells  & \multicolumn{2}{c|}{$\ote$} & \multicolumn{2}{c|}{$\othe$} & \multicolumn{2}{c|}{$\ofe$}  \\
			\hline
			& $L_1$-error  & Order & $L_1$-error & Order & $L_1$-error & Order \\
			\hline
			
			40 & 5.16E-03 & -- & 1.33E-04 & -- & 2.44E-05 & -- \\
            80 & 1.71E-03 & 1.592660617 & 1.66E-05 & 2.997894069 & 1.86E-06 & 3.711973538\\
            160 & 4.75E-04 & 1.848363303 & 2.07E-06 & 2.999790989 & 1.31E-07 & 3.830772838\\
            320 & 1.29E-04 & 1.884148797 & 2.59E-07 & 3.000067743 & 9.05E-09 & 3.852261967\\
            640 & 3.48E-05 & 1.888114719 & 3.24E-08 & 3.000102719 & 6.22E-10 & 3.863408653\\
			\hline
			\multicolumn{7}{c}{\textbf{$L_1$-errors and order of accuracy for $\pii$}} \\
			\hline
			Number of cells  & \multicolumn{2}{c|}{$\ote$} & \multicolumn{2}{c|}{$\othe$} & \multicolumn{2}{c|}{$\ofe$}  \\
			\hline
			& $L_1$-error  & Order & $L_1$-error & Order & $L_1$-error & Order \\
			\hline
			
			40 & 5.13E-03 & -- & 1.41E-04 & -- & 2.46E-05 & -- \\
            80 & 1.72E-03 & 1.574981354 & 1.78E-05 & 2.98947927 & 1.86E-06 & 3.7199136\\
            160 & 4.77E-04 & 1.851188296 & 2.22E-06 & 2.997912249 & 1.32E-07 & 3.824447721\\
            320 & 1.29E-04 & 1.882149615 & 2.78E-07 & 2.999345868 & 9.10E-09 & 3.853476906\\
            640 & 3.49E-05 & 1.88953533 & 3.48E-08 & 2.999805725 & 6.24E-10 & 3.865837857\\
			\hline
		\end{tabular}
	\end{center}
\end{table}

In Table\eqref{tab:2}, we have presented the $L_1$-errors of electron and ion pressures, $\pe$ and $\pii,$ respectively, at the final simulation time $t=1.0$, using different numerical schemes. Again, we observe that all the schemes have achieved the desired order of accuracy.

\subsection{Double rarefaction Riemann problem}
\label{test:double_rarefaction}
In this test case, we consider first one-dimensional Riemann problem from~\cite{cheng2024high}, where solution is consist of two rarefaction waves. We consider the computational domain of $[-5,5]$ with Neumann boundary conditions. The initial states are separated by a discontinuity at $x=0$ and are given as follows:
\[(\rho, \bu, \pe, \pii) = \begin{cases}
	(2, -1, 0, 0.6, 0.4), & \textrm{if } x\leq 0\\
	(2, 1, 0, 0.4, 0.6), & \textrm{otherwise,}
\end{cases}\]
with $\gamma_e = 1.4$ and $\gamma_i = 1.67$. We compute the solutions using $2000$ cells till the final simulation time of $t=2.0$.
\begin{figure}[!htbp]
	\begin{center}
		\subfigure[Density $\rho$]{\includegraphics[width=0.45\textwidth]{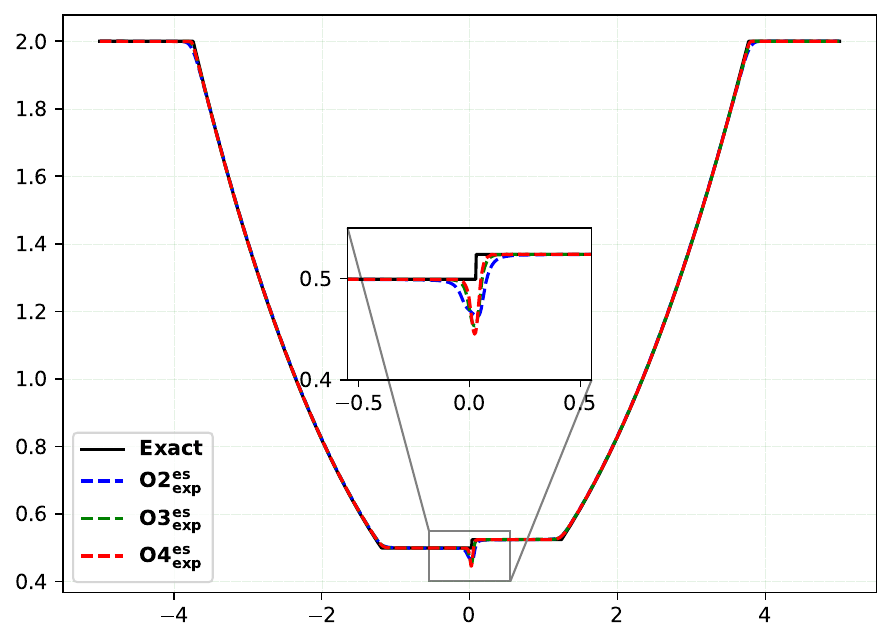} \label{fig:dr_rho}}~
		\subfigure[Electron Pressure $\pe$]{\includegraphics[width=0.45\textwidth]{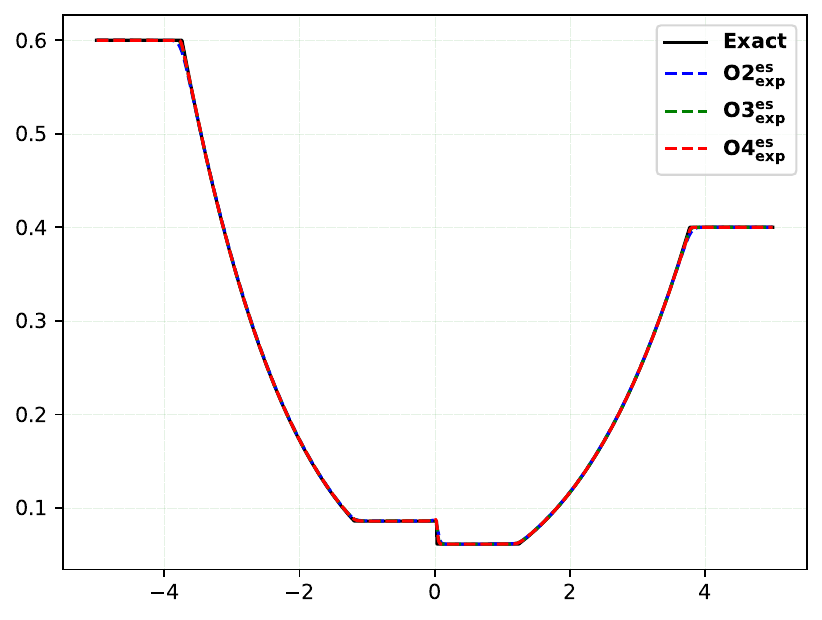}\label{fig:dr_pe}}\\
		\subfigure[Ion Pressure $\pii$]{\includegraphics[width=0.45\textwidth]{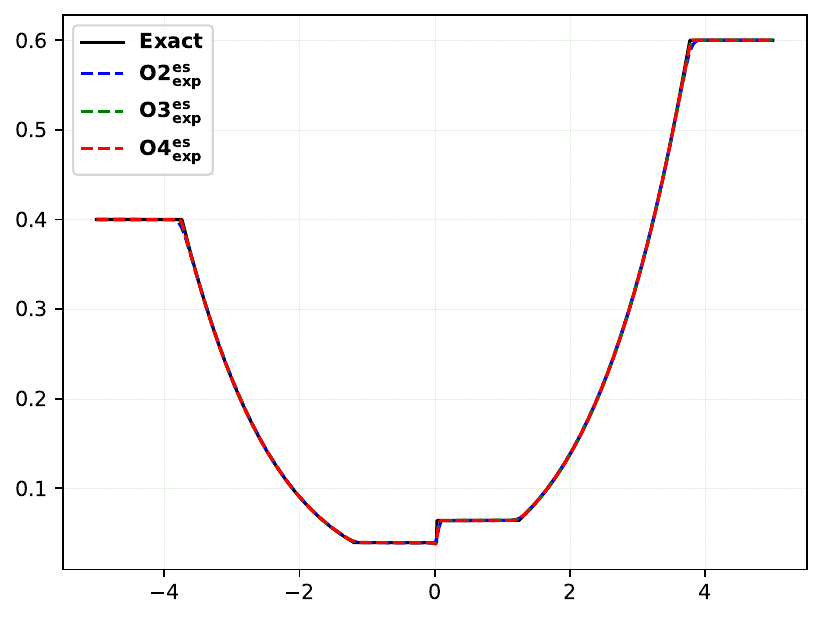}\label{fig:dr_pi}}
		\subfigure[Total entropy change with time]{\includegraphics[width=0.45\textwidth]{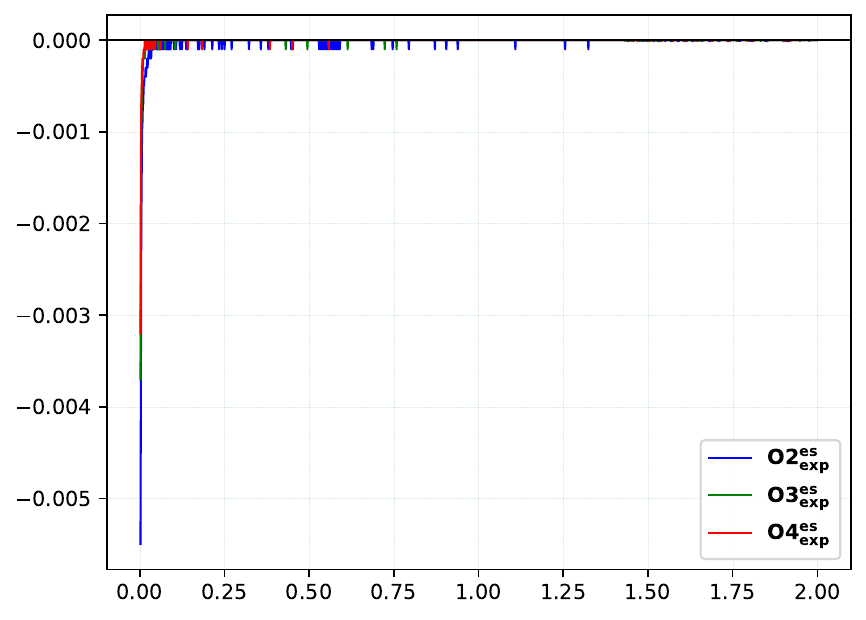}\label{fig:dr_ent}}
		\caption{\textbf{\nameref{test:double_rarefaction}}: Plots of density, electron pressure, and ion pressure at $t = 2.0$ using $2000$ cells. We have also plotted total entropy change with time.}
		\label{fig:dr}
	\end{center}    
\end{figure}

Numerical results for the $\ote$, $\othe$ and $\ofe$ schemes are presented in Figure~\eqref{fig:dr}. We have plotted density, electron, and ion pressures and total entropy decay at each time step. We observe that all the waves are resolved when compared with the exact solutions. Also, we note that the second-order scheme $\ote$ is more diffusive than the third ($\othe$) and fourth ($\ofe$) order schemes. The numerical results match the exact solution and are similar to those presented in~\cite{cheng2024high}. In Figure~\eqref{fig:dr_ent}, we have plotted the total entropy change (see ~\eqref{eq:tot_ent_exp}) at every time step. We observe that due to the presence of an initial discontinuity, we have large entropy decay initially. However, after that, the entropy decay is almost negligible as no new waves are generated.
\subsection{Sod Riemann problem}\label{test:sod}
In this test from~\cite{cheng2024high}, we again consider a Riemann problem on the computational domain of $[-5,5]$, with Neumann boundary conditions. The initial conditions are given by,
\[(\rho, \bu, \pe, \pii) = \begin{cases}
	(1, 0, 0, 0.4, 0.6), & \textrm{if } x\leq 0\\
	(0.125, 0, 0, 0.06, 0.04), & \textrm{otherwise}
\end{cases}\]
with $\gamma_e = \gamma_i = 1.4$. We use $2000$ cells and compute till the final time $t=2.0.$
\begin{figure}[!htbp]
	\begin{center}
		\subfigure[Density ($\rho$)]{\includegraphics[width=0.5\textwidth]{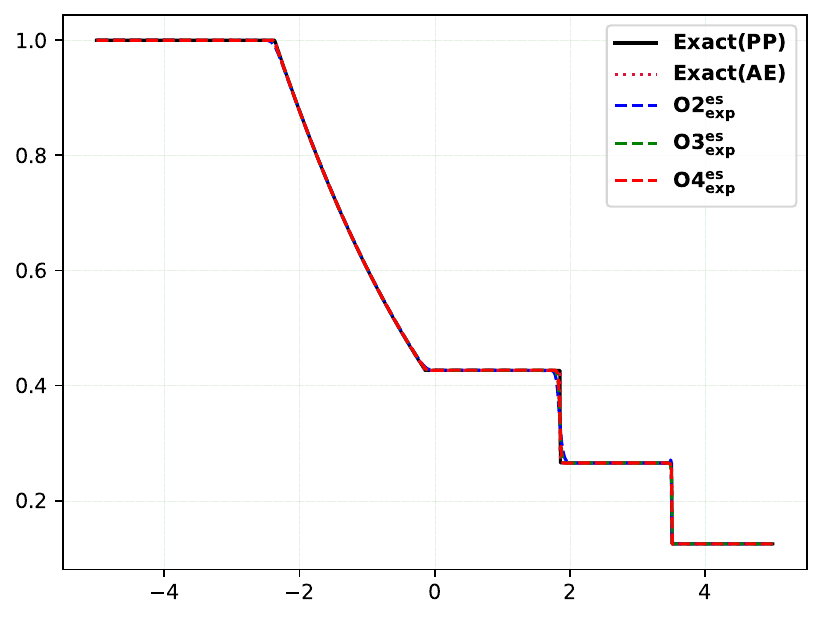} \label{fig:sod_rho}}~
        \subfigure[Total entropy change with time]{\includegraphics[width=0.5\textwidth]{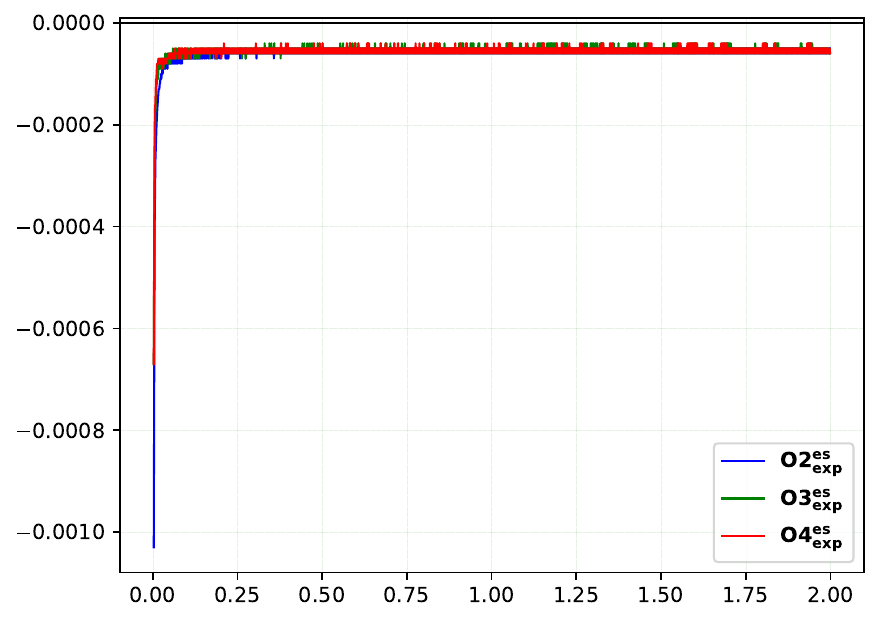}\label{fig:sod_ent}}\\
		\subfigure[Electron pressure ($\pe$)]{\includegraphics[width=1.0\textwidth]{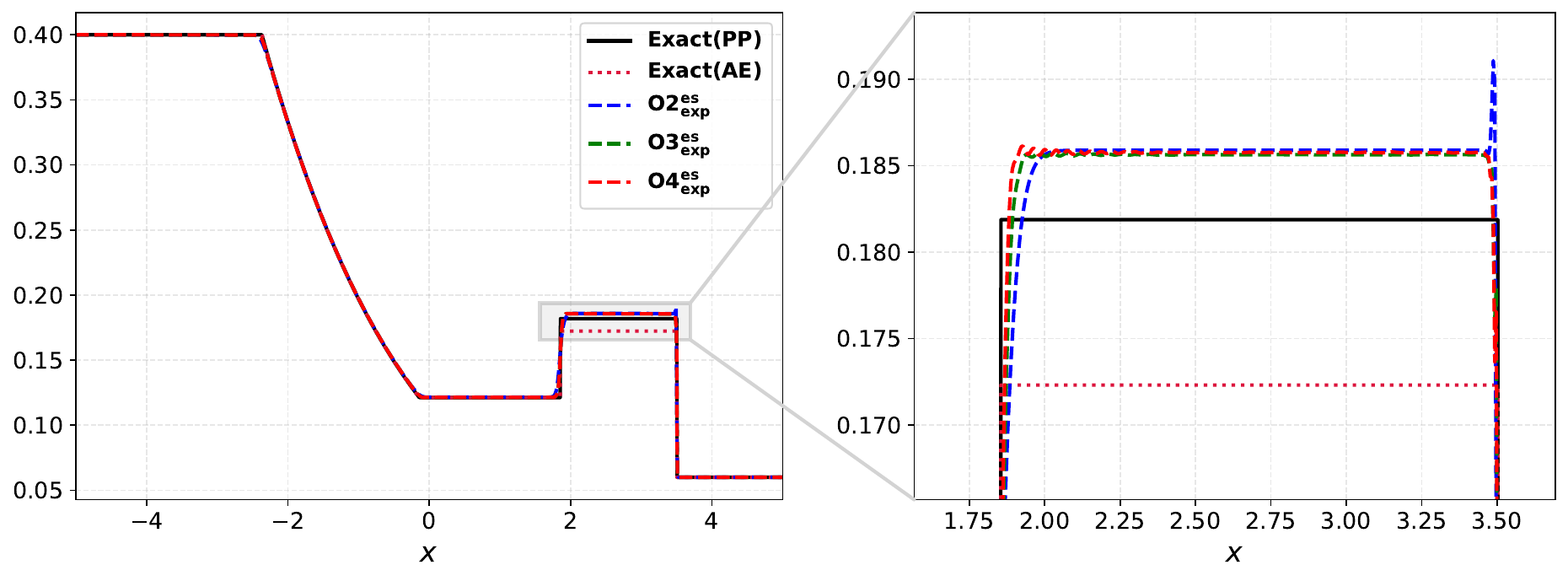}\label{fig:sod_pe}}\\
		\subfigure[Ion pressure ($\pii$)]{\includegraphics[width=1.0\textwidth]{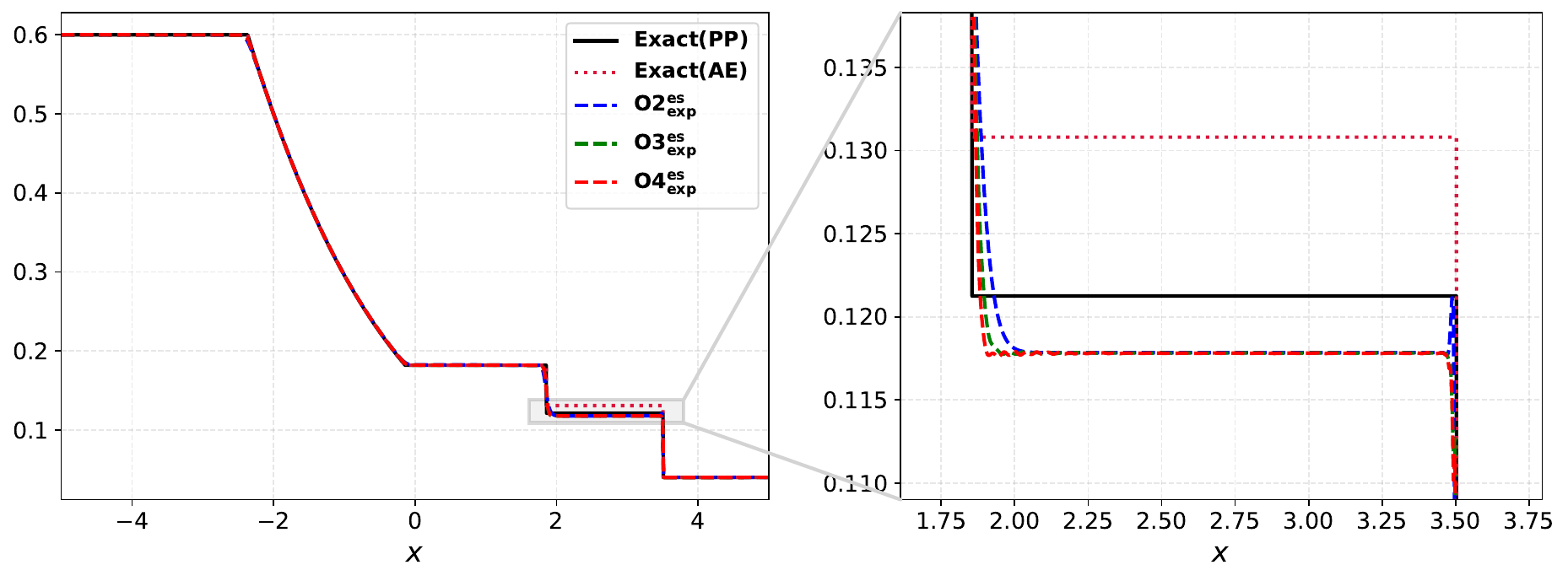}\label{fig:sod_pi}}
		\caption{\textbf{\nameref{test:sod}}: Plots of density, electron pressure, and ion pressure at $t = 2.0$ using $2000$ cells. We have also plotted total entropy change with time.}
		\label{fig:sod}
	\end{center}    
\end{figure}

We have presented numerical results in Fig.\eqref{fig:sod} for the $\ote$, $\othe$ and $\ofe$ schemes at time $t=2.0$ using $2000$ cells. We have plotted density, electron pressure, ion pressure and total entropy change with time. We observe that all the schemes are able to resolve the shocks, contact and rarefaction waves efficiently. The density matches the exact solution, which is the same for the PP and AE conditions. For the electron and ion pressures, the exact solutions differ between the contact and the shock. In this region, the numerical electron pressure is slightly higher and the ion pressure slightly lower than both exact solutions. A similar behavior is observed in~\cite{cheng2024high}, where the numerical electron and ion pressures also differ from the exact solutions behind the shock. We do see small oscillations in the second-order scheme $\ote$ but they are stable in the sense that they do not increase with refinement. Furthermore, higher-order schemes $\othe$ and $\ofe$ do not have any oscillations. In addition, from the entropy decay plot we see that all the schemes are entropy stable and continue to have some entropy decay due to the presence of a shock in the solution. We also observe that $\ote$ is slightly more diffusive than the $\othe$ and $\ofe$ schemes.

\subsection{Lax Riemann problem}
\label{test:lax}
In another test case from ~\cite{cheng2024high}, we consider a Riemann problem on the computational domain of $[-5,5]$ with Neumann boundary conditions. The initial conditions are given by,
\[(\rho, \bu, \pe, \pii) = \begin{cases}
	(0.445, 0.689, 0, 1.764, 1.764), & \textrm{if } x\leq 0,\\
	(0.5, 0, 0, 0.2855, 0.2855), & \textrm{otherwise},
\end{cases}\]
with $\gamma_e = 1.4$ and $\gamma_i = 1.67$.  
\begin{figure}[!htbp]
	\begin{center}
		\subfigure[Density ($\rho$)]{\includegraphics[width=1.0\textwidth]{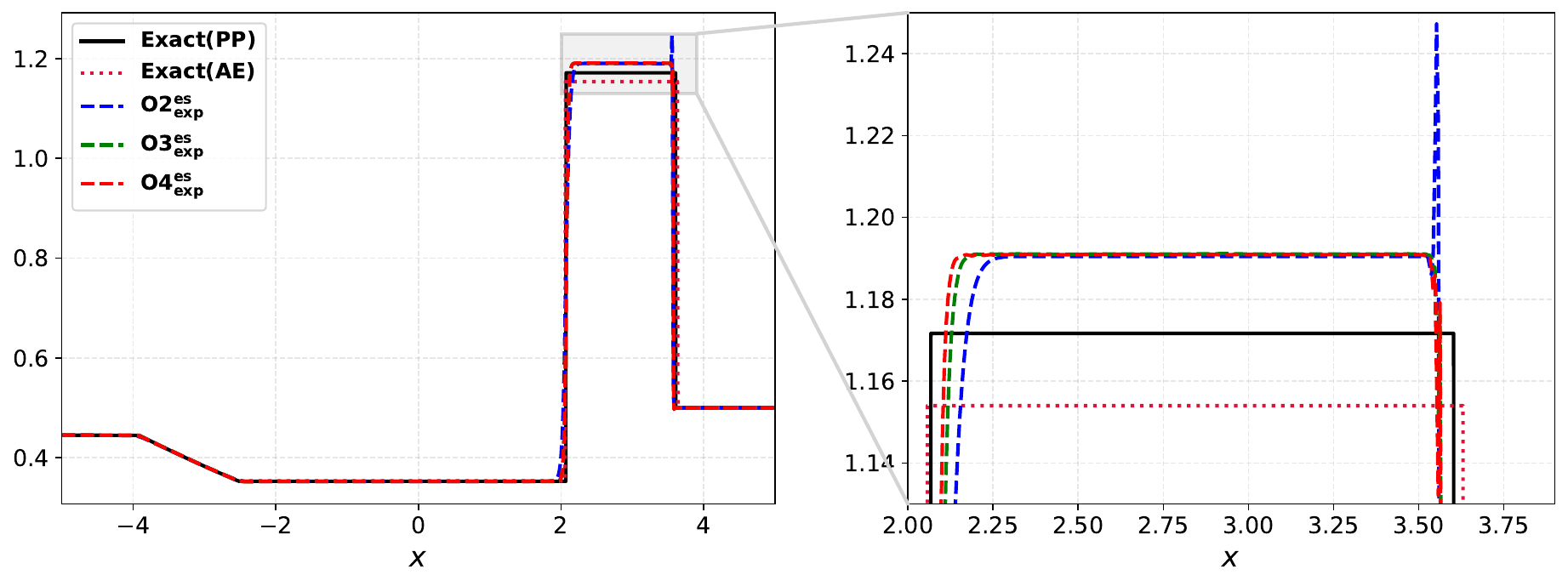} \label{fig:lax_rho}}\\
		\subfigure[Electron pressure ($\pe$)]{\includegraphics[width=1.0\textwidth]{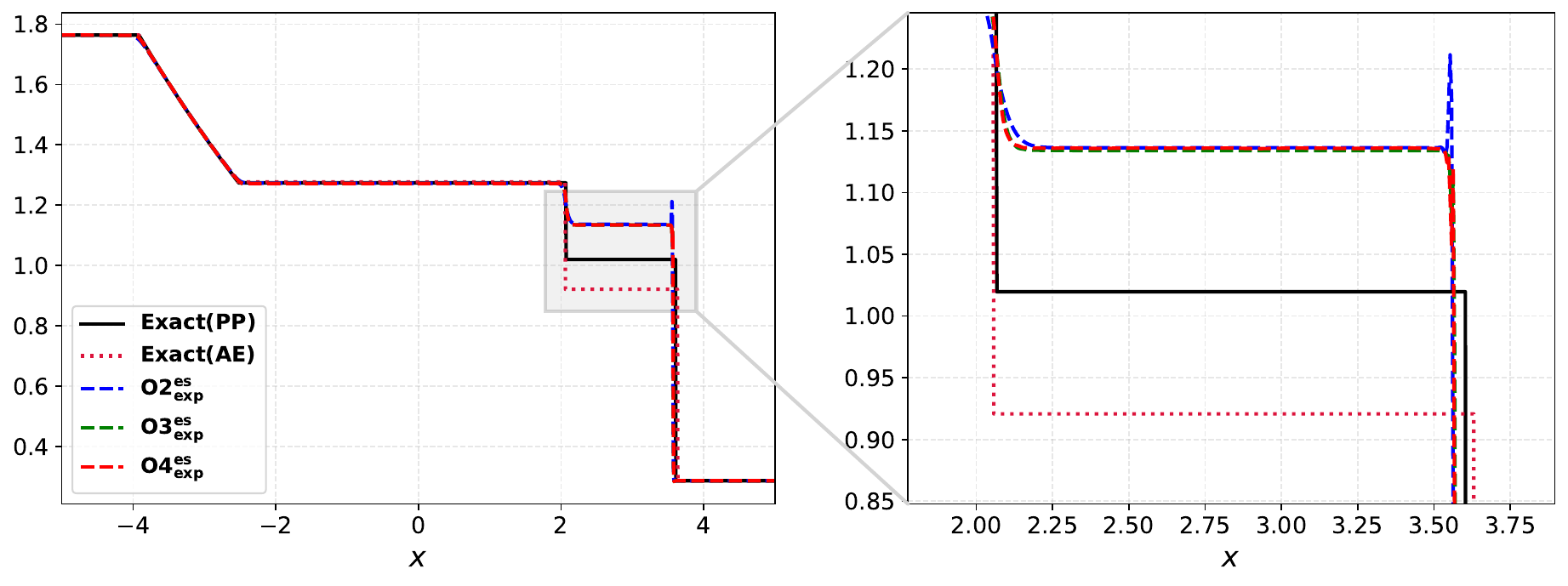}\label{fig:lax_pe}}\\
		\subfigure[Ion pressure ($\pii$)]{\includegraphics[width=1.0\textwidth]{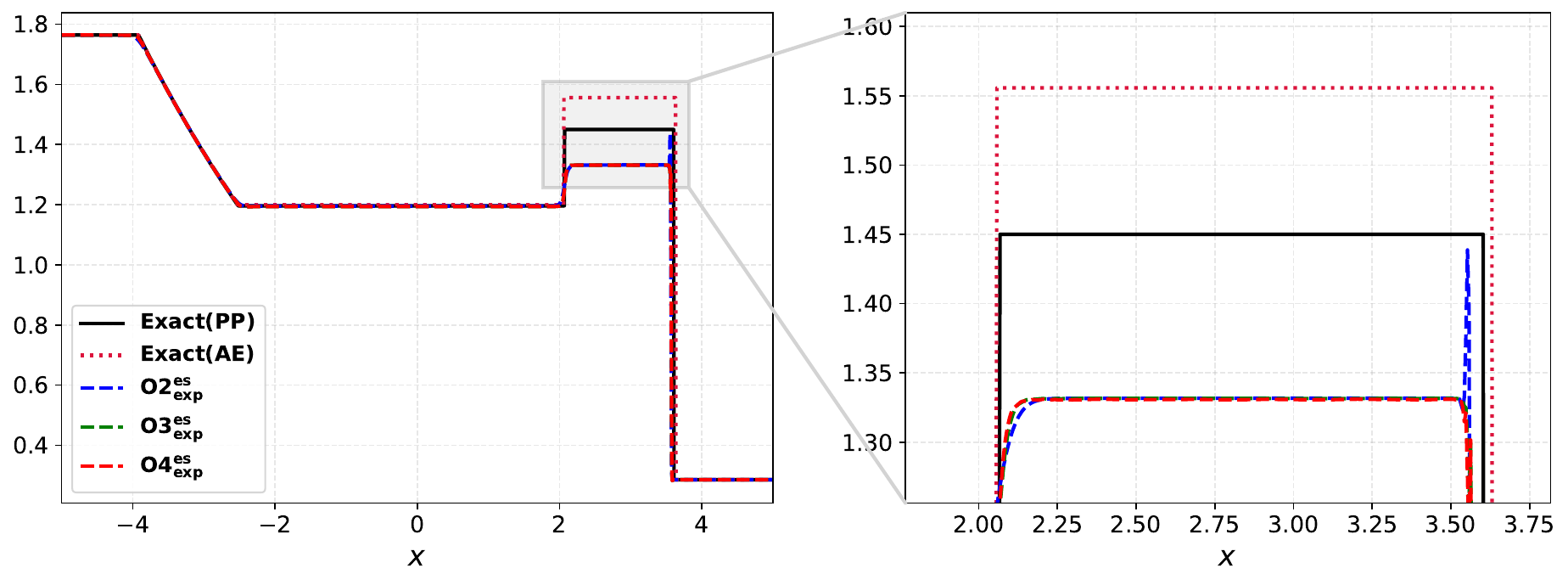}\label{fig:lax_pi}}
		\caption{\textbf{\nameref{test:lax}}: Plots of density, electron pressure, and ion pressure at $t = 1.4$ using $2000$ cells.}
		\label{fig:lax}
    \end{center}    
\end{figure}
 \begin{figure}[!htbp]
	 	\begin{center}
	 		\includegraphics[width=0.5\textwidth]{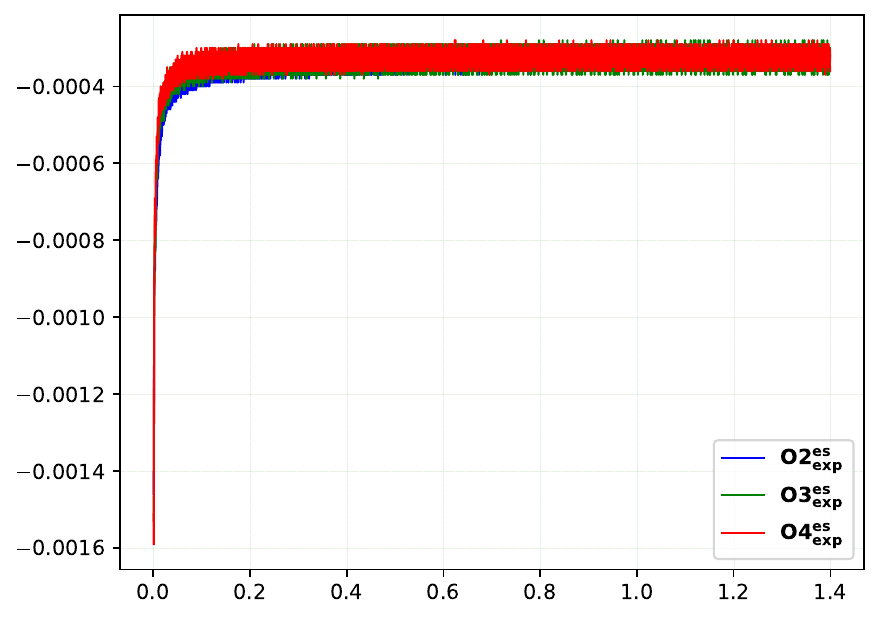}
	 		\caption{\textbf{\nameref{test:lax}}: Plots of total entropy change with time using $2000$ cells for $\ote$, $\othe$ and $\ofe$ schemes.}
	 		\label{fig:lax_ent}
	 	\end{center}    
	 \end{figure}
We have presented numerical results in Fig.\eqref{fig:lax} for the $\ote$, $\othe$, and $\ofe$ schemes at the final time of $t=1.4$ on $2000$ cells. The density, electron pressure, ion pressure, and change in total entropy with time are plotted in Figures~\eqref{fig:lax_rho},~\eqref{fig:lax_pe}, \eqref{fig:lax_pi} and ~\eqref{fig:lax_ent}, respectively. The PP and AE exact solutions differ in the post-shock density, electron pressure and ion pressure, and also in the shock position. In this region, the numerical density and electron pressure are higher, and the ion pressure is lower than both exact solutions. The numerical shock is also slightly slower. A similar behavior is observed in~\cite{cheng2024high}, where the numerical solutions also differ from the exact solutions behind the shock. We again see that all the waves are resolved and $\ote$ being most diffusive. We also observe small oscillations, $\ote$, but they are stable with respect to refinement. We do not observe any oscillations in $\othe$ and $\ofe$ schemes. In Figure~\eqref{fig:lax_ent}, we note that total entropy decays at every time step. We also see that total entropy decay is almost constant after initial changes. Furthermore, $\ote$ decays more entropy when compared with $\othe$ and $\ofe$.
\subsection{Two-dimensional accuracy test}\label{test:liu_2025_2D}
To verify the accuracy and convergence rate of our numerical schemes, we consider a smooth problem defined in~\cite{liu2025entropy}, which is essentially a scalar problem and has only one non-trivial component $\rho$. We convert this problem into a two-dimensional test problem. The problem is initialized as follows:
\begin{align*}
        \rho(x,y,0) &= 1+0.2 \sin{(x+y)},\\
	\left(\bu,\pe,\pii\right)& = \left(0.5,0.5,2,2\right).
\end{align*}
with $\gamma_e = \gamma_i = 1.4$. We use periodic boundary conditions on the domain $(x,y)\in[0,2\pi]\times[0,2\pi]$. For this setup, the exact solution is $\rho(x,y,t) = 1+0.2 \sin{(x+y-t)}$. Table\eqref{tab:3} presents the $L_1$-errors of density for the $\ote$, $\othe$ and $\ofe$ schemes at time $t=1.3$. We observe that each scheme achieves the expected order of accuracy. 

\begin{table}[tbhp]
	\footnotesize
	\caption{\textbf{\nameref{test:liu_2025_2D}}: $L_1$ errors and order of accuracy for $\rho$.}
	\label{tab:3}
	\begin{center}
		\begin{tabular}{c|c|c|c|c|c|c|}
			\hline Number of cells  & \multicolumn{2}{|c}{$\ote$} & \multicolumn{2}{|c}{$\othe$} & \multicolumn{2}{|c}{$\ofe$}  \\
			\hline   & $L_1$ error  &  Order &  $L_1$ error      & Order & $L_1$ error      & Order \\
			\hline 24 $\times$ 24 & 1.26E-02 & -- & 7.00E-04 & -- & 1.60E-04 & -- \\
			48 $\times$ 48 & 5.03E-03 & 1.318732756 & 8.86E-05 & 2.982213119 & 1.22E-05 & 3.714219366 \\
			96 $\times$ 96 & 1.45E-03 & 1.791429385 & 1.11E-05 & 2.996847801 & 8.72E-07 & 3.801274782 \\
			192 $\times$ 192 & 3.97E-04 & 1.873327563 & 1.39E-06 & 2.998557202 & 6.01E-08 & 3.860152037 \\
			384$\times$ 384 & 1.08E-04 & 1.874857877 & 1.74E-07 & 2.999868426 & 4.02E-09 & 3.900473544 \\
			\hline
		\end{tabular}
	\end{center}
\end{table} 
\subsection{Two-dimensional Riemann problem}\label{test:RP_2D_1} In this test case, we consider a two-dimensional Riemann problem from ~\cite{cheng2024high}, which is motivated from the similar problem for the Euler equation of compressible flow in~\cite{lax1998solution,pan2017few}. We consider a computational domain of $[0,1]\times[0,1]$ with Neumann boundary conditions. The initial conditions are given by,
\[\left(\rho, \bu, \pe, \pii\right)  = \begin{cases}
	(1.5, 0, 0, 0.75, 0.75), & \textrm{if } x>0.8,~y>0.8\\
	(0.5323, 1.206, 0, 0.15, 0.15), & \textrm{if } x<0.8,~y>0.8\\
	(0.138, 1.206, 1.206, 0.0145, 0.0145), & \textrm{if } x<0.8,~y<0.8\\
	(0.5323, 0, 1.206, 0.15, 0.15), & \textrm{if } x>0.8,~y<0.8\\
\end{cases}\]
Similar to~\cite{cheng2024high}, we consider two cases, with different values of gas constants. In the first case, we take $\gamma_e = \gamma_i = 1.4$, and in second case, we take $\gamma_e = 1.4,$ $\gamma_i = 1.67$.  Due to the highly nonlinear interaction of different waves, the solutions contain highly complicated multi-scale structures. We compute the results using $400\times400$ cells.

\begin{figure}[!htbp]
	\begin{center}
		\subfigure[Density ($\rho$) for $\ote$ scheme]{\includegraphics[width=0.3\textwidth]{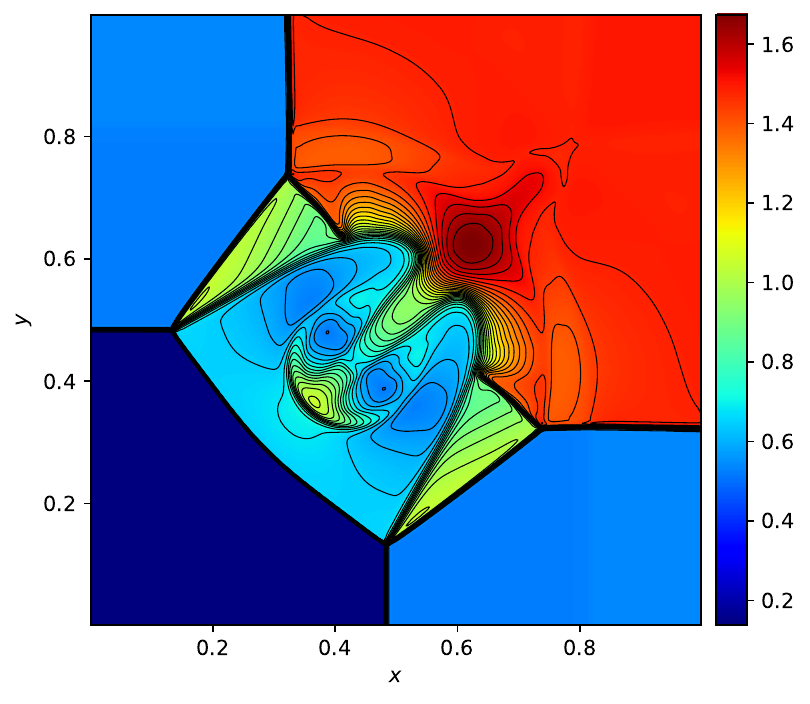}\label{fig:RP_sg_rho_o2}}~
		\subfigure[Density ($\rho$) for $\othe$ scheme]{\includegraphics[width=0.3\textwidth]{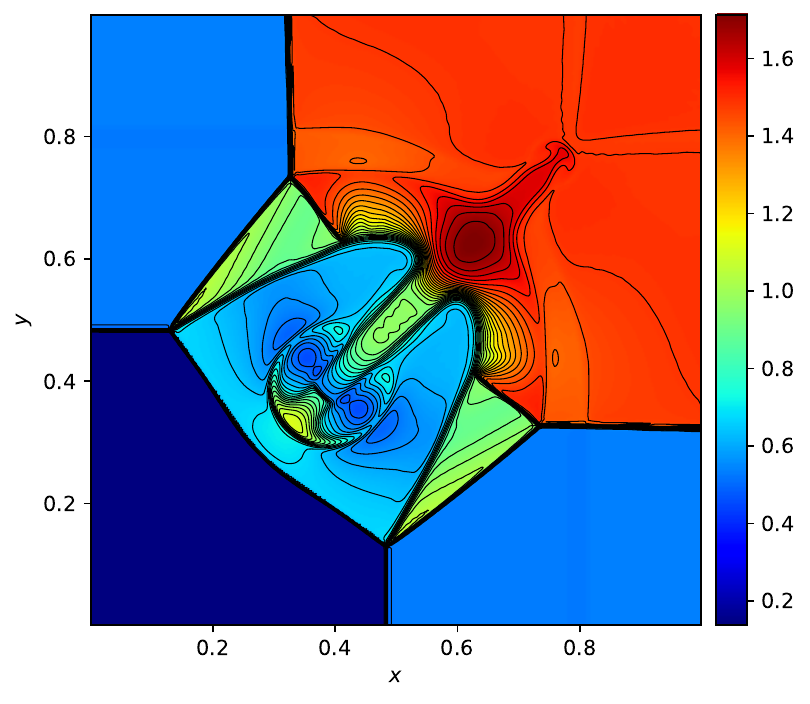}\label{fig:RP_sg_rho_o3}}~
		\subfigure[Density ($\rho$) for $\ofe$ scheme]{\includegraphics[width=0.3\textwidth]{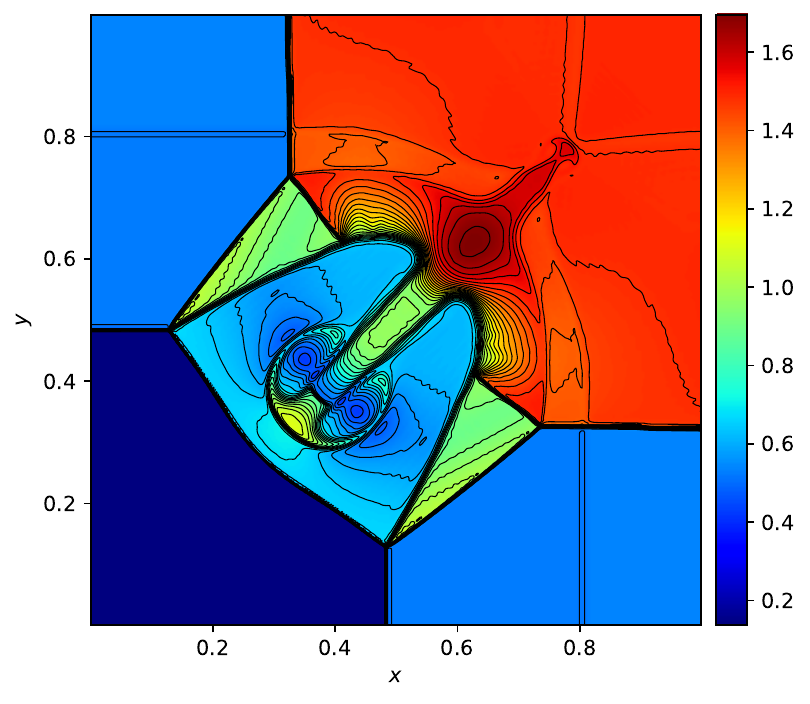}\label{fig:RP_sg_rho_o4}}\\
		\subfigure[Electron pressure ($\pe$) for $\ote$ scheme]{\includegraphics[width=0.3\textwidth]{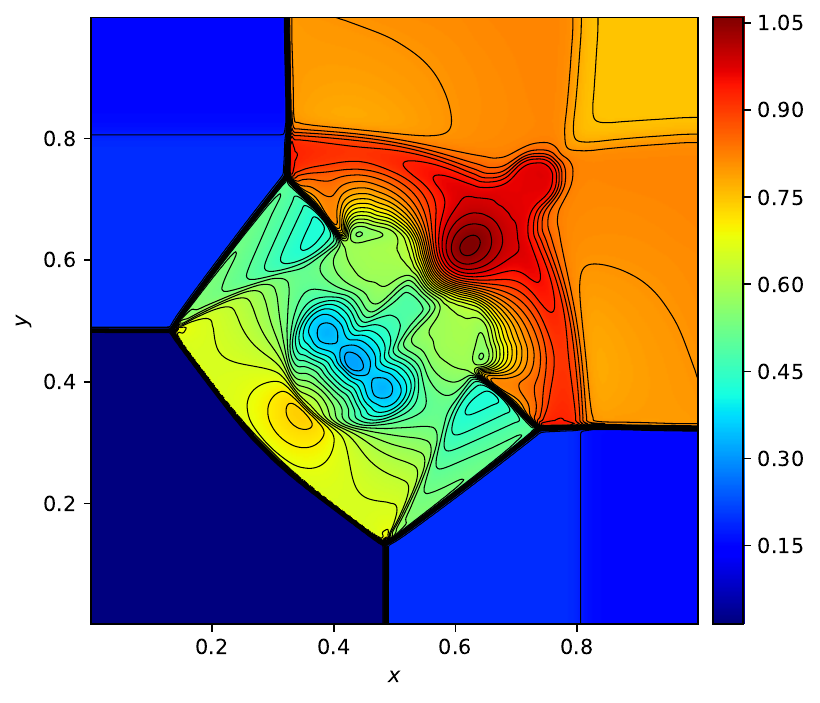}\label{fig:RP_sg_pe_o2}}~
		\subfigure[Electron pressure ($\pe$) for $\othe$ scheme]{\includegraphics[width=0.3\textwidth]{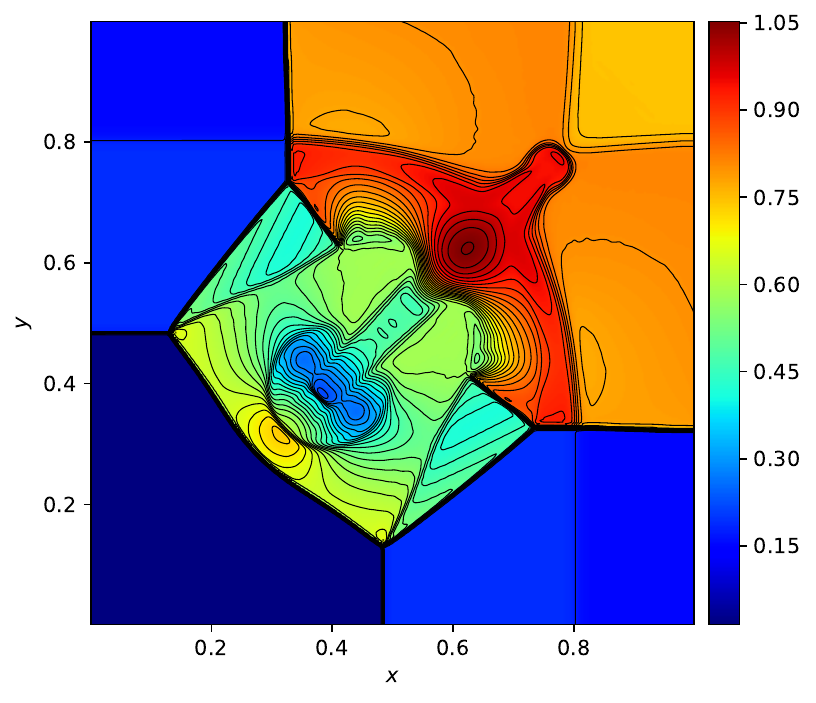}\label{fig:RP_sg_pe_o3}}~
		\subfigure[Electron pressure ($\pe$) for $\ofe$ scheme]{\includegraphics[width=0.3\textwidth]{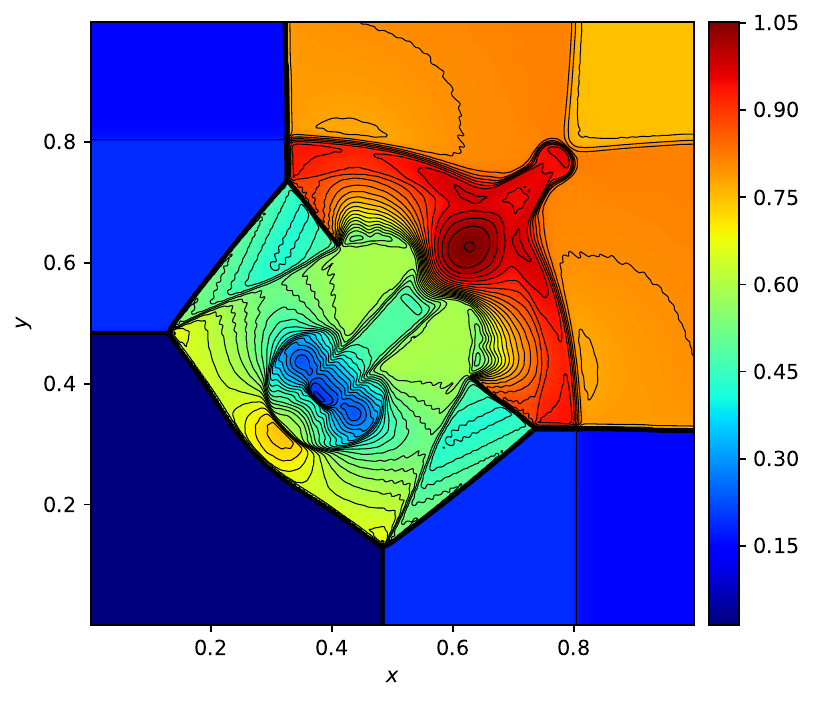}\label{fig:RP_sg_pe_o4}}\\
		\subfigure[Ion pressure ($\pii$) for $\ote$ scheme]{\includegraphics[width=0.3\textwidth]{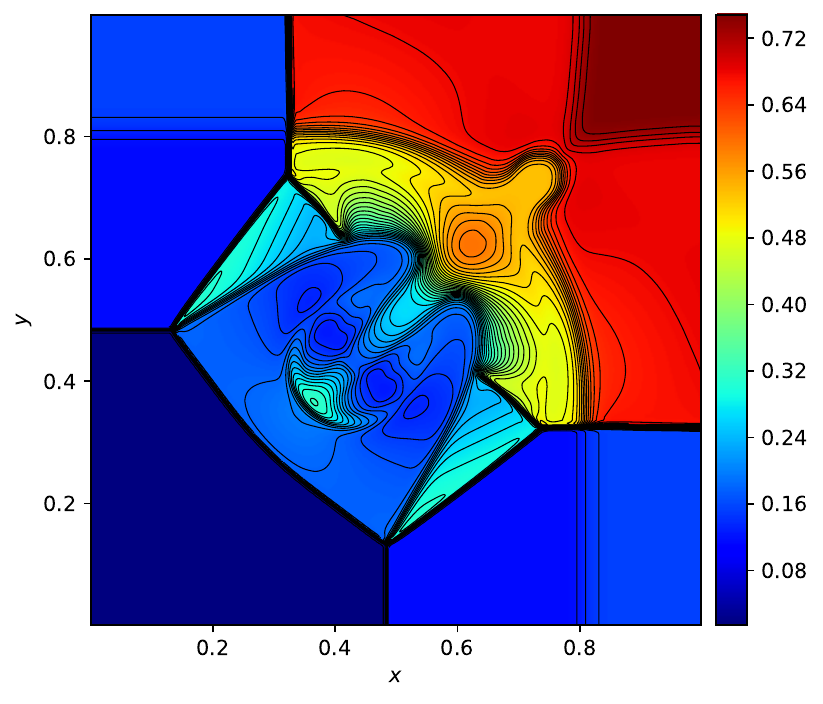}\label{fig:RP_sg_pi_o2}}~
		\subfigure[Ion pressure ($\pii$) for $\othe$ scheme]{\includegraphics[width=0.3\textwidth]{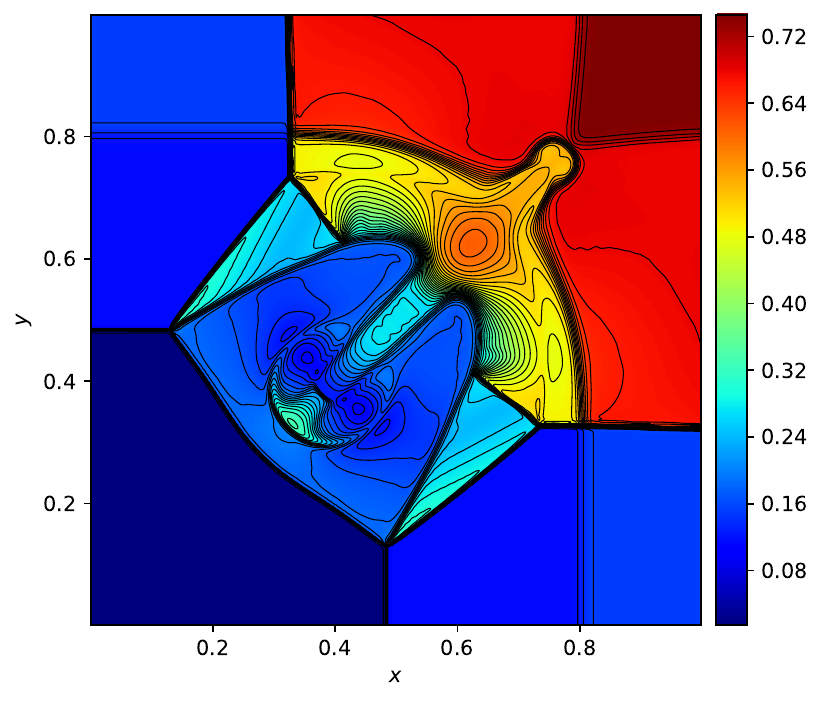}\label{fig:RP_sg_pi_o3}}~
		\subfigure[Ion pressure ($\pii$) for $\ofe$ scheme]{\includegraphics[width=0.3\textwidth]{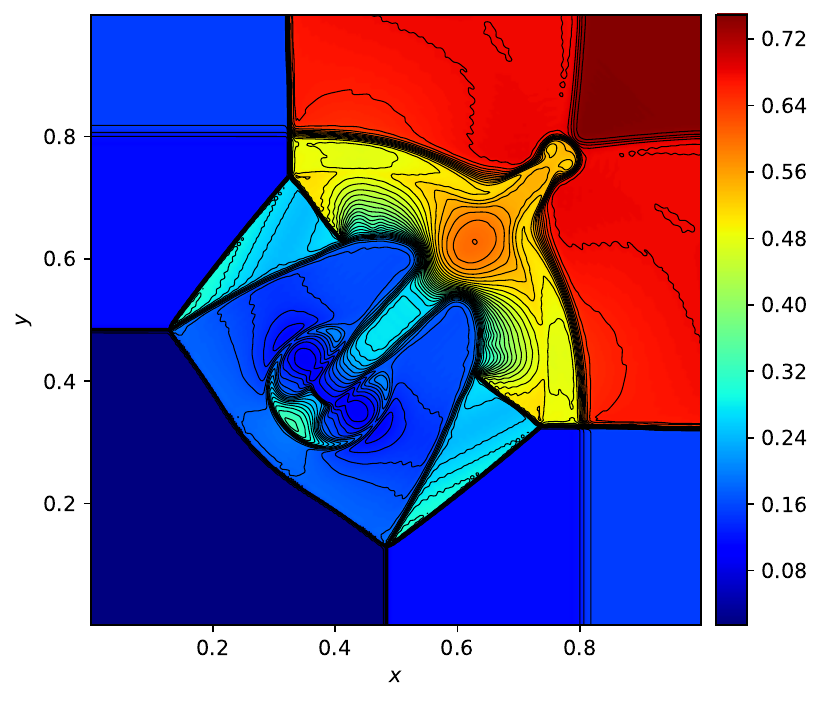}\label{fig:RP_sg_pi_o4}}\\
		\caption{\textbf{\nameref{test:RP_2D_1}}: Plots of density, electron pressure, and ion pressure at time $t = 0.75$ using $400\times400$ cells for $\gamma_e = \gamma_i = 1.4$.}
		\label{fig:RP_sg}
	\end{center}
\end{figure}

\begin{figure}[!htbp]
	\begin{center}
		\subfigure[Density ($\rho$) for $\ote$ scheme]{\includegraphics[width=0.3\textwidth]{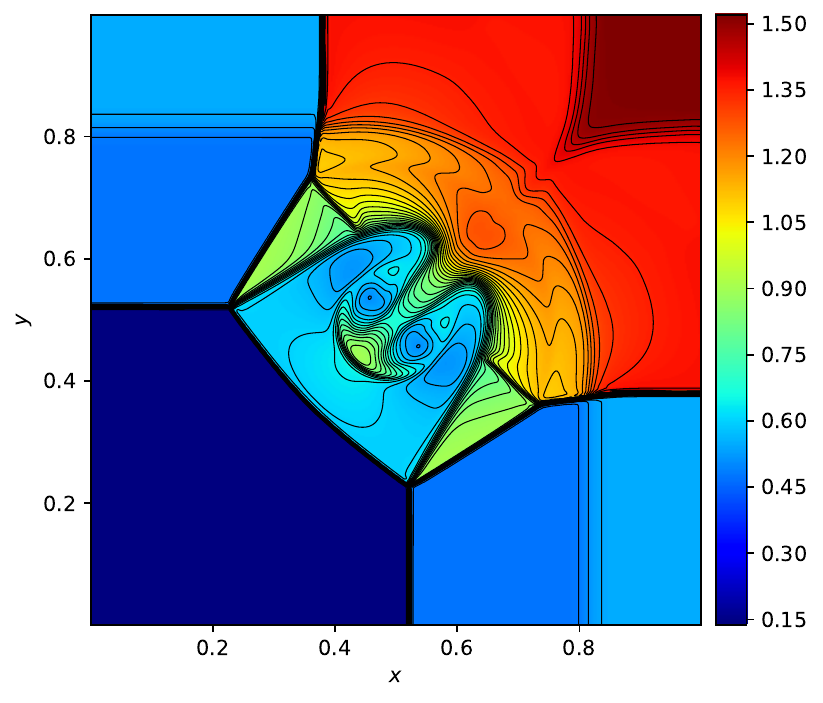}\label{fig:RP_dg_rho_o2}}~
		\subfigure[Density ($\rho$) for $\othe$ scheme]{\includegraphics[width=0.3\textwidth]{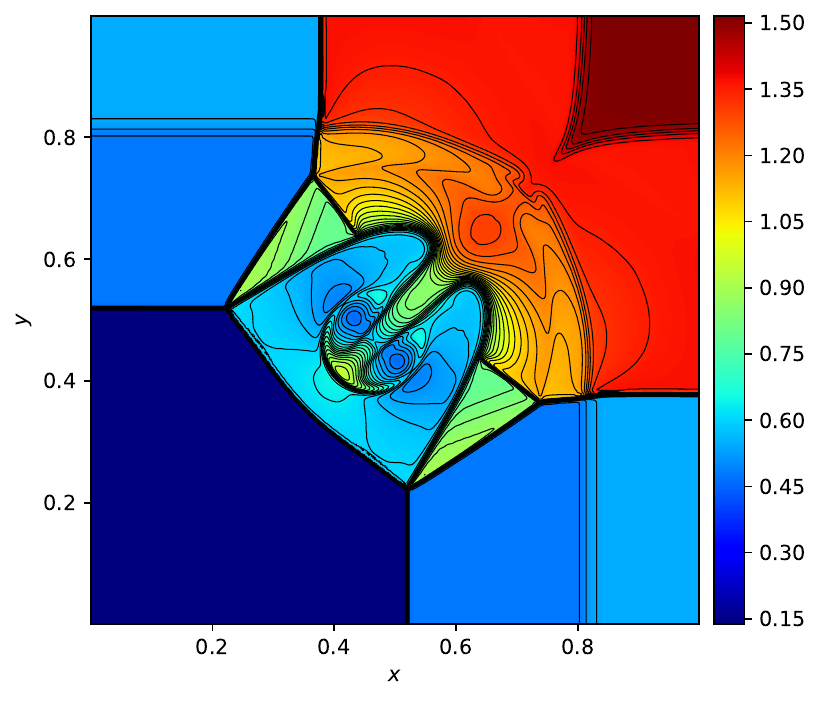}\label{fig:RP_dg_rho_o3}}~
		\subfigure[Density ($\rho$) for $\ofe$ scheme]{\includegraphics[width=0.3\textwidth]{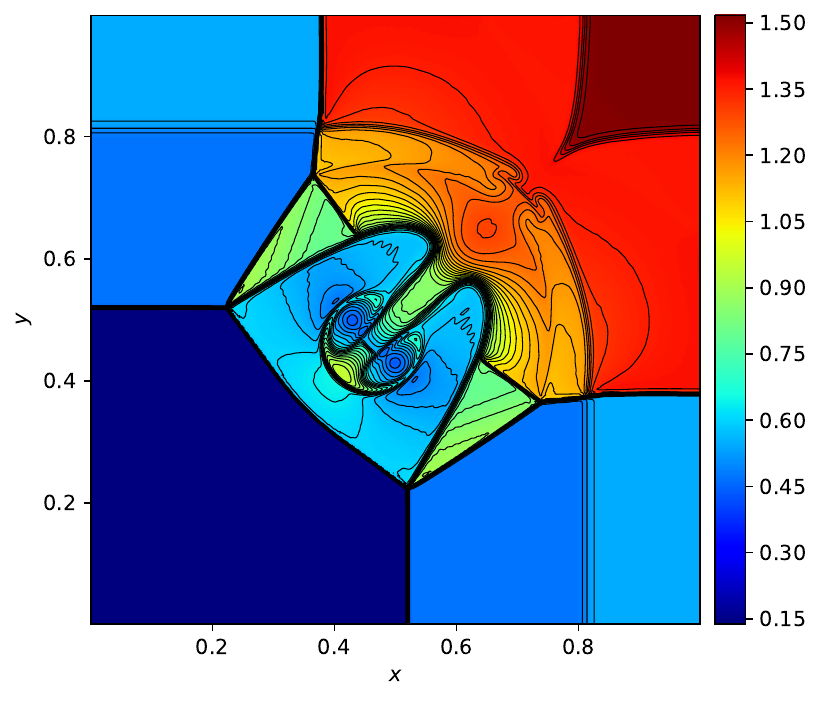}\label{fig:RP_dg_rho_o4}}\\
		\subfigure[Electron pressure ($\pe$) for $\ote$ scheme]{\includegraphics[width=0.3\textwidth]{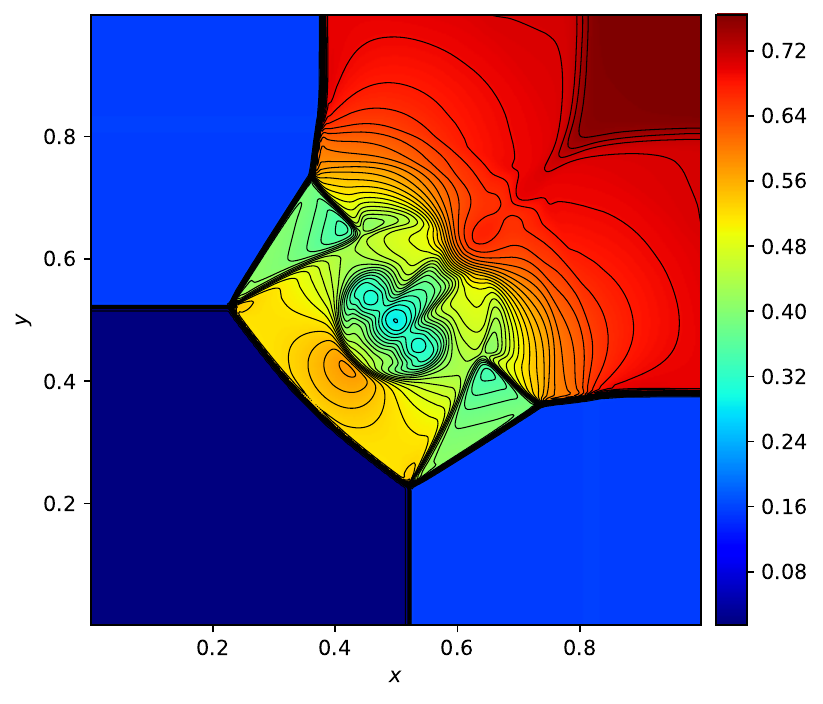}\label{fig:RP_dg_pe_o2}}~
		\subfigure[Electron pressure ($\pe$) for $\othe$ scheme]{\includegraphics[width=0.3\textwidth]{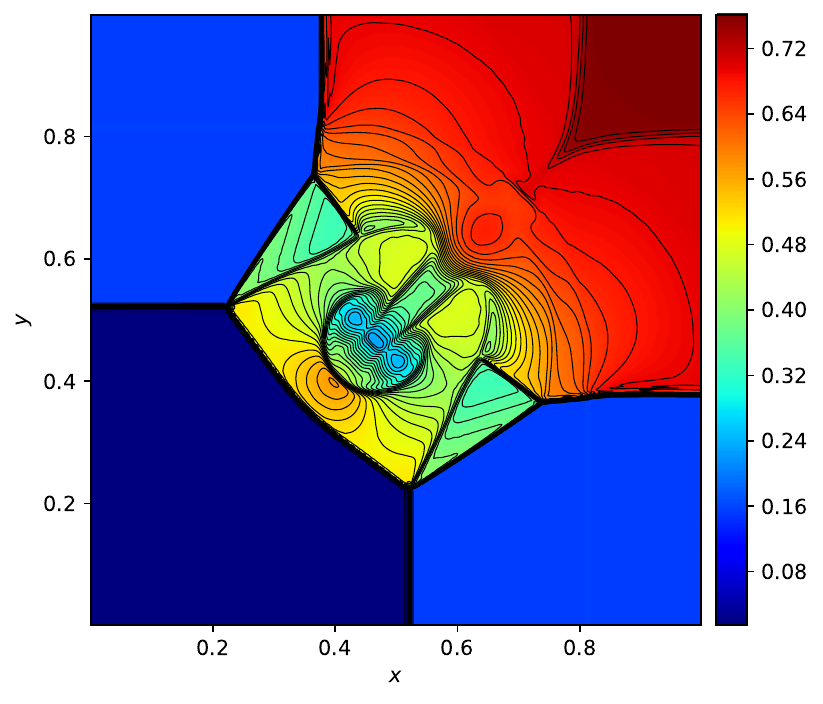}\label{fig:RP_dg_pe_o3}}~
		\subfigure[Electron pressure ($\pe$) for $\ofe$ scheme]{\includegraphics[width=0.3\textwidth]{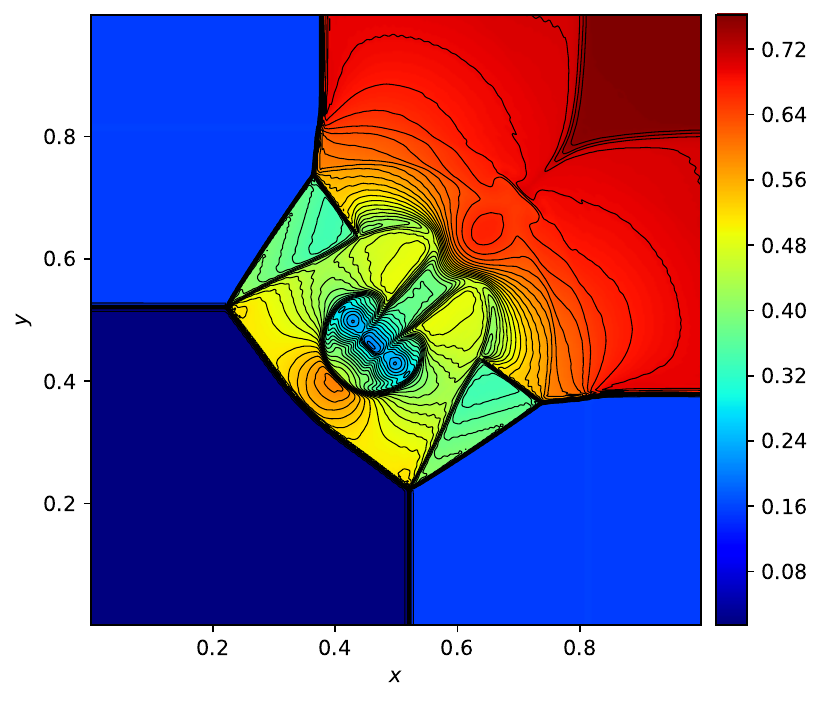}\label{fig:RP_dg_pe_o4}}\\
		\subfigure[Ion pressure ($\pii$) for $\ote$ scheme]{\includegraphics[width=0.3\textwidth]{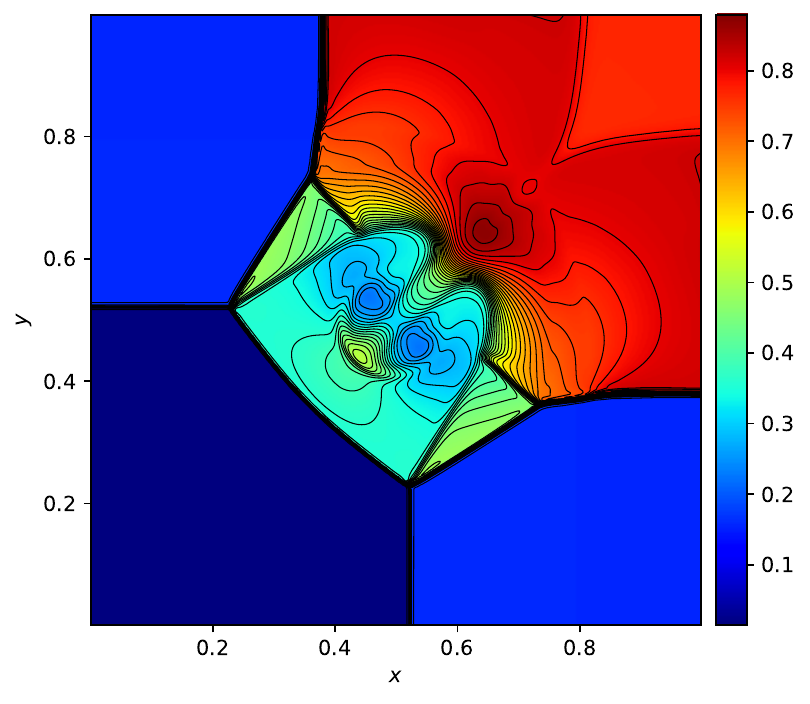}\label{fig:RP_dg_pi_o2}}~
		\subfigure[Ion pressure ($\pii$) for $\othe$ scheme]{\includegraphics[width=0.3\textwidth]{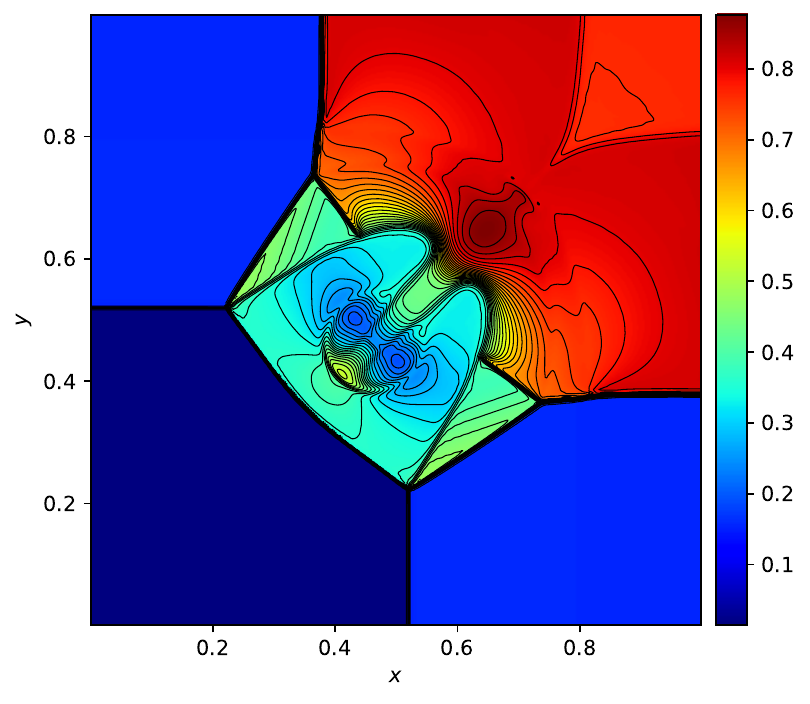}\label{fig:RP_dg_pi_o3}}~
		\subfigure[Ion pressure ($\pii$) for $\ofe$ scheme]{\includegraphics[width=0.3\textwidth]{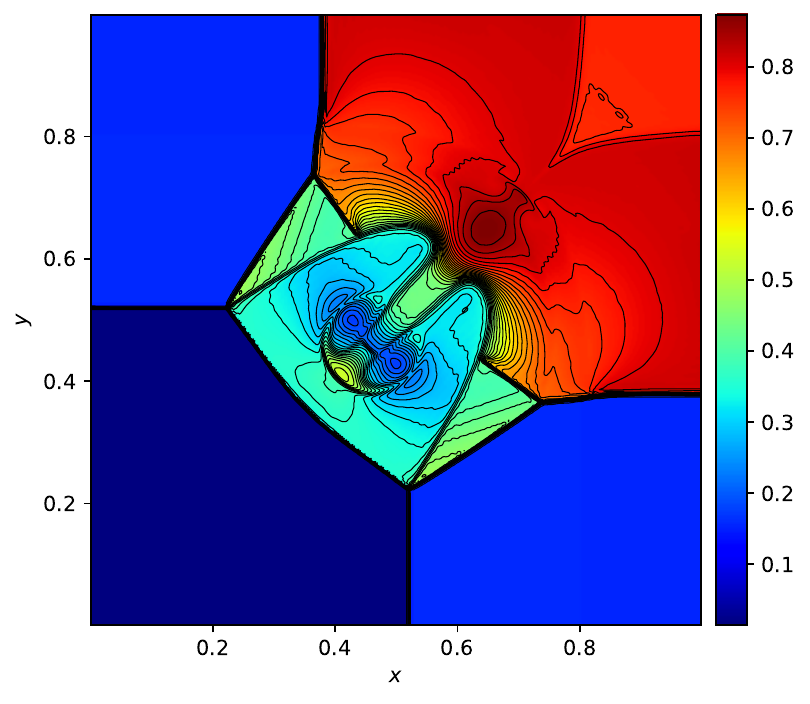}\label{fig:RP_dg_pi_o4}}\\
		\caption{\textbf{\nameref{test:RP_2D_1}}: Plots of density, electron pressure, and ion pressure at time $t = 0.59$ using $400\times400$ cells for $\gamma_e = 1.4,$ $\gamma_i = 1.67$.}
		\label{fig:RP_dg}
	\end{center}
\end{figure}
\begin{figure}[!htbp]
	\begin{center}
		\subfigure[Total entropy change with time for $\gamma_e = 1.4,$ $\gamma_i = 1.4$.]{\includegraphics[width=0.48\textwidth]{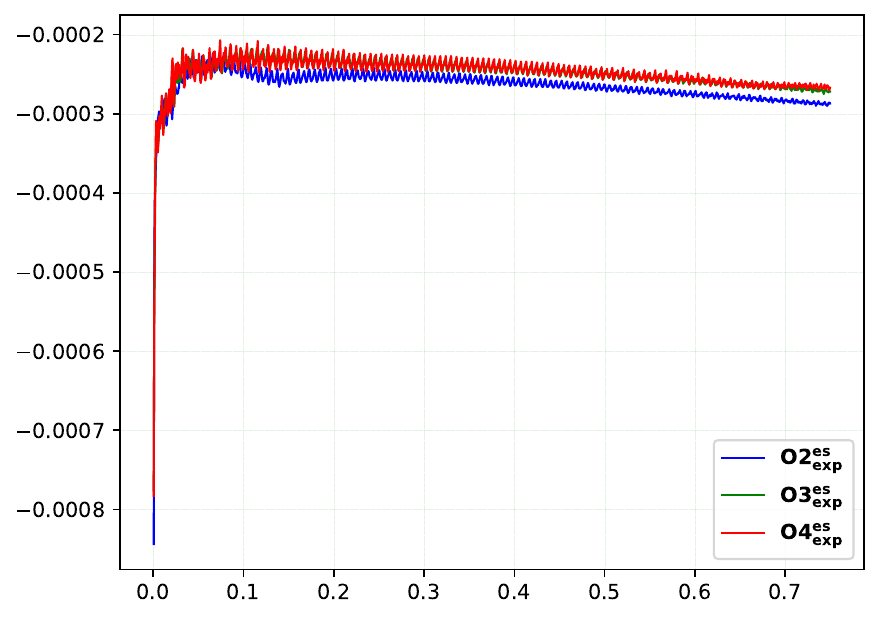} \label{fig:sg_entropy}}~
		\subfigure[Total entropy change with time for $\gamma_e = 1.4,$ $\gamma_i = 1.67$.]{\includegraphics[width=0.48\textwidth]{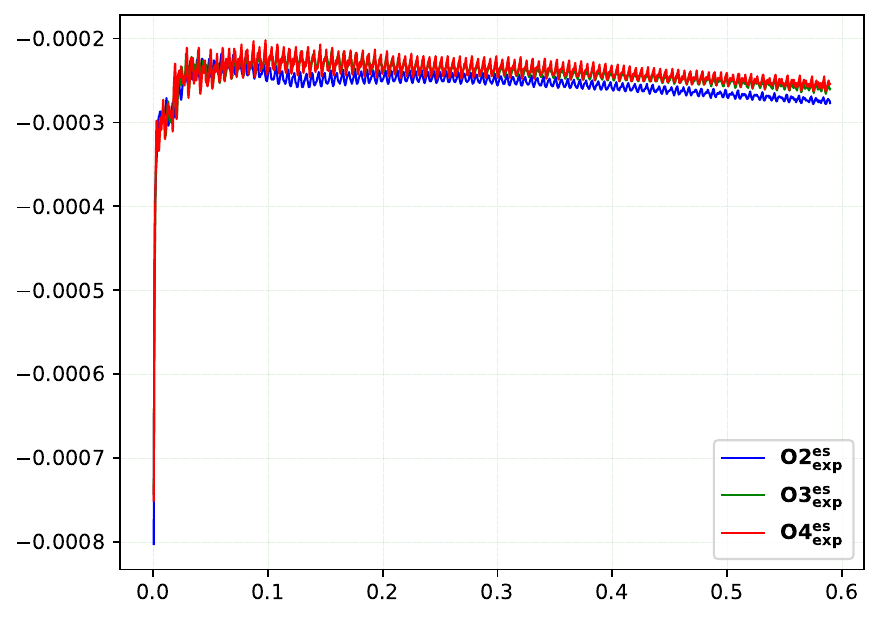} \label{fig:dg_entropy}}
		\caption{\textbf{\nameref{test:RP_2D_1}}: Plots of total entropy change with time using $400\times400$ cells for both the cases using $\ote,$ $\othe$ and $\ofe$ schemes.}
		\label{fig:RP_entropy}
	\end{center}    
\end{figure}

The numerical results for the first case ($\gamma_e = \gamma_i = 1.4$) are presented in Figures~\eqref{fig:RP_sg} at final time $t=0.75$, where we have plots of density, electron pressure, and ion pressure for all three schemes. We observe the $\othe$ and $\ofe$ schemes result in a much more detailed solution than the $\ote$ scheme.

In Figure~\eqref{fig:RP_dg}, we have presented the numerical results for the second case ($\gamma_e = 1.4,$ $\gamma_i = 1.67$) at time $t=0.59$ on $400\times400$ cells. We have plotted the density, electron pressure, and ion pressure for all three schemes. We note that the solution structure has changed significantly. Numerically, we observe that $\ote$ is much more diffusive than the $\othe$ and $\ofe$ schemes.

Total entropy change with time for both cases is plotted in Figure~\eqref{fig:RP_entropy}. We again see that in both cases $\ote$ is more diffusive than $\othe$ and $\ofe$, which are both comparable.

\subsection{Shock-bubble interaction problem}\label{test:S_bubble}
Following~\cite{cheng2024highweno}, we consider the shock bubble interaction problem for this test. We consider a computational domain of $[0,6.5]\times[0,0.89]$. The initial bubble has its centre at $(3.5,0)$ with a radius of $0.5$, and the internal initial state is given by,
\[\left(\rho, \bu, \pe, \pii\right)  = 
(0.1819, 0, 0, 0.220458, 0.220458)
\]
In the rest of the domain, we consider the states,
\[\left(\rho, \bu, \pe, \pii\right)  = \begin{cases}
	(1, 0, 0, 0.3571425, 0.3571425), & \textrm{if } 0\leq x<4.5\\
	(1.3764, -0.3336, 0, 0.560643, 0.560643), & \textrm{if } 4.5\leq x<6.5.\\
\end{cases}\]
The gas constants are $\gamma_e = \gamma_i = 1.4$. We use reflective boundary conditions at the bottom and top boundaries, whereas Dirichlet boundary conditions are used on the left and right boundaries of the domain.
\begin{figure}[!htbp]
	\centering
	\begin{minipage}{0.8\textwidth}
		\centering
		\subfigure[Density ($\rho$) for $\ote$ scheme]{\includegraphics[width=\textwidth]{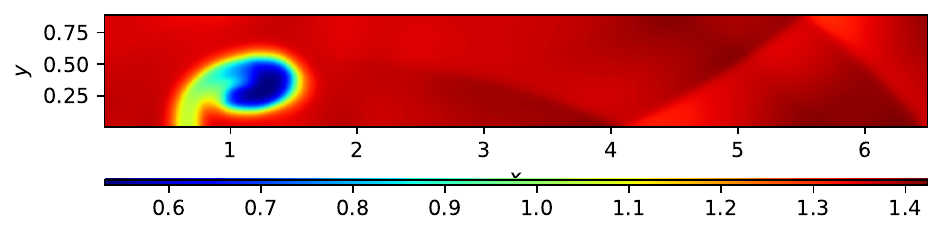}\label{fig:bouble_o2}}\\[1ex]
		\subfigure[Density ($\rho$) for $\othe$ scheme]{\includegraphics[width=\textwidth]{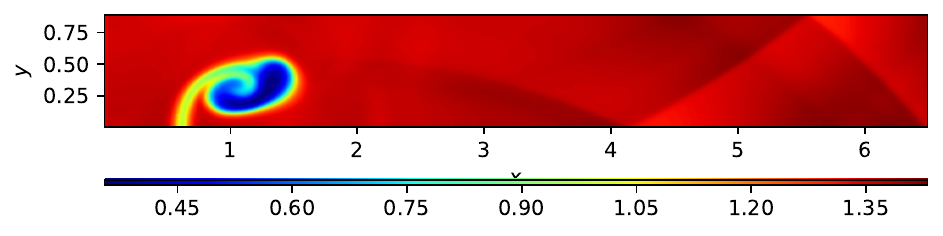}\label{fig:bouble_o3}}\\[1ex]
		\subfigure[Density ($\rho$) for $\ofe$ scheme]{\includegraphics[width=\textwidth]{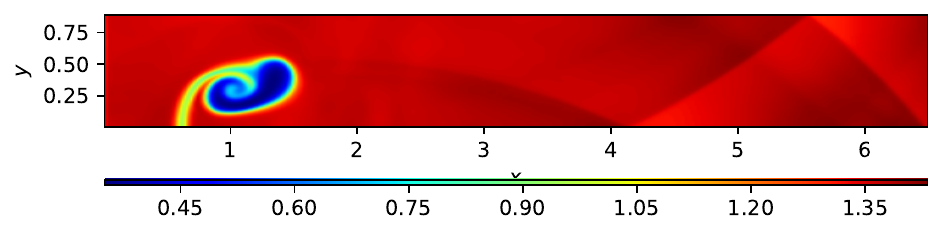}\label{fig:bouble_o4}}\\
	\end{minipage}
	\hfill
	\begin{minipage}{0.45\textwidth}
		\centering
		\subfigure[Total entropy change with time]{\includegraphics[width=\textwidth]{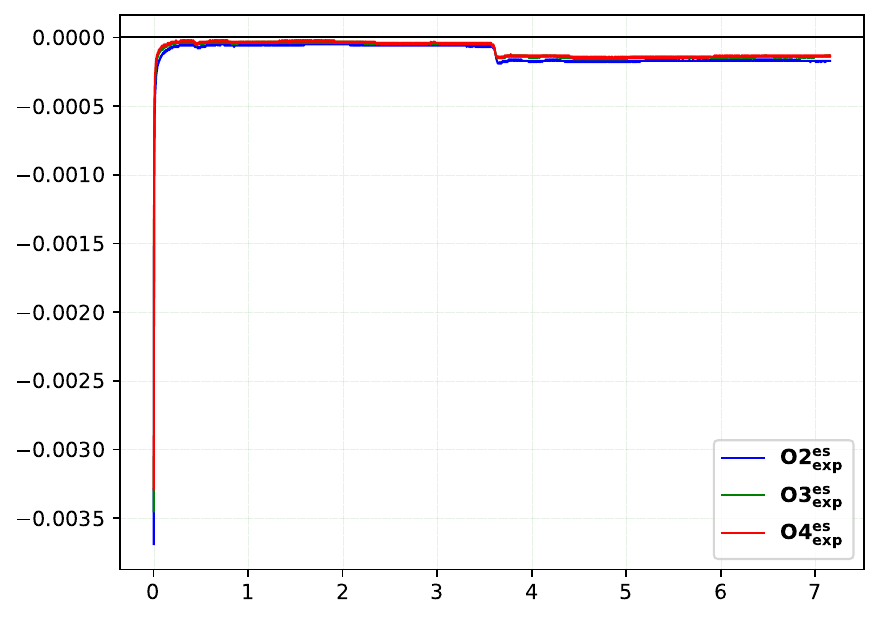}\label{fig:bouble_ent}}
	\end{minipage}
	\caption{\textbf{\nameref{test:S_bubble}}: Plots of density and total entropy evolution at time $t = 7.1571$ using $400\times144$ cells. We have also plotted the total entropy change with time.}
	\label{fig:S_bouble}
\end{figure}
\begin{figure}[!htbp]
	\begin{center}
		\subfigure[Density ($\rho$) for $\ofe$ scheme at time $t =0$]{\includegraphics[width=0.5\textwidth]{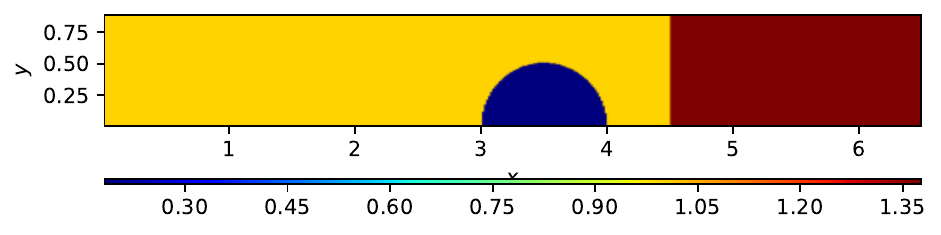}\label{fig:bouble_t0}}~
		\subfigure[Density ($\rho$) for $\ofe$ scheme at time $t =0.6294$]{\includegraphics[width=0.5\textwidth]{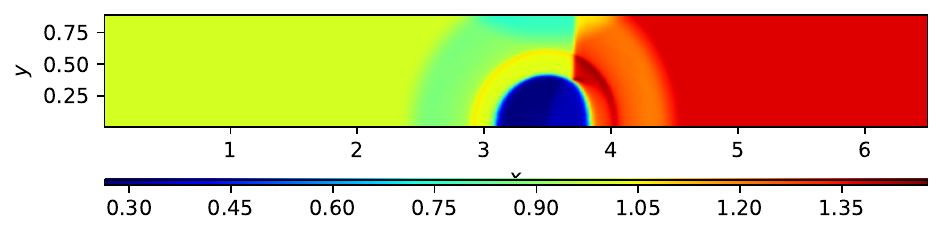}\label{fig:bouble_t06}}\\
		\subfigure[Density ($\rho$) for $\ofe$ scheme at time $t =1.1099$]{\includegraphics[width=0.5\textwidth]{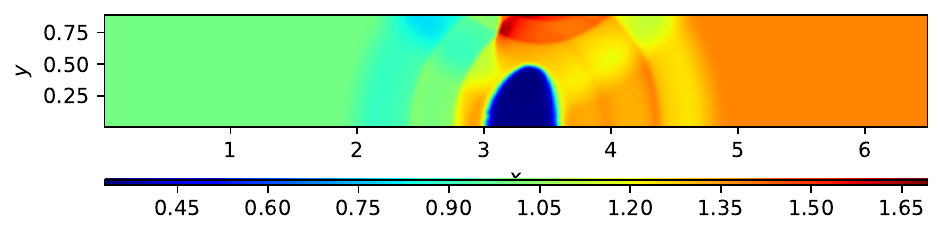}\label{fig:bouble_t1}}~
		\subfigure[Density ($\rho$) for $\ofe$ scheme at time $t =3.3408$]{\includegraphics[width=0.5\textwidth]{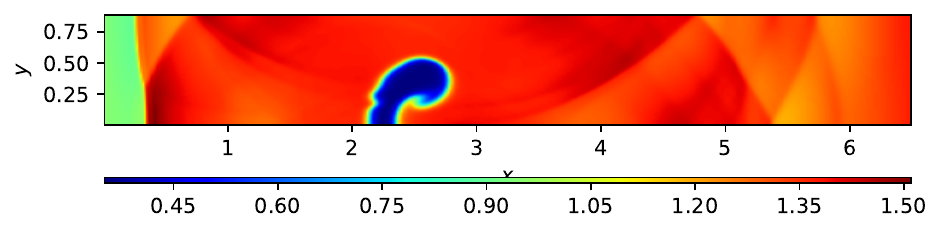}\label{fig:bouble_t3}}\\
		\subfigure[Density ($\rho$) for $\ofe$ scheme at time $t =5.0358$]{\includegraphics[width=0.5\textwidth]{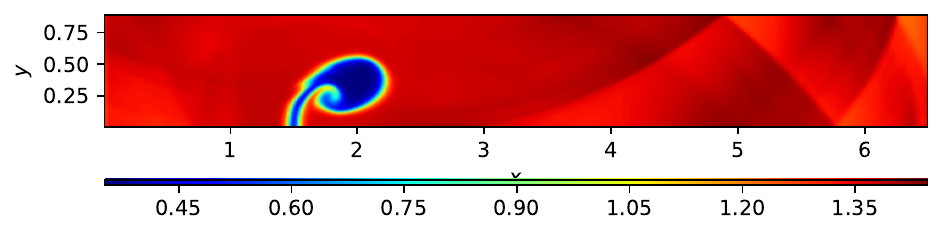}\label{fig:bouble_t5}}~
		\subfigure[Density ($\rho$) for $\ofe$ scheme at time $t =7.1571$]{\includegraphics[width=0.5\textwidth]{bouble/bouble_O4.pdf}\label{fig:bouble_t7}}\\
		\caption{\textbf{\nameref{test:S_bubble}}: Plots of density for $\ofe$ scheme using $400\times144$ cells at different times.}
		\label{fig:S_bouble_t}
	\end{center}
\end{figure}

In Figure~\eqref{fig:S_bouble}, we have plotted the density for $\ote$, $\othe$ and $\ofe$ schemes at time $t=7.1571$ using $400\times144$ cells. We observe that the $\ote$ scheme is the most diffusive and thus fails to capture small-scale structures effectively. In contrast, the $\ofe$ scheme is more accurate than the $\othe$ scheme, and both are able to capture small-scale structures effectively. We have also plotted the total entropy change with time for all the schemes. We note that all the schemes are entropy stable. Again, we note the diffusive behavior of $\ote$ when compared with higher-order schemes $\othe$ and $\ofe.$
 
In Figure~\eqref{fig:S_bouble_t}, we have also plotted the density at different times $t=0.6294, 1.1099,$ $3.3408, 5.0358, 7.1571$, for $\ofe$ scheme. These simulations illustrate the evolution of the bubble after the shock interaction, its deformation, and finally generate a vortex ring. 

\subsection{Richtmyer-Meshkov instability}\label{test:RM_instability}
	 \begin{figure}[!htbp]
	 	\begin{center}
	 		\subfigure[Density ($\rho$) for $\ote$ scheme]{\includegraphics[width=0.3\textwidth]{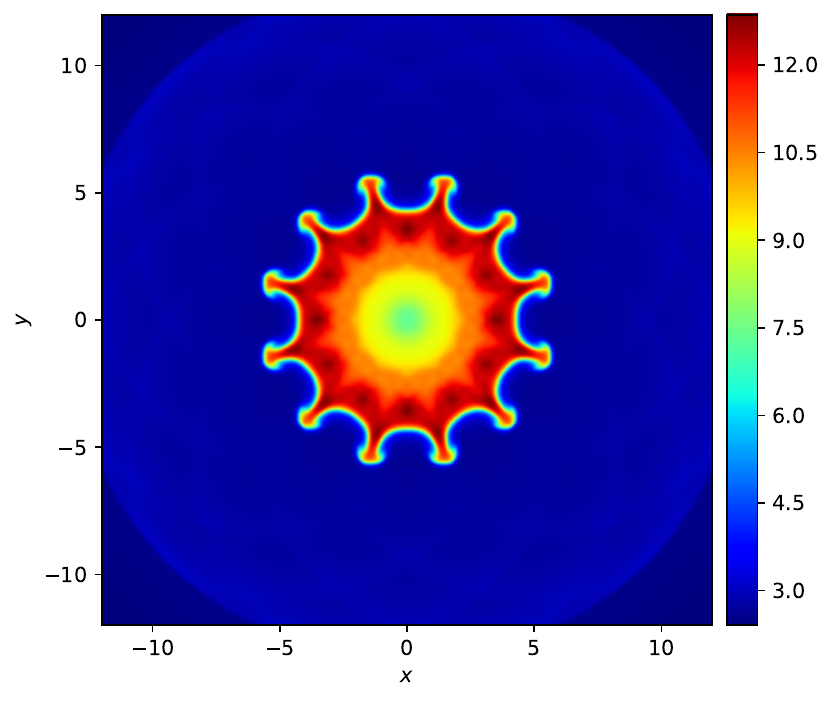}\label{fig:RM_rho_o2}}~
	 		\subfigure[Density ($\rho$) for $\othe$ scheme]{\includegraphics[width=0.3\textwidth]{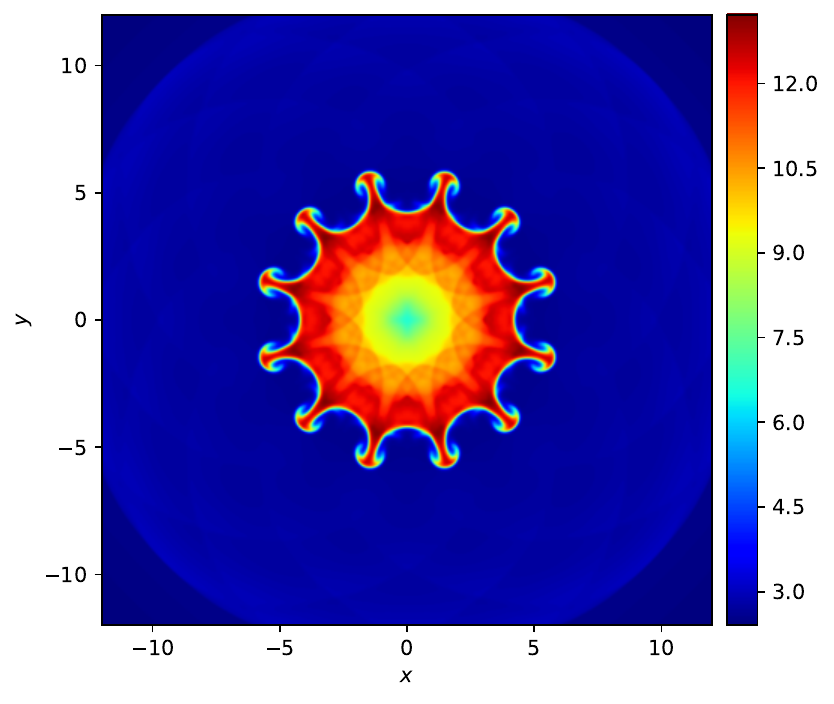}\label{fig:RM_rho_o3}}~
	 		\subfigure[Density ($\rho$) for $\ofe$ scheme]{\includegraphics[width=0.3\textwidth]{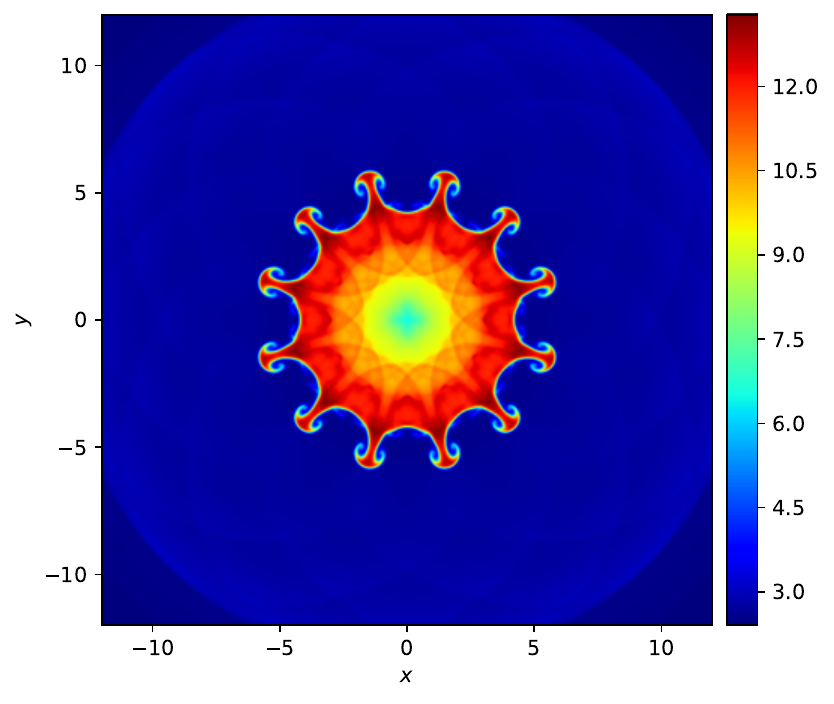}\label{fig:RM_rho_o4}}\\
	 		\subfigure[Electron pressure ($\pe$) for $\ote$ scheme]{\includegraphics[width=0.3\textwidth]{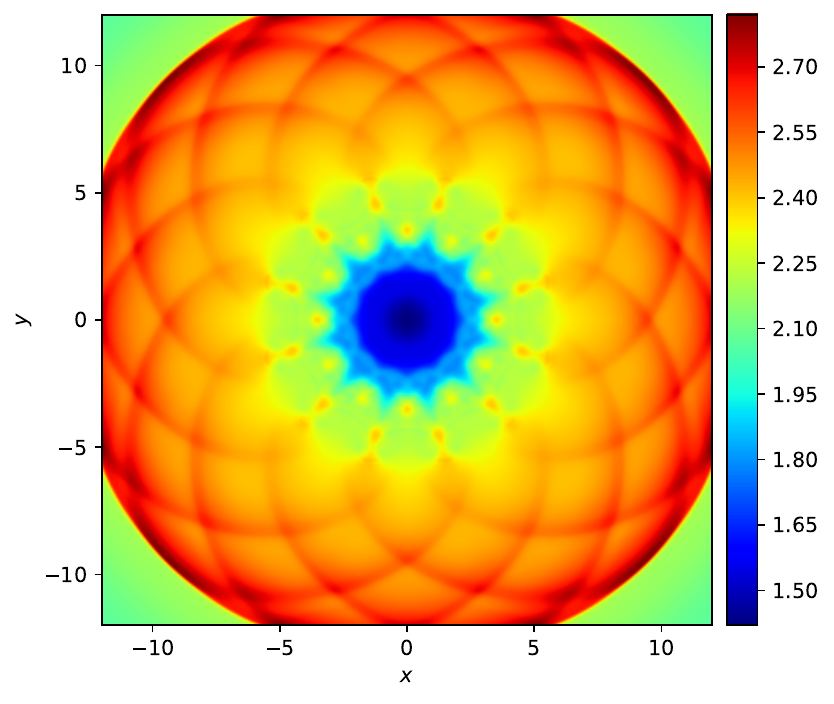}\label{fig:RM_pe_o2}}~
	 		\subfigure[Electron pressure ($\pe$) for $\othe$ scheme]{\includegraphics[width=0.3\textwidth]{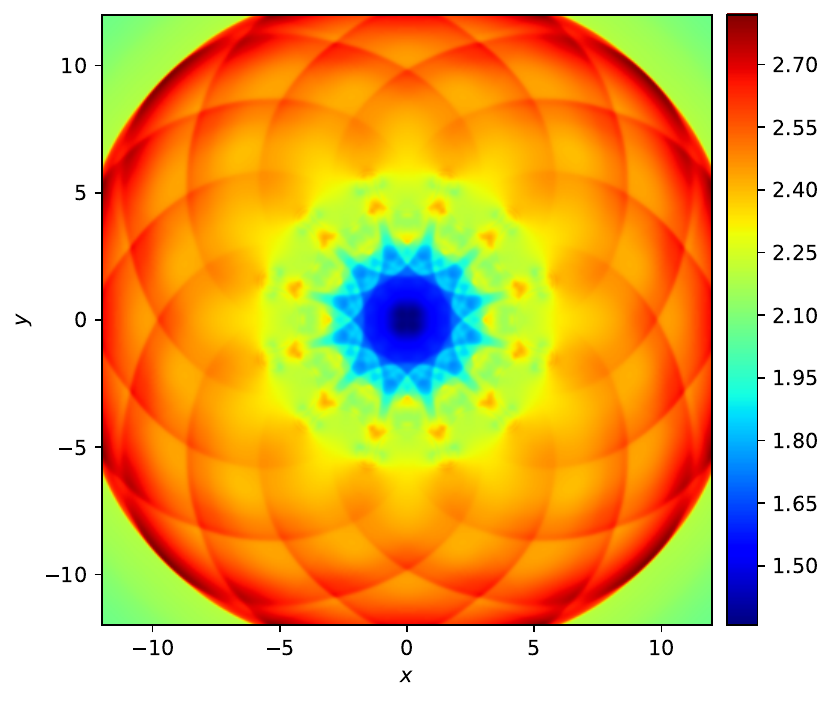}\label{fig:RM_pe_o3}}~
	 		\subfigure[Electron pressure ($\pe$) for $\ofe$ scheme]{\includegraphics[width=0.3\textwidth]{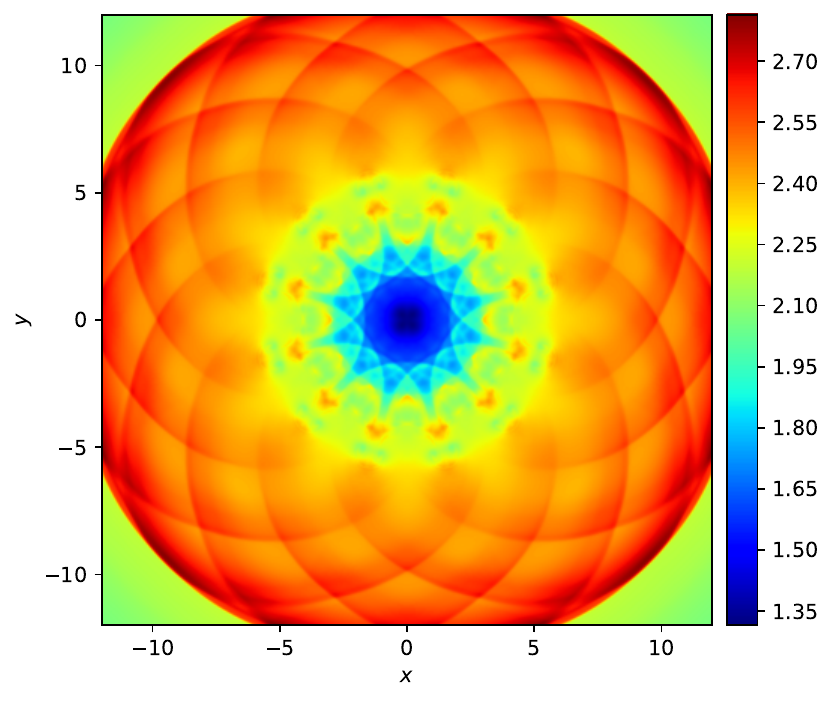}\label{fig:RM_pe_o4}}\\
	 		\subfigure[Ion pressure ($\pii$) for $\ote$ scheme]{\includegraphics[width=0.3\textwidth]{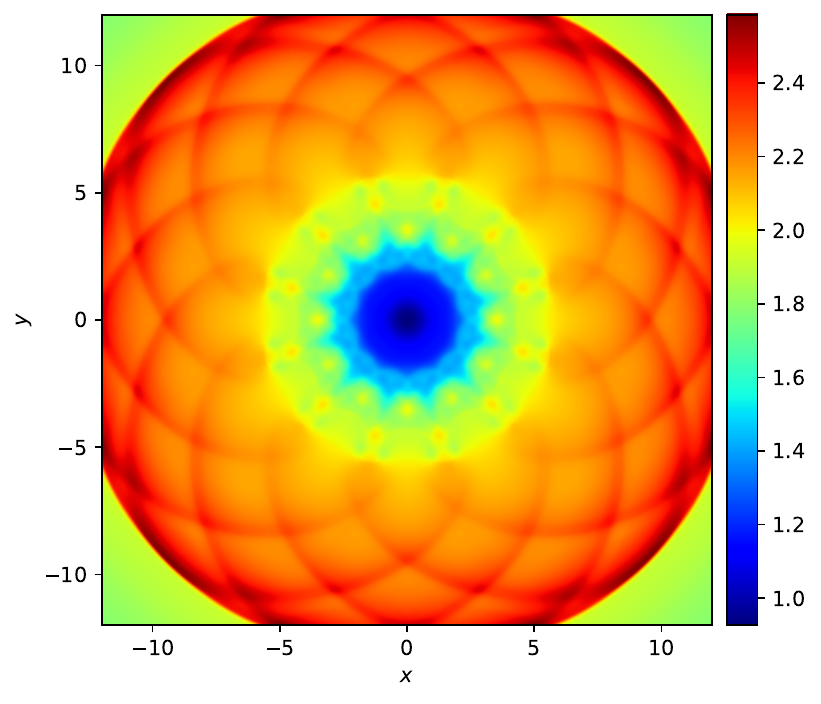}\label{fig:RM_pi_o2}}~
	 		\subfigure[Ion pressure ($\pii$) for $\othe$ scheme]{\includegraphics[width=0.3\textwidth]{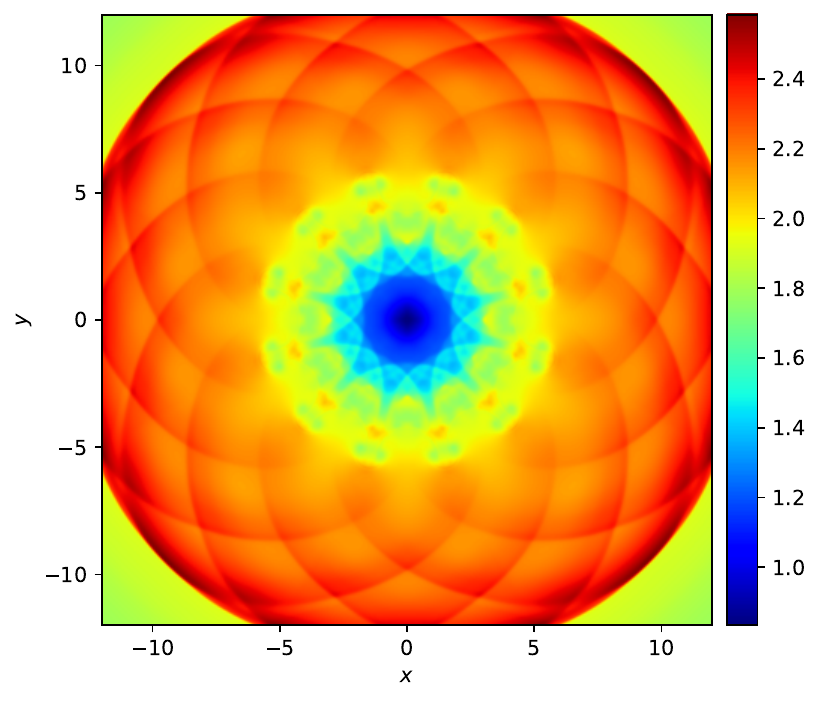}\label{fig:RM_pi_o3}}~
	 		\subfigure[Ion pressure ($\pii$) for $\ofe$ scheme]{\includegraphics[width=0.3\textwidth]{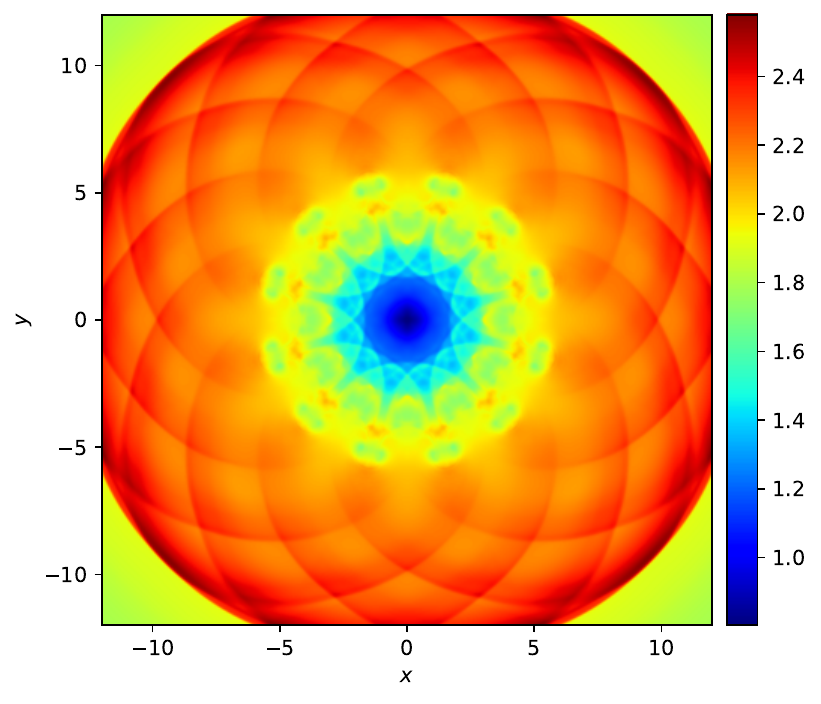}\label{fig:RM_pi_o4}}\\
	 		\caption{\textbf{\nameref{test:RM_instability}}: Plots of density, electron pressure and ion pressure at time $t = 17.46$ using $800\times800$ cells.}
	 		\label{fig:RM}
	 	\end{center}
	 \end{figure}
	 \begin{figure}[!htbp]
	 	\begin{center}
	 		\subfigure[Density ($\rho$) for $\ote$ scheme]{\includegraphics[width=0.3\textwidth]{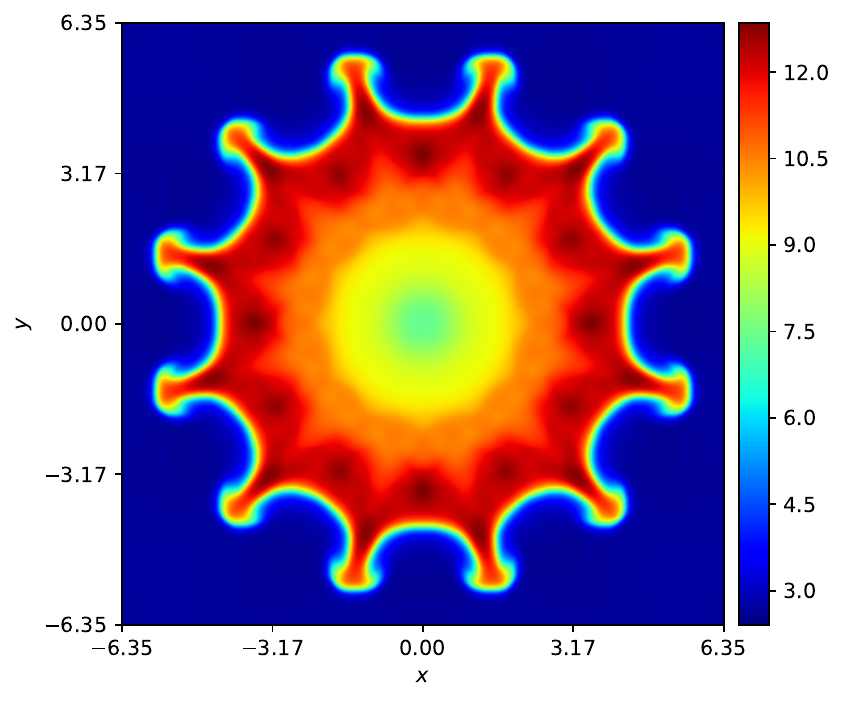}\label{fig:RM_rho_z_o2}}~
	 		\subfigure[Density ($\rho$) for $\othe$ scheme]{\includegraphics[width=0.3\textwidth]{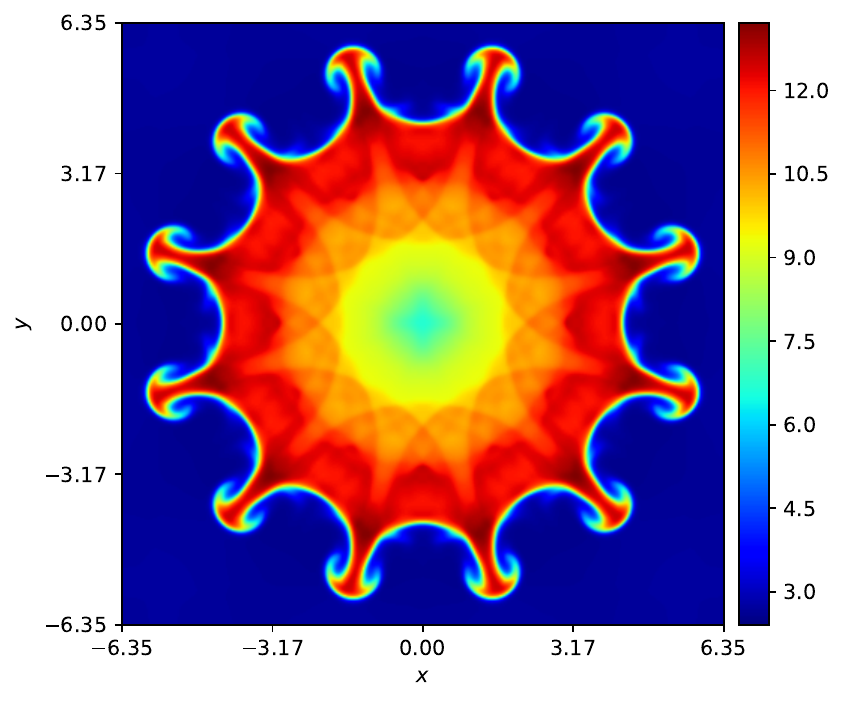}\label{fig:RM_rho_z_o3}}~
	 		\subfigure[Density ($\rho$) for $\ofe$ scheme]{\includegraphics[width=0.3\textwidth]{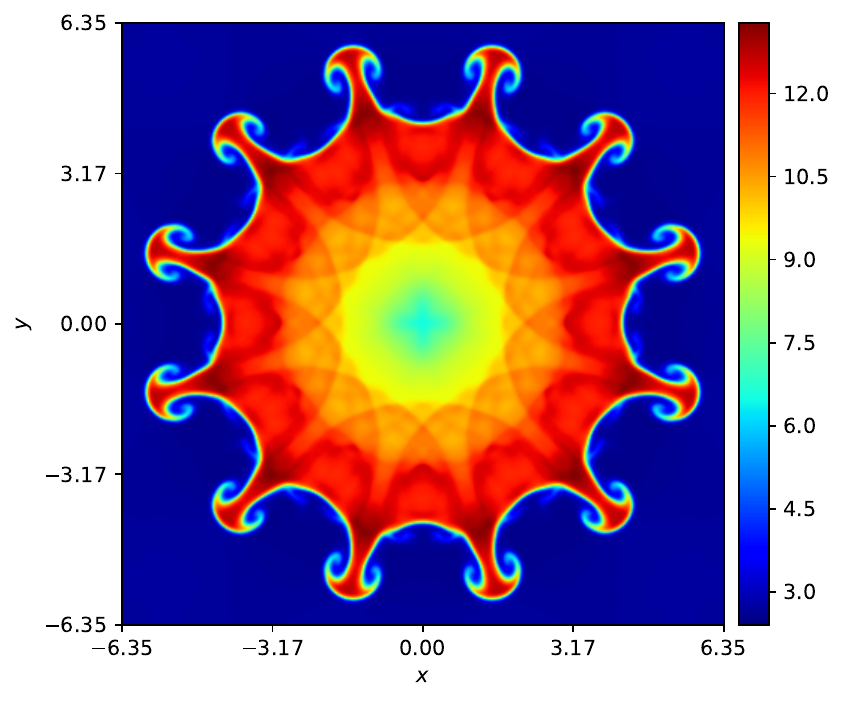}\label{fig:RM_rho_z_o4}}\\
	 		\subfigure[Electron pressure ($\pe$) for $\ote$ scheme]{\includegraphics[width=0.3\textwidth]{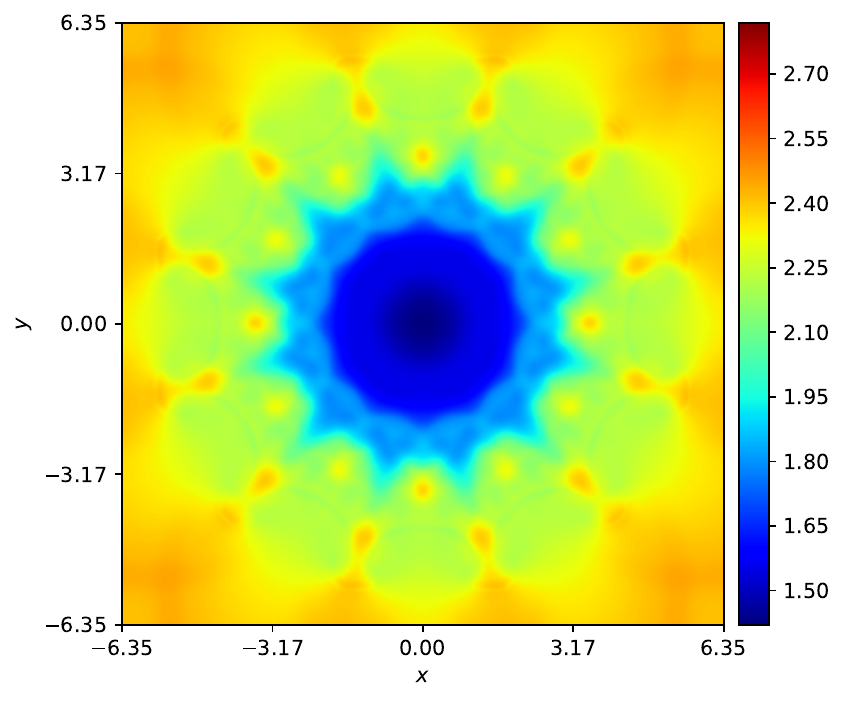}\label{fig:RM_pe_z_o2}}~
	 		\subfigure[Electron pressure ($\pe$) for $\othe$ scheme]{\includegraphics[width=0.3\textwidth]{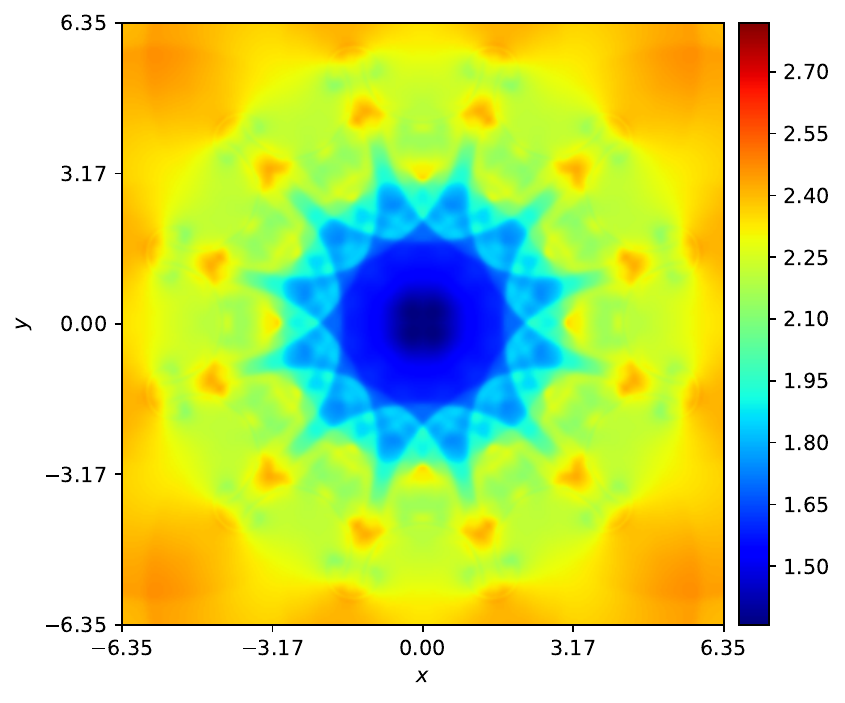}\label{fig:RM_pe_z_o3}}~
	 		\subfigure[Electron pressure ($\pe$) for $\ofe$ scheme]{\includegraphics[width=0.3\textwidth]{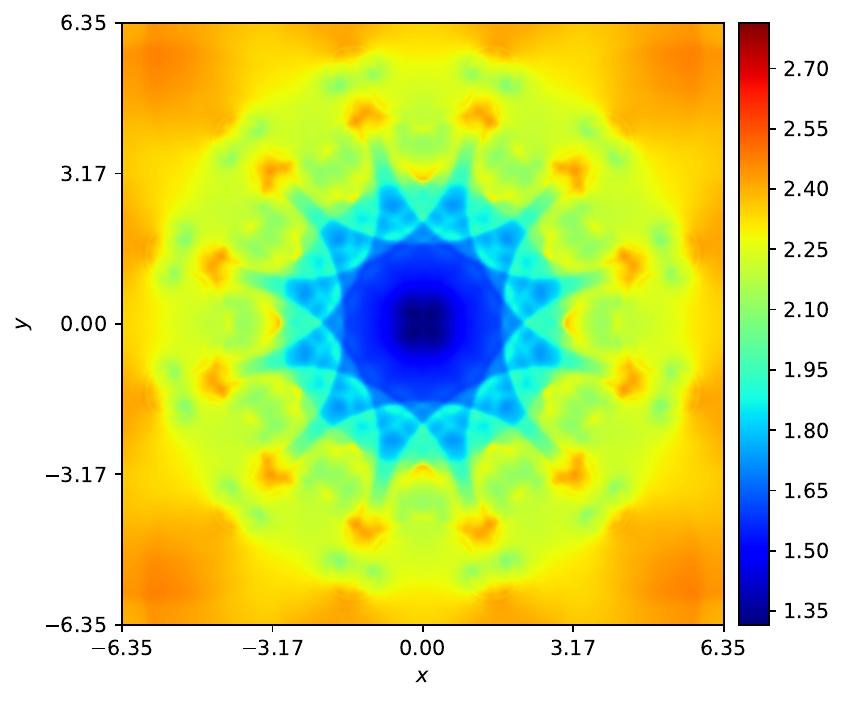}\label{fig:RM_pe_z_o4}}\\
	 		\subfigure[Ion pressure ($\pii$) for $\ote$ scheme]{\includegraphics[width=0.3\textwidth]{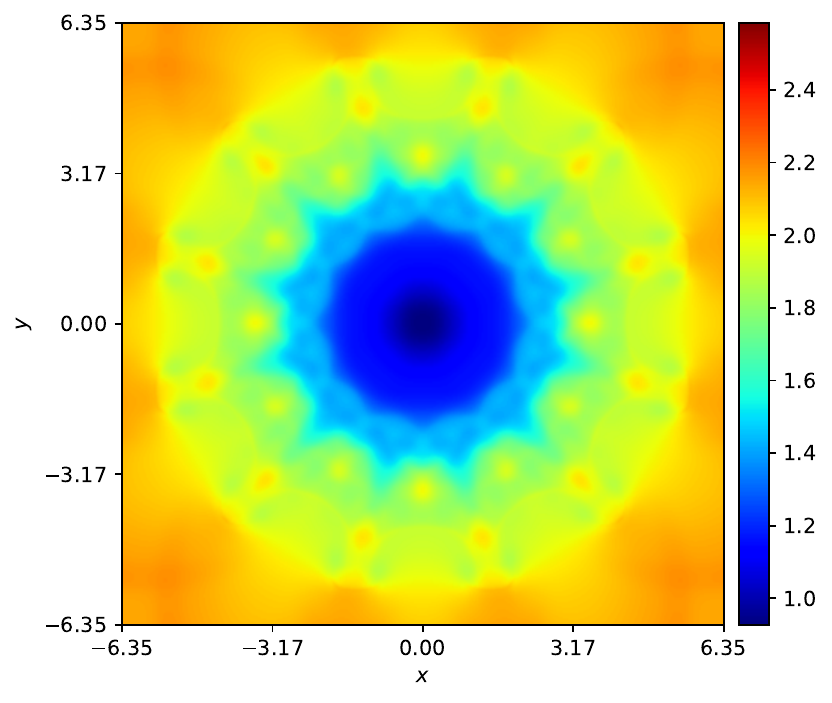}\label{fig:RM_pi_z_o2}}~
	 		\subfigure[Ion pressure ($\pii$) for $\othe$ scheme]{\includegraphics[width=0.3\textwidth]{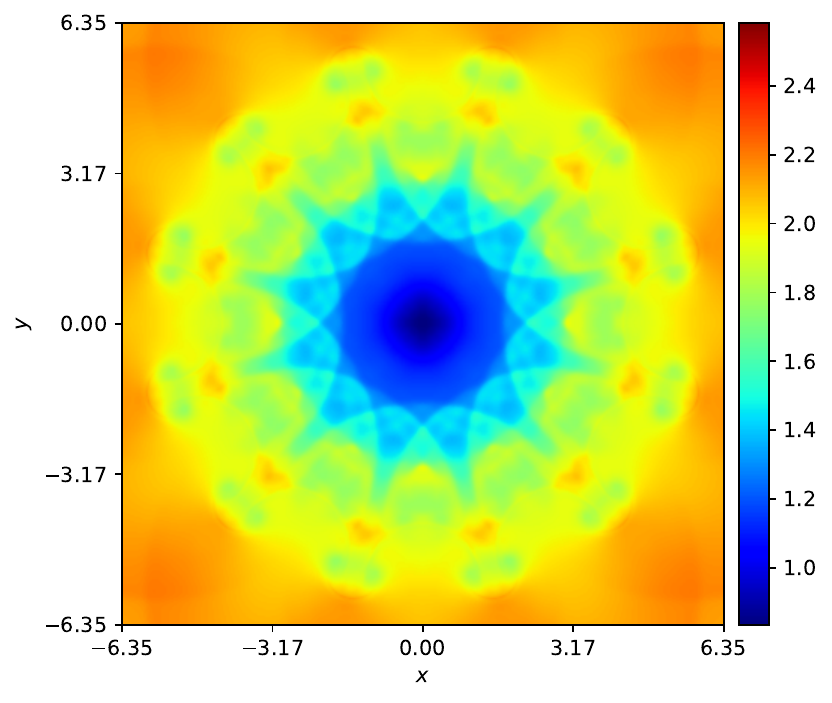}\label{fig:RM_pi_z_o3}}~
	 		\subfigure[Ion pressure ($\pii$) for $\ofe$ scheme]{\includegraphics[width=0.3\textwidth]{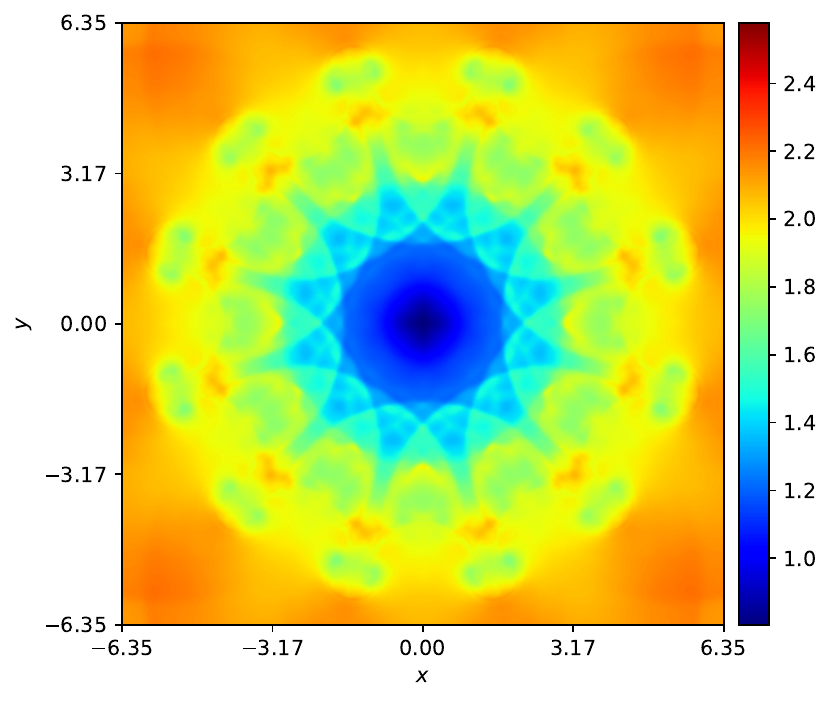}\label{fig:RM_pi_z_o4}}\\
	 		\caption{\textbf{\nameref{test:RM_instability}}: Zoomed plots of density, electron pressure and ion pressure at time $t = 17.46$ using $800\times800$ cells.}
	 		\label{fig:RM_zoom}
	 	\end{center}
	 \end{figure}
	 \begin{figure}[!htbp]
	 	\begin{center}
	 		\includegraphics[width=0.45\textwidth]{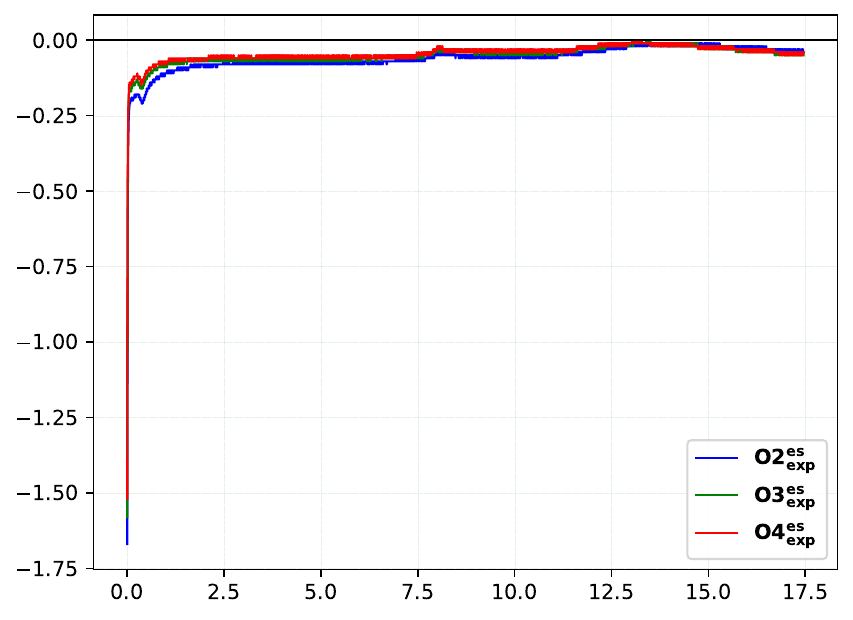}
	 		\caption{\textbf{\nameref{test:RM_instability}}: Plots of total entropy change with time using $800\times800$ cells for $\ote$, $\othe$ and $\ofe$ schemes.}
	 		\label{fig:RM_entropy}
	 	\end{center}    
	 \end{figure}
In the final test case, we consider the Richtmyer-Meshkov instability problem for the OFTT-Euler system proposed in~\cite{cheng2024high}, which has applications in ICF (see~\cite{tian2006numerical,sauppe2019using}). We consider a computational domain of $[-12,12]\times[-12,12]$.  The incident shock wave is initially located at $r_s=7.5$ and the perturbed contact discontinuity is initially located at $r_c=r_0+A_0 \cos{(\omega\theta)}$, where $r_0 = 7.162$, the perturbation amplitude $A_0 = 0.24$, and the perturbation wave number $\omega = 12$.  The hydrodynamic instability is caused by the incident shock wave hitting the perturbed contact discontinuity. Following~\cite{cheng2024high}, the problem is initialized as follows:
\[\left(\rho, \bu, \pe, \pii\right)  = \begin{cases}
	(5.04, 0, 0, 0.6, 0.4), & \textrm{if } x^2 + y^2<r_c^2,\\
	(1, 0, 0, 0.6, 0.4), & \textrm{if } r_c^2<x^2+y^2<r_s^2,\\
	(1.479, -0.518\cos{(\theta)}, -0.518\sin{(\theta)}, 1.041, 0.788), & \textrm{otherwise,}\\
\end{cases}\]
with $\gamma_e = 1.4$ and $\gamma_i = 1.67$. We use Neumann boundary conditions.

 We have presented numerical results in Figure~\eqref{fig:RM} for the $\ote$, $\othe$ and $\ofe$ schemes at time $t=17.46$ using $800\times800$ cells. We have plotted density, electron pressure, and ion pressure. We observe that $\ote$ is the most diffusive and fails to capture small-scale structures effectively. Whereas the higher-order schemes $\othe$ and $\ofe$ are much more accurate and able to capture the small-scale features. To demonstrate this, we have also plotted the zoomed-in plots of the results in Figure~\eqref{fig:RM_zoom}, where we have zoomed at the center of the domain. We clearly see that $\othe$ and $\ofe$ schemes produce much more detail of small-scale structures than the $\ote$ scheme. The results are also similar to those presented in~\cite{cheng2024high}. In Figure~\eqref{fig:RM_entropy}, we have plotted the total entropy change with time. The plot clearly demonstrates the entropy stability of the numerical schemes.
 
\section{Conclusion}
\label{sec:conc}
In this work, we have designed higher-order entropy-stable finite difference schemes for the One-Fluid Two-Temperature Euler Non-equilibrium Hydrodynamics system, which are a set of hyperbolic PDEs with non-conservative terms. We show that the presence of the non-conservative terms results in non-symmetrizability of the system. To design entropy-stable numerical schemes, we present a novel reformulation of the equations so that the conservative part is symmetrizable and the non-conservative part does not contribute to the entropy evolution. Finally, an entropy-stable discretization for the new conservative part is proposed, which results in the entropy stability of the complete discretization. We have presented extensive numerical results to demonstrate the accuracy and entropy stability of the proposed numerical schemes.
\section*{Acknowledgements}
Harish Kumar acknowledges support from a Vajra award (VJR/2018/00129).
\bibliographystyle{unsrt}  
\bibliography{reference}
\appendix
\section{Entropy invariance of non-conservative terms in reformulated OFTT-Euler system}
\label{Non_con_entropy_evolution}
	\begin{lemma}
The non-conservative terms of the reformulated system~\eqref{eq:oftt_ref_noncons} do not contribute to the entropy evolution, i.e., the coefficient matrices $\bc_x(\con)$ and $\bc_y(\con)$ satisfy,
\begin{equation*}
	\evar^\top \bc_x(\con) = \evar^\top \bc_y(\con) = 0.
\end{equation*}
\end{lemma}
\begin{proof}
We compute $\evar^\top\cdot C_i$ for $i=\{1,2,3,4,5\}$, where $C_i$ are the columns of the matrix $\bc_x(\con)$. We get,
\begin{align*}
	\evar^\top\cdot C_1&= \evar^\top\cdot \left\{0,-\frac{(\gamma_e-1)|\bu|^2}{2}, 0, -\left(\frac{(\gamma_e-1)|\bu|^2}{2} + \frac{\pii-\pe}{\rho}\right)v_x, -\frac{(\gamma_i-1)\pii v_x}{\rho}\right\} \\
	&=-\textcolor{black}{\frac{(\gamma_e-1)|\bu|^2}{2}\left(\beta_e v_x\right)} -\left(\textcolor{black}{\frac{(\gamma_e-1)|\bu|^2}{2}} + \textcolor{black}{\frac{\pii-\pe}{\rho}}\right)v_x\left(-\beta_e\right) \\
	&- \textcolor{black}{\frac{(\gamma_i-1)\pii v_x}{\rho}\left(\frac{\beta_e-\beta_i}{\gamma_i-1}\right)}\\
	&=0,\\
	\evar^\top\cdot C_2&= \evar^\top\cdot \left\{0, (\gamma_e-1)v_x, 0, (\gamma_e-1) v_x^2 + \frac{\pii-\pe}{\rho}, \frac{(\gamma_i-1)\pii}{\rho}\right\} \\
	&=\textcolor{black}{(\gamma_e-1)v_x\left(\beta_e v_x\right)} +\left(\textcolor{black}{(\gamma_e-1) v_x^2} + \textcolor{black}{\frac{\pii-\pe}{\rho}}\right)\left(-\beta_e\right) + \textcolor{black}{\frac{(\gamma_i-1)\pii}{\rho}\left(\frac{\beta_e-\beta_i}{\gamma_i-1}\right)} \\
	&=0,\\
	&\\
	\evar^\top\cdot C_3&= \evar^\top\cdot \left\{0,(\gamma_e-1)v_y, 0, (\gamma_e-1)v_x v_y, 0\right\} = 0,\\
	\evar^\top\cdot C_4&= \evar^\top\cdot \left\{0,-(\gamma_e-1), 0, -(\gamma_e-1)v_x, 0\right\} = 0,\\
\text{and}&\\
	\evar^\top\cdot C_5&= \evar^\top\cdot \left\{0,\left(\frac{\gamma_e-1}{\gamma_i-1} + 1\right), 0, \left(\frac{\gamma_e-1}{\gamma_i-1} + 1\right)v_x, 0\right\} = 0.
\end{align*}
Hence  $\evar^\top \bc_x(\con) = 0$. Similarly, we can easily prove that $\evar^\top \bc_y(\con) = 0$.
\end{proof}  

\section{Non-symmetrizability of OFTT-Euler system}
\label{Non_Symmetrizability}
Here, we investigate symmetrization of the OFTT-Euler system. Let us recall,
\begin{definition}
The conservation law 
\begin{equation}
	\frac{\p \con}{\p t} + \frac{\p \f_x}{\p x}=0,
	\label{eq:1d_con}
\end{equation}
is said to be symmetrizable if there exists a change of variable $\con \rightarrow \evar$ which symmetrizes it, i.e., \eqref{eq:1d_con} with a change of variable can be written as,
$$
\frac{\p \con}{\p \evar}\frac{\p \evar}{\p {t}}+\frac{\p \textbf{f}_x}{\p \con}\frac{\p \con}{\p \evar}\frac{\p \evar}{\p {x}}=0.
$$
where $\frac{\p \con}{\p \evar}$ is a symmetric positive definite matrix.  Furthermore, $\frac{\p \textbf{f}_x}{\p \con}  \frac{\p \con}{\p \evar}$ is symmetric matrix.
\end{definition}
We also have the following results from~\cite{godlewski2013numerical}:
\begin{thm}\label{thm_CS}
The conservative system has a strictly convex entropy $\ent(\con)$ if and only if the system is symmetrizable. 
\end{thm} 
\begin{remark}
The conservative part of the OFTT-Euler system is symmetrizable because it satisfies,
\begin{align*}
	\frac{\p \textbf{f}_x}{\p \evar}=\bigg(\frac{\p \textbf{f}_x}{\p \evar}\bigg)^{\top}.
\end{align*}
\end{remark}
We now consider the complete OFTT-Euler system in one dimension, i.e.
\begin{equation}\label{eq:1d_noncon}
\frac{\p \con}{\p t}+\frac{\p \f_{x}}{\p x} + \bc_{x}(\con)\frac{\p \con}{\p x} =0.
\end{equation}
where $\con,~\f_{x}$ and $\bc_{x}(\con)$ are defined in Section\eqref{sec:reformulation}. Following~\cite{yadav2023entropy}, to prove the symmetrizability of the system, we need to prove that the matrix
$$
\mathcal{H}(\con) = \left(\frac{\p \textbf{f}_x}{\p \con} +\bc_{x}(\con) \right)\frac{\p \con}{\p \evar},
$$ 
is symmetric. However, a lengthy calculation shows that the difference $\mathcal{H}(\con) -\mathcal{H}(\con)^\top$ is equal to
\begin{align*}
   \begin{pmatrix}
	0 & \frac{\pe-\pii}{2} & 0 & \frac{(\pe-\pii)v_x}{2} & 0\\
	\frac{\pii-\pe}{2} & 0 & \frac{(\pii-\pe)v_y}{2} & \Gamma & -\frac{(\pe-\pii)\pii(-1+2\gamma_i)}{2\rho}\\
	0 & \frac{(\pe-\pii)v_y}{2} & 0 & \frac{(\pe-\pii)v_x v_y}{2} & 0\\
	\frac{(\pii-\pe)v_x}{2} & -\Gamma & \frac{(\pii-\pe)v_x v_y}{2} & 0 & -\frac{(\pe-\pii)\pii v_x(-1+2\gamma_i)}{2\rho}\\
	0 & \frac{(\pe-\pii)\pii(-1+2\gamma_i)}{2\rho} & 0 & \frac{(\pe-\pii)\pii v_x(-1+2\gamma_i)}{2\rho} & 0
\end{pmatrix}.
\end{align*}
with $\Gamma=\frac{(\pe-\pii)\left(-2\pe\left(\gamma_i-1\right) + \left(\gamma_e-1\right) \left(2\pii\left(1-2\gamma_i\right) + \rho\left(v_x^2-v_y^2\right)\left(\gamma_i-1\right) \right) \right)}{4\rho\left(\gamma_e-1\right)\left(\gamma_i-1\right)}$. This is not zero. Hence, the system is not symmetrizable. However, we note that when $\pe=\pii$ the matrix $\mathcal{H}(\con)$ is symmetric, which is expected as in that case, the OFTT-Euler system reduces to compressible Euler equations.

\section{Consistency of the Numerical Entropy Flux}\label{Cons_entropy_flux}
We have the following lemma:
\begin{lemma}
	\label{lem:entropy_flux_consistency}
	The numerical entropy flux
	\begin{equation*}
		\tilde{\entf}_{d}\left(\con_l,\con_r\right)=\bar{\evar} \cdot \tilde{\F}_{d}\left(\con_l,\con_r\right) - \bar{\mathcal{F}}_{d},~~~~\forall~d\in\{x,y\}
	\end{equation*}
	is consistent with the entropy flux $q_d$ in \eqref{entropy-pair}.
\end{lemma}
\begin{proof}
    Let $\con_l=\con_r=\con$. Then the arithmetic averages reduce to point values, i.e.\ $\bar{\evar}=\evar\left(\con\right)$, $\bar{\mathcal{F}}_d=\mathcal{F}_d\left(\con\right)$, $\bar{\rho}=\rho$, $\bar{v}_d=v_d$ and $\bar{\beta}_e=\beta_e$. Moreover, the logarithmic mean satisfies $a^{\ln}\to a$ as $a_r\to a_l = a$. Consequently, $\rho^{\ln}=\rho$,
	$\beta_e^{\ln}=\beta_e$ and $\beta_i^{\ln}=\beta_i$. Therefore, the numerical flux is consistent, yielding $\tilde{\F}_d\left(\con,\con\right) = \mathbf{f}_d\left(\con\right)$. Furthermore, from~\eqref{entropy_potential}, we have $\mathcal{F}_d = \evar\cdot\mathbf{f}_d - \entf_d$. Hence, the corresponding numerical entropy flux satisfies
    \begin{equation*}
		\tilde{\entf}_{d}\left(\con,\con\right)=\bar{\evar} \cdot \tilde{\F}_{d}\left(\con_l,\con_r\right) - \bar{\mathcal{F}}_{d}=\evar\cdot\mathbf{f}_d - \left(\evar\cdot\mathbf{f}_d - \entf_d\right) = \entf_d .
	\end{equation*}
\end{proof}
For higher-order numerical entropy flux, we have the following notes:
\begin{note}\label{note:consis_1}
    Following \cite{singh2024entropy}, the fourth-order numerical entropy flux
	in the $x$-direction is constructed from the second-order numerical entropy flux
	$\tilde{q}_x$ as
	\begin{equation*}
		\tilde{\entf}^4_{x,i+\frac12,j}
		= \frac{4}{3}\,\tilde{\entf}_x\left(\con_{i,j},\con_{i+1,j}\right)
		- \frac{1}{6}\Big(\tilde{\entf}_x\left(\con_{i-1,j},\con_{i+1,j}\right)
		+ \tilde{\entf}_x\left(\con_{i,j},\con_{i+2,j}\right)\Big).
	\end{equation*}
	To verify consistency, let all the stencil values be equal to $\con$. By Lemma~\ref{lem:entropy_flux_consistency}, each second-order flux on the
	right-hand side equals $\entf_x\left(\con\right)$. Hence,
	\begin{equation*}
		\tilde{\entf}^{4}_{x,i+\frac12,j}
		= \Big(\frac{4}{3}-\frac{1}{6}-\frac{1}{6}\Big)\entf_x\left(\con\right)
		= \entf_x\left(\con\right).
	\end{equation*}
	Thus, $\tilde{\entf}^{4}_{x,i+\frac{1}{2},j}$ is consistent with the entropy
	flux $\entf_x$. Similarly, the fourth-order numerical entropy flux $\tilde{\entf}^{4}_{y,i,j+\frac{1}{2}}$ in the $y$-direction is consistent with the entropy flux $\entf_y$.
\end{note}

\begin{note}\label{note:consis_2}
	In Theorem~\eqref{thm:ES_higher_order}, the numerical entropy flux is given by
	\begin{equation*}
		\hat{\entf}^k_{x,i+\frac{1}{2},j}=  \tilde{\entf}^{2p}_{x,i+\frac{1}{2},j} - \frac{1}{2}\bar{\evar}_{i+\frac{1}{2},j}^{\top}  \textbf{D}_{x,i+\frac{1}{2},j}[\![ \hat{\evar}]\!]^k_{x,i+\frac{1}{2},j}.
	\end{equation*}
	If all the stencil values are equal to $\con$, the reconstructed
	jump $[\![ \hat{\evar}]\!]^k_{x,i+\frac{1}{2},j}$ is equal to zero. So
	$\hat{\entf}^k_{x,i+\frac{1}{2},j}=  \tilde{\entf}^{2p}_{x,i+\frac{1}{2},j} = \entf_x\left(\con\right)$. Hence $\hat{\entf}^k_{x,i+\frac{1}{2},j}$ is consistent with $\entf_x$. Similarly, the fourth-order numerical entropy flux $\hat{\entf}^k_{y,i,j+\frac{1}{2}}$ in the $y$-direction is consistent with the entropy flux $\entf_y$.
\end{note}
\section{Entropy-scaled right eigenvectors for conservative part of the reformulated OFTT-Euler system}\label{Eiegn_vector_for_con}
We consider the conservative part of the OFTT-Euler system
\begin{equation}
	\df{\con}{t} + \sum_{d\in\left\{x,y\right\}}\frac{\p \f_d}{\p \con}\df{\con}{d} = 0.\label{eq:jacobi}
\end{equation}
Let us consider the primitive variables $\textbf{W} = \{\rho,v_x,v_y,\pe,\pii\}$. For system~\eqref{eq:jacobi}, Section~\eqref{RE} describes the right eigenvectors with respect to the primitive variables, while Section~\eqref{ESRE} shows the expressions for the scaled eigenvectors with respect to the primitive variables.
\subsection{Right eigenvectors}\label{RE}
To compute eigenstructure, we rewrite \eqref{eq:jacobi} in terms of primitive variables $\mathbf{W}$. The eigenvalues $\tilde{\mathbf{\Lambda}}_d$ of the Jacobian $\frac{\p \f_d}{\p \con}$ are then
\begin{equation*}
		\tilde{\mathbf{\Lambda}}_d=(v_d,~ v_d,~ v_d, ~v_d\pm c_f),~~~~~~\text{where}~c_f=\sqrt{\frac{2\pe(2\gamma_e-1)}{\rho}}
		\label{eq:eigenvalues}
	\end{equation*}
The corresponding right eigenvectors corresponding to eigenvalues $v_x,~ v_x,~ v_x, ~v_x\pm c_f$ are given below,
    \begin{align*}
		\begin{pmatrix*}
			  1\\
			0\\
			0\\
			0\\
			0 
		\end{pmatrix*},~\begin{pmatrix*}
			  0\\
			0\\
			1\\
			0\\
			0 
		\end{pmatrix*},~\begin{pmatrix*}
			  0\\
			0\\
			0\\
			0\\
			1 
		\end{pmatrix*},~\begin{pmatrix*}
			  \rho^2\\
			\pm\sqrt{2\pe(2\gamma_e-1)\rho}\\
			0\\
			\pe(2\gamma_e-1)\rho\\
			\pii\rho 
		\end{pmatrix*}.
	\end{align*}
Similarly, the corresponding right eigenvectors corresponding to eigenvalues $v_y,~ v_y,~ v_y, ~v_y\pm c_f$ are given below, 
\begin{align*}
		\begin{pmatrix*}
			  1\\
			0\\
			0\\
			0\\
			0 
		\end{pmatrix*},~\begin{pmatrix*}
			  0\\
			1\\
			0\\
			0\\
			0 
		\end{pmatrix*},~\begin{pmatrix*}
			  0\\
			0\\
			0\\
			0\\
			1 
		\end{pmatrix*},~\begin{pmatrix*}
			  \rho^2\\
			0\\
			\pm\sqrt{2\pe(2\gamma_e-1)\rho}\\
			\pe(2\gamma_e-1)\rho\\
			\pii\rho 
		\end{pmatrix*}.
	\end{align*}
Under the assumption $\gamma_e>1$, $\gamma_i>1$ and equation~\eqref{eq:tt_domain}, all eigenvalues are real, and the set of eigenvectors in both directions is linearly independent. 
\subsection{Entropy scaled right eigenvectors via Barth scaling process}\label{Barth}
Here, we will compute the entropy-scaled right eigenvectors in $x$-direction using the Barth scaling procedure~\cite{barth1999numerical}. The right eigenvectors of the Jacobian matrix $\frac{\p \f_x}{\p \con}$ for the system~\eqref{eq:jacobi} are described in~\ref{RE} in terms of primitive variables.  Next, we define
\begin{equation*}
	R^x = \frac{\p \con}{\p \textbf{W}} R^x_{\textbf{W}},
\end{equation*}
where $R^x$ is the right eigenvector matrix for the Jacobian matrix $\frac{\p \f_x}{\p \con}$ in terms of working variables and $\frac{\p \con}{\p \textbf{W}}$ is the Jacobian matrix for the change of variables, which is given below
\begin{align*}
	\frac{\p \con}{\p \textbf{W}}=
	\begin{pmatrix}
		1 & 0 & 0 & 0 & 0 \\
		v_x & \rho & 0 & 0 & 0 \\
		v_y & 0 & \rho & 0 & 0 \\
		\frac{|\bu|^2}{2} & \rho v_x & \rho v_y & \frac{1}{\gamma_e-1} & \frac{1}{\gamma_i-1} \\
		0 & 0 & 0 & 0 & 1  
	\end{pmatrix}.
\end{align*}
The matrix $R_{\textbf{W}}^x$ is the right eigenvector matrix for the Jacobian matrix $\frac{\p \f_x}{\p \con}$ in terms of the primitive variable (see~\ref{RE}), which is given below,
\begin{align*}
	\begin{pmatrix}
		\rho^2 & 1 & 0 & 0 & \rho^2\\
            -\sqrt{2\pe(2\gamma_e-1)\rho} & 0 & 0 & 0 & \sqrt{2\pe(2\gamma_e-1)\rho}\\
            0 & 0 & 1 & 0 & 0\\
            \pe(2\gamma_e-1)\rho & 0 & 0 & 0 & \pe(2\gamma_e-1)\rho\\
            \pii\rho & 0 & 0 & 1 & \pii\rho
	\end{pmatrix}.
\end{align*}
We want to determine the scaling matrix $T^x$ that will allow the scaled right-eigenvector matrix $\tilde{R}^x=R^x T^x$ to satisfy
\begin{equation}
	\frac{\p \con}{\p \evar} = \tilde{R^x} \tilde{R^x}^{\top},
\end{equation}
where the entropy variable $\evar$ is defined in~\eqref{eq:envar}. Now we follow~\cite{barth1999numerical} and define the matrix 
\begin{equation*}
	\mathcal{Y}^x = (R^x_{\textbf{W}})^{-1} \frac{\p \textbf{W}}{\p \evar}\left(\frac{\p \con}{\p \textbf{W}}\right)^{-\top}(R^x_{\textbf{W}})^{-\top}.
\end{equation*}
This results in,
\begin{align*}
	\mathcal{Y}^x = \begin{pmatrix}
		\frac{1}{4\rho^3\left(2\gamma_e-1\right)} & 0 & 0 & 0 & 0 \\
		0 & \frac{\rho\left(\gamma_e-1\right)}{\left(2\gamma_e-1\right)} & 0 & \frac{\pii\left(\gamma_e-1\right)}{\left(2\gamma_e-1\right)} & 0 \\
		0 & 0 & \frac{\pe}{\rho^2} & 0 & 0 \\
		0 & \frac{\pii\left(\gamma_e-1\right)}{\left(2\gamma_e-1\right)} & 0 & -\frac{\pii^2\left(\gamma_e+\gamma_i-2\gamma_e\gamma_i\right)}{\rho\left(2\gamma_e-1\right)} & 0 \\
		0 & 0 & 0 & 0 & \frac{1}{4\rho^3\left(2\gamma_e-1\right)} 
	\end{pmatrix}.
\end{align*}
Then the scaling matrix $T^x$ is the square root of $\mathcal{Y}^x,$ which is given by,
\begin{align*}
	\begin{pmatrix}
		\frac{1}{\sqrt{4\rho^3\left(2\gamma_e-1\right)}} & 0 & 0 & 0 & 0 \\
		0 & \frac{\pii\Theta_1+\rho\left(\gamma_e-1\right)}{\left(2\gamma_e-1\right)\Theta} & 0 & \frac{\pii\left(\gamma_e-1\right)}{\left(2\gamma_e-1\right)\Theta} & 0 \\
		0 & 0 & \frac{\sqrt{\pe}}{\rho} & 0 & 0 \\
		0 & \frac{\pii\left(\gamma_e-1\right)}{\left(2\gamma_e-1\right)\Theta} & 0 & \frac{\pii^2\left(-\gamma_i+\gamma_e\left(2\gamma_i-1\right)\right) + \pii\rho\Theta_1}{\rho\left(2\gamma_e-1\right)\Theta} & 0 \\
		0 & 0 & 0 & 0 & \frac{1}{\sqrt{4\rho^3\left(2\gamma_e-1\right)}} 
	\end{pmatrix},
\end{align*}
where, $\Theta_1=\sqrt{\left(\gamma_e-1\right)\left(\gamma_i-1\right)\left(2\gamma_e-1\right)}$ and $\Theta = \sqrt{\frac{2\pii\rho\Theta_1 - \pii^2\left(\gamma_e+\gamma_i-2\gamma_e\gamma_i\right) + \rho^2\left(\gamma_e-1\right)}{\rho\left(2\gamma_e-1\right)}}$. 

Similarly,  we can obtain the scaling matrix $T^y$, which turns out to be same as the matrix $T^x$.
\subsection{Entropy scaled right eigenvectors in $x$ and $y$ direction}\label{ESRE}
Finally, we provide the complete expressions for the entropy-scaled eigenvectors. In the $x$-direction, the entropy-scaled right eigenvectors matrix in terms of primitive variables is 
\begin{align*}
	\begin{pmatrix}
		\sqrt{\frac{\rho}{4\left(2\gamma_e-1\right)}} & \frac{\pii\Theta_1+\rho\left(\gamma_e-1\right)}{\left(2\gamma_e-1\right)\Theta} & 0 & \frac{\pii\left(\gamma_e-1\right)}{\left(2\gamma_e-1\right)\Theta} & \sqrt{\frac{\rho}{4\left(2\gamma_e-1\right)}} \\
		-\frac{1}{\rho}\sqrt{\frac{\pe}{2}} & 0 & 0 & 0 & \frac{1}{\rho}\sqrt{\frac{\pe}{2}} \\
		0 & 0 & \frac{\sqrt{\pe}}{\rho} & 0 & 0 \\
		\frac{\pe}{2}\sqrt{\frac{\left(2\gamma_e-1\right)}{\rho}} & 0 & 0 & 0 & \frac{\pe}{2}\sqrt{\frac{\left(2\gamma_e-1\right)}{\rho}} \\
		\frac{\pii}{2}\sqrt{\frac{1}{\left(2\gamma_e-1\right)\rho}} & \frac{\pii\left(\gamma_e-1\right)}{\left(2\gamma_e-1\right)\Theta} & 0 & \frac{\pii^2\left(-\gamma_i+\gamma_e\left(2\gamma_i-1\right)\right) + \pii\rho\Theta_1}{\rho\left(2\gamma_e-1\right)\Theta} & \frac{\pii}{2}\sqrt{\frac{1}{\left(2\gamma_e-1\right)\rho}} 
	\end{pmatrix}.
\end{align*}
Similarly, in the $y$-direction, the entropy-scaled right eigenvectors matrix in terms of primitive variables is 
\begin{align*}
	\begin{pmatrix}
		\sqrt{\frac{\rho}{4\left(2\gamma_e-1\right)}} & \frac{\pii\Theta_1+\rho\left(\gamma_e-1\right)}{\left(2\gamma_e-1\right)\Theta} & 0 & \frac{\pii\left(\gamma_e-1\right)}{\left(2\gamma_e-1\right)\Theta} & \sqrt{\frac{\rho}{4\left(2\gamma_e-1\right)}} \\
        0 & 0 & \frac{\sqrt{\pe}}{\rho} & 0 & 0 \\
		-\frac{1}{\rho}\sqrt{\frac{\pe}{2}} & 0 & 0 & 0 & \frac{1}{\rho}\sqrt{\frac{\pe}{2}} \\
		\frac{\pe}{2}\sqrt{\frac{\left(2\gamma_e-1\right)}{\rho}} & 0 & 0 & 0 & \frac{\pe}{2}\sqrt{\frac{\left(2\gamma_e-1\right)}{\rho}} \\
		\frac{\pii}{2}\sqrt{\frac{1}{\left(2\gamma_e-1\right)\rho}} & \frac{\pii\left(\gamma_e-1\right)}{\left(2\gamma_e-1\right)\Theta} & 0 & \frac{\pii^2\left(-\gamma_i+\gamma_e\left(2\gamma_i-1\right)\right) + \pii\rho\Theta_1}{\rho\left(2\gamma_e-1\right)\Theta} & \frac{\pii}{2}\sqrt{\frac{1}{\left(2\gamma_e-1\right)\rho}} 
	\end{pmatrix}.
\end{align*}
Where, $\Theta_1=\sqrt{\left(\gamma_e-1\right)\left(\gamma_i-1\right)\left(2\gamma_e-1\right)}$ and $\Theta = \sqrt{\frac{2\pii\rho\Theta_1 - \pii^2\left(\gamma_e+\gamma_i-2\gamma_e\gamma_i\right) + \rho^2\left(\gamma_e-1\right)}{\rho\left(2\gamma_e-1\right)}}$. 

\end{document}